\documentclass[11pt
]{article}

\usepackage[margin=1in]{geometry}
\usepackage{fancyhdr}
\usepackage{amsmath}
\usepackage{amssymb}
\usepackage{amsthm}
\usepackage{amsfonts}
\usepackage{amsopn}
\usepackage{xcolor}
\usepackage{bbm}
\usepackage{MnSymbol}
\usepackage{tikz-cd}
\usepackage{mathrsfs}
\usepackage{soul} %added to get the \st command to strikethrough text
\usepackage[all]{xy}
\usepackage{hyperref}
\usepackage{multirow}
\usepackage{multicol}

\usepackage{pdflscape}

\newcommand{\mtrx}{\boldsymbol}

\newcommand{\ZZ}{\mathbb{Z}}

\DeclareMathOperator{\Char}{char}
\DeclareMathOperator{\Pf}{Pf}
\newcommand{\maxI}{\mathfrak{m}}

\DeclareMathOperator{\soc}{soc}
\DeclareMathOperator{\reg}{reg}
\DeclareMathOperator{\sgn}{sgn}
\DeclareMathOperator{\projdim}{pd}

\newtheorem{proposition}{Proposition}[section]
\newtheorem*{proposition*}{Proposition}
\newtheorem{theorem}[proposition]{Theorem}
\newtheorem*{theorem*}{Theorem}
\newtheorem{corollary}[proposition]{Corollary}
\newtheorem{lemma}[proposition]{Lemma}
\newtheorem*{lemma*}{Lemma}

\newtheorem{theoremAlpha}{Theorem}
\theoremstyle{definition}
\newtheorem{definition}[proposition]{Definition}
\newtheorem*{definition*}{Definition}

\newtheorem{notation}[proposition]{Notation}

\theoremstyle{remark}

\newtheorem*{exercise*}{Exercise}
\newtheorem{example}[proposition]{Example}

\newtheorem{remark}[proposition]{Remark}
\newtheorem*{remark*}{Remark}

\title{Determining the Betti numbers of $R/(x^{p^e},y^{p^e},z^{p^e})$ for most even degree hypersurfaces in odd characteristic}
\author{Heath Camphire}
\date{October 13, 2023}

\begin{document}
\maketitle

\begin{abstract}
Let $k$ be a field of odd characteristic $p$. Fix an even number $d<p+1$ and a power $q\geq d+3$ of $p$. For most choices of degree $d$ standard graded hypersurfaces $R=k[x,y,z]/(f)$ with homogeneous maximal ideal $\maxI$, we can determine the graded Betti numbers of $R/\maxI^{[q]}$. In fact, given two fixed powers $q_0,q_1\geq d+3$, for most choices of $R$ the graded Betti numbers in high homological degree of $R/\maxI^{[q_0]}$ and $R/\maxI^{[q_1]}$ are the same up to a constant shift. This thesis shows this fact by combining our results with the work of Miller, Rahmati, and R.G on link-$q$-compressed polynomials in \cite{RGpaper}. We show that link-$q$-compressed polynomials are indeed fairly common in many polynomial rings.
\end{abstract}

\section{Introduction}

Kustin and Ulrich wrote in \cite{KU09} about a peculiar phenomenon they observed: given a Noetherian graded algebra $(R,\maxI)$ over a field $k$ of characteristic $p>0$, for certain choices of $\maxI$-primary ideal $J$ of $R$, the tail of the resolution of $R/J^{[p^e]}$ appeared to be constant as a function of $e$ (up to a graded shift). This is unexpected as the first several syzygy modules often have generators of vastly different degrees for different $e$, however it is possible to find examples of $J$ such that the $\dim R$-th syzygy module of $R/J^{[p^e]}$ is independent of $e$.

There have been a few results so far related to this phenomenon.
Kustin and Vraciu in \cite{KV07} found conditions that would show that $R/J$ and $R/J^{[p^e]}$ both have finite projective dimension, and so the tails of their resolutions are the same because they're both zero.
Also in \cite{KU09}, Kustin and Ulrich found conditions that the tails of the resolutions of $R/J$ and $R/J^{[p^e]}$ are isomorphic (up to a graded shift) as graded modules, though not necessarily as complexes with differential.
Kustin, Rahmati, and Vraciu explicitly determined the resolutions of $R/J^{[p^e]}$ in \cite{KRV12}, where $J=\maxI^{[N]}$ for some $N>0$, 
$\maxI$ is the homogeneous maximal ideal, 
and $R$ is a diagonal hypersurface in two or three variables. From this they determined that the tail of the resolution of $R/J^{[p^e]}$ is a periodic (not necessarily constant) function of $e$ up to a graded shift. This is true even if the tail of the resolution is treated as a complex with differential, not just as a graded module.

Kustin, R.G., and Vraciu studied diagonal hypersurfaces $R$ in four indeterminates. In \cite{KRGV22} they discovered that the Betti numbers in fixed homological degree of $R/\maxI^{[N]}$ (where $\maxI$ is the homogeneous maximal ideal) increase rather than stay constant as $N$ increases. In fact, they determined in \cite{KRGV23} that the tail of the resolution of $R/\maxI^{[N]}$ is the direct sum of resolutions of $R/\maxI^{[N']}$ for certain $N'\leq N$. Compared to three indeterminates, the four indeterminate case has much more complicated resolutions and Betti numbers.

Our results show that for $R=k[x,y,z]/\left(\left(xy-z^2\right)^D\right)$, the $R$-free resolution of $R/\maxI^{[p^e]}$ has its tail independent of $e$ if $p>2D-1$ and $p$ is odd; see Corollary \ref{cor:FDResRailIndep}. We prove this with the method used in \cite{KRV12} while proving our main results, which we then combine with the work by Miller, Rahmati, and R.G. in \cite{RGpaper}.

Miller, Rahmati, and R.G. developed a property of algebras called \textit{$\mathfrak{c}$-compressed} (see Definition \ref{def:compressed}) in \cite{RGpaper} as a stronger version of the relatively compressed property, which originally appeared in \cite{IK99,MMRN05} and only required the Hilbert function to reach a maximal value among graded Artinian algebras, but not necessarily their theoretical maximum.
A homogeneous polynomial $f \in k[x,y,z]$ is \textit{link-$q$-compressed} if $\maxI^{[q]}:f$ is $\maxI^{[q]}$-compressed where $\maxI=(x,y,z)$ is the homogeneous maximal ideal (see Definition \ref{def:lqc}), and  Miller, Rahmati, and R.G. showed that if $f \in k[x,y,z]$ is link-$q$-compressed for two powers $q=q_0,q_1$ of $p$, then the graded Betti numbers for the resolutions of $R/\maxI^{[q_0]}$ and $R/\maxI^{[q_1]}$ eventually agree up to a constant shift (see the exact results in Theorem \ref{thm:RGlqcBetti}). Their work also gives formulas for algebraic invariants like the Hilbert-Kunz function of $R$ and the Castelnuovo-Mumford regularity of $R/\maxI^{[q]}$ when $f$ is link-$q$-compressed. In \cite{RGpaper} it is shown that link-$p^e$-compressed is a \textit{Zariski-open} condition on the coefficients of $f$ for any given values of $e$ and $d$. This means that the condition holds for general choices of $f$ only if it holds for at least one choice of $f$. We discuss this further in Remark \ref{rmk:lqcInGeneral}.

The main result of this thesis is Theorem \ref{thm:FDlqc}, which is the following:
\begin{theoremAlpha}
\label{thmA:FDlqc}
Let $k$ be a field of odd characteristic $p$. Fix a number $D<\frac{p+1}{2}$ and a power $q>1$ of $p$. The polynomial $\left(xy-z^2\right)^D \in k[x,y,z]$ is link-$q$-compressed.
\end{theoremAlpha}
The set of link-$q$-compressed polynomials is Zariski open as shown in \cite{RGpaper}, and our results show that it is also nonempty. This means that a general choice of polynomial $f$ is link-$q$-compressed, and thus the conclusions in theorems about link-$q$-compressed polynomials in \cite{RGpaper} apply to $f$. We summarize these results as follows, which is Theorem \ref{thm:generalBettiAndMore}:
\begin{theoremAlpha}
\label{thmA:generalBettiAndMore}
Let $k$ be a field of odd characteristic $p$. Fix an even number $d<p+1$ and a power $q\geq d+3$ of $p$. For a general choice of degree $d$ standard graded hypersurface $k$-algebra $R$ over three indeterminates with homogeneous maximal ideal $\maxI$, the following hold:
\begin{itemize}
    \item The minimal graded $R$-free resolution of $R/\maxI^{[q]}$ has the following eventually 2-periodic form \[\cdots \xrightarrow{\mtrx{\partial}_4} R^{2d} \xrightarrow{\mtrx{\partial}_3} R^{2d} \xrightarrow{\mtrx{\partial}_4} R^{2d} \xrightarrow{\mtrx{\partial}_3} R^{2d} \xrightarrow{\mtrx{\partial}_2} R^{3} \xrightarrow{\mtrx{\partial}_1} R \to 0\] whose differentials are maps of pure graded degrees \[\deg(\mtrx{\partial}_1) = q, \; \deg(\mtrx{\partial}_2) = \frac{1}{2}(q+d-1), \; \deg(\mtrx{\partial}_3) = 1, \; \deg(\mtrx{\partial}_4) = d-1\/.\]
    \item The Castelnuovo-Mumford regularity is given by $\reg(R/\maxI^{[q]}) = \frac{3}{2} q + \frac{1}{2} d - \frac{5}{2}$.
    \item The Hilbert-Kunz function of $R$ at $q$ is $HK_R(q) = \frac{3}{4} dq^2 - \frac{1}{12} (d^3 - d)$.
\end{itemize}
\end{theoremAlpha}

This thesis has the following structure:
Section \ref{sec:Background} starts with Subsection \ref{subsec:Pfafs}, which covers matrix notation and results on Pfaffians, an invariant of skew-symmetric matrices analogous to determinants. The rest of Section \ref{sec:Background} gives prior results related to the link-$q$-compressed condition, beginning with its original definition in Definitions \ref{def:compressed} and \ref{def:lqc}.

Sections \ref{sec:NumThry} and \ref{sec:DetCalcs} detail intermediate calculations used in our main results. Specifically, Section \ref{sec:NumThry} covers relevant number theory results, while Section \ref{sec:DetCalcs} lists determinant calculations related to the matrix $\mtrx{M}$ in Notation \ref{not:MandLn}.

Section \ref{sec:MainResults} includes our main results. We prove the first major result Theorem \ref{thm:FDlqc}, which we described above as Theorem \ref{thmA:FDlqc}. We combine this result with the results in \cite{RGpaper} using their theory of link-$q$-compressed polynomials to prove Theorem \ref{thm:generalBettiAndMore}. At the end of the section we list other examples and theorems related to the link-$q$-compressed condition.

\section{Background}
\label{sec:Background}

\subsection{Pfaffians}
\label{subsec:Pfafs}

Here we give the definition of a Pfaffian, as well as list several properties, formulas, and notations that we use throughout this thesis.
More information on Pfaffians, including proofs for the properties to come, can be found in sources such as \cite{God93} and \cite{Ham01}.

Much of our work involves calculating determinants and Pfaffians of matrices. We introduce notation used throughout our paper when performing those calculations:

\begin{notation}
\begin{itemize}
    \item 
For a matrix $\mtrx{T}$ of size $M \times N$, the entry of $\mtrx{T}$ in the $i$th row and $j$th column is denoted $\mtrx{T}_{i,j}$.
    \item 
We also let $\mtrx{T}_{(\hat{\imath}_1\cdots\hat{\imath}_m),(\hat{\jmath}_1\cdots\hat{\jmath}_n)}$ denote the matrix obtained by removing rows $1 \leq i_1,\ldots,i_m \leq M$ and columns $1 \leq j_1,\ldots,j_n \leq N$ from $\mtrx{T}$.
    \item 
If we don't remove any rows (resp. columns) we use $(-)$ in place of $(\hat{\imath}_1\cdots\hat{\imath}_m)$ (resp. $(\hat{\jmath}_1\cdots\hat{\jmath}_n)$), which looks like $\mtrx{T}_{(-),(\hat{\jmath}_1\cdots\hat{\jmath}_n)}$ (resp. $\mtrx{T}_{(\hat{\imath}_1\cdots\hat{\imath}_m),(-)}$).
    \item 
To make notation more compact, we use $\mtrx{T}_{\hat{\hat{\imath}}_1\cdots\hat{\hat{\imath}}_m}$ as a shorthand for $\mtrx{T}_{(\hat{\imath}_1\cdots\hat{\imath}_m),(\hat{\imath}_1\cdots\hat{\imath}_m)}$.
    \item 
Here we also note that all matrices in this thesis are denoted in boldface (which looks like $\mtrx{T}$), and column vectors have a vector arrow above them (which looks like $\vec{T}$), with row vectors written a vector arrow and transpose symbol (which looks like $\vec{T}^\top$).
    \item 
We also use $\vec{T}_{\hat{\imath}_1\cdots\hat{\imath}_m}$ to denote $\vec{T}_{(\hat{\imath}_1\cdots\hat{\imath}_m),(-)}$ if $\vec{T}$ is a single column, and $\vec{T}_{\hat{\jmath}_1\cdots\hat{\jmath}_n}^\top$ to denote $(\vec{T}^\top)_{(-),(\hat{\jmath}_1\cdots\hat{\jmath}_n)}$ if $\vec{T}^\top$ is a single row.
\end{itemize}
\end{notation}

\begin{definition}[{\cite[Corollary 2.2]{Ham01}}]
\label{def:Pfaffs}
If a skew-symmetric matrix $\mtrx{A}$ is size $s \times s$ where $s>0$ is even, then its \textit{Pfaffian} can be calculated by using an analog of the cofactor expansion formula: \[
\Pf\mtrx{A}
 = 
\sum_{i=1}^s (-1)^{i+j+H(j-i)} \mtrx{A}_{i,j} \Pf\left( \mtrx{A}_{\hat{\hat{\imath}}\hat{\hat{\jmath}}} \right)
\/,\] where $j$ is any fixed index from $1$ to $s$ and $H(x) = \begin{cases}1,&x\geq0\\0,&x<0\end{cases}$ is the Heaviside step function.
If $s=0$, we define \[
\Pf\mtrx{A} = 1
\/.\] If $s$ is odd, then we set \[
\Pf\mtrx{A} = 0
\/.\]
\end{definition}
\begin{remark}
The summation formula doesn't change if we fix $i$ and sum over $j$ instead (see \cite[Corollary 4.3]{Ham01}). This formula is especially useful when when most entries in the $i$-th row or $j$-th column of $\mtrx{A}$ are zero.
\end{remark}
\begin{remark}
\label{rmk:PfaffSqDet}
Pfaffians have a close relation to determinants: if $\mtrx{A}$ is skew-symmetric, then $\det\mtrx{A} = (\Pf\mtrx{A})^2$ (see {\cite[Theorem 2.3]{God93}}).
It should be no surprise, given our definition is an analog of the cofactor expansion formula for determinants.
\end{remark}

\begin{remark}
\label{rmk:PfaffBlockSquare}
If $\mtrx{M}$ is a $n \times n$ matrix, then \[
\Pf \begin{bmatrix} \mtrx{0} & \mtrx{M} \\ -\mtrx{M}^\top & \mtrx{0} \end{bmatrix} = (-1)^{n(n-1)/2} \det\mtrx{M}
\/,
\] where $\mtrx{M}^\top$ is the transpose of $\mtrx{M}$ (see \cite[Lemma 3.1]{Ham01}).
\end{remark}
\begin{lemma}
\label{lem:PfaffBlockNonSquare}
If $\mtrx{M}$ is a non-square matrix, then $\Pf\begin{bmatrix}\mtrx{0}&\mtrx{M}\\-\mtrx{M}^\top&\mtrx{0}\end{bmatrix}=0$.
\end{lemma}
\begin{proof}
Fix $\ell>0$. To prove that $\Pf\begin{bmatrix}\mtrx{0}&\mtrx{M}\\-\mtrx{M}^\top&\mtrx{0}\end{bmatrix}=0$, we assume first that $\mtrx{M}$ has size $(n+\ell)\times n$ and then that it has size $n\times(n+\ell)$, and use induction on $n$ both times. In either of these cases, $\Pf\begin{bmatrix}\mtrx{0}&\mtrx{M}\\-\mtrx{M}^\top&\mtrx{0}\end{bmatrix}$ has size $s\times s$ where $s = (n+\ell)+n = 2n+\ell$.

First we assume that $\mtrx{M}$ has $\ell$ more rows than columns (so it has size $(n+\ell)\times n$).

If $\mtrx{M}$ is a $\ell\times0$ matrix, then $\Pf\begin{bmatrix}\mtrx{0}&\mtrx{M}\\-\mtrx{M}^\top&\mtrx{0}\end{bmatrix} = \Pf\mtrx{0} = 0$.

Let $n\geq0$ and assume that $\Pf\begin{bmatrix}\mtrx{0}&\mtrx{M}\\-\mtrx{M}^\top&\mtrx{0}\end{bmatrix} = 0$ for any matrix $\mtrx{M}$ of size $(n+\ell)\times n$. Let $\mtrx{M}$ be a $((n+1)+\ell)\times(n+1)$ matrix.
Fix $\beta=((n+1)+\ell)+(n+1)=2(n+1)+\ell$ and $j=n+1=\beta-((n+1)+\ell)$.
If $1\leq\alpha\leq (n+1)+\ell$, let $i=\alpha$, then we have  \[
\Pf\left( \begin{bmatrix}\mtrx{0}&\mtrx{M}\\-\mtrx{M}^\top&\mtrx{0}\end{bmatrix}_{\hat{\hat{\alpha}}\hat{\hat{\beta}}} \right)
 = 
\Pf\begin{bmatrix}\mtrx{0}_{(\hat{\imath}),(\hat{\imath})}&\mtrx{M}_{(\hat{\imath}),(\hat{\jmath})}\\\left(-\mtrx{M}^\top\right)_{(\hat{\jmath}),(\hat{\imath})}&\mtrx{0}_{(\hat{\jmath}),(\hat{\jmath})}\end{bmatrix}
 = 
\Pf\begin{bmatrix}\mtrx{0}&\mtrx{M}_{(\hat{\imath}),(\hat{\jmath})}\\-\left(\mtrx{M}_{(\hat{\imath}),(\hat{\jmath})}\right)^\top&\mtrx{0}\end{bmatrix}
 = 
0
\] because $\mtrx{M}_{(\hat{\imath}),(\hat{\jmath})}$ has size $(n+\ell)\times n$.
If instead $(n+1)+\ell+1\leq\alpha\leq 2(n+1)+\ell$, let $i=\alpha-((n+1)+\ell)$, then $(\alpha,\beta)$ are coordinates for the bottom right $\mtrx{0}$ block of $\begin{bmatrix}\mtrx{0}&\mtrx{M}\\-\mtrx{M}^\top&\mtrx{0}\end{bmatrix}$, and so \[
\begin{bmatrix}\mtrx{0}&\mtrx{M}\\-\mtrx{M}^\top&\mtrx{0}\end{bmatrix}_{\alpha,\beta} = \mtrx{0}_{i,j} = 0
\/.\]
In either case, $\begin{bmatrix}\mtrx{0}&\mtrx{M}\\-\mtrx{M}^\top&\mtrx{0}\end{bmatrix}_{\alpha,\beta} \Pf\left( \begin{bmatrix}\mtrx{0}&\mtrx{M}\\-\mtrx{M}^\top&\mtrx{0}\end{bmatrix}_{\hat{\hat{\alpha}}\hat{\hat{\beta}}} \right) = 0$, and thus by Definition \ref{def:Pfaffs},
\begin{align*}
\Pf\begin{bmatrix}\mtrx{0}&\mtrx{M}\\-\mtrx{M}^\top&\mtrx{0}\end{bmatrix}
=&
\sum_{\alpha=1}^{2(n+1)+\ell} (-1)^{\alpha+\beta+H(\beta-\alpha)} \begin{bmatrix}\mtrx{0}&\mtrx{M}\\-\mtrx{M}^\top&\mtrx{0}\end{bmatrix}_{\alpha,\beta} \Pf\left( \begin{bmatrix}\mtrx{0}&\mtrx{M}\\-\mtrx{M}^\top&\mtrx{0}\end{bmatrix}_{\hat{\hat{\alpha}}\hat{\hat{\beta}}} \right)
\\=&
\sum_{\alpha=1}^{2(n+1)+\ell} (-1)^{\alpha+\beta+H(\beta-\alpha)} (0)
\\=&
0
\/.\end{align*}
Therefore, by induction we have that $\Pf\begin{bmatrix}\mtrx{0}&\mtrx{M}\\-\mtrx{M}^\top&\mtrx{0}\end{bmatrix} = 0$ for any $\mtrx{M}$ of size $(n+\ell) \times n$ with $\ell>0$.

The argument for $n \times(n+\ell)$ matrices is similar.
\end{proof}
\begin{definition}[{\cite[Equation 1.16]{KRV12}}]
\label{def:PfaffEll}
If a skew-symmetric matrix $\mtrx{A}$ is size $s \times s$ and $s$ is odd, we take Pfaffians of submatrices of $\mtrx{A}$ in order to still gain some information about $\mtrx{A}$. For any index $\ell$ from $1$ to $s$, we define \[
\Pf_\ell\mtrx{A} = (-1)^{\ell+1} \Pf\left(\mtrx{A}_{\hat{\hat{\ell}}}\right)
\/.\]
\end{definition}
\begin{definition}
\label{def:detAdj}
We denote the \textit{classical adjoint} of a matrix $\mtrx{M}$ as $\overline{\mtrx{M}}$, meaning $\mtrx{M} \overline{\mtrx{M}} = \overline{\mtrx{M}} \mtrx{M} = (\det \mtrx{M}) \mtrx{I}$. This matrix is given by the formula \[
\overline{\mtrx{M}}_{i,j}
 = 
(-1)^{j+i} \det\left( \mtrx{M}_{(\hat{\jmath}),(\hat{\imath})} \right)
\/.
\]
\end{definition}
\begin{definition}[{\cite[Definition 1.17 and Observation 1.18]{KRV12}}]
\label{def:PfaffAdj}
The Pfaffian version of the classical adjoint is denoted $\mtrx{A}^\vee$ for a matrix $\mtrx{A}$, where $\mtrx{A} \mtrx{A}^\vee = \mtrx{A}^\vee \mtrx{A} = (\Pf\mtrx{A}) \mtrx{I}$. This matrix is given by the formula \[
\left(\mtrx{A}^\vee\right)_{i,j}
 = 
(-1)^{j+i+H(i-j)} \Pf\left( \mtrx{A}_{\hat{\hat{\jmath}}\hat{\hat{\imath}}} \right)
\/.\]
\end{definition}
\begin{lemma}
\label{lem:PfaffBlockAdj}
If $\mtrx{M}$ is a $n \times n$ matrix, then \[
\begin{bmatrix} \mtrx{0} & \mtrx{M} \\ -\mtrx{M}^\top & \mtrx{0} \end{bmatrix}^\vee
 = 
(-1)^{n(n-1)/2} \begin{bmatrix} \mtrx{0} & -\overline{\mtrx{M}}^\top \\ \overline{\mtrx{M}} & \mtrx{0} \end{bmatrix}
\/.\]
\end{lemma}
\begin{proof}
The matrix $\begin{bmatrix} \mtrx{0} & \mtrx{M} \\ -\mtrx{M}^\top & \mtrx{0} \end{bmatrix}^\vee$ is a $2\times2$ block matrix consisting of $n\times n$ blocks.
Let $1\leq\alpha,\beta\leq2n$.

Assume $1\leq\alpha,\beta\leq n$. Let $i=\alpha$ and $j=\beta$.
The $(\alpha,\beta)$ coordinate of $\begin{bmatrix} \mtrx{0} & \mtrx{M} \\ -\mtrx{M}^\top & \mtrx{0} \end{bmatrix}^\vee$ is the $(i,j)$ coordinate of the top left block of $\begin{bmatrix} \mtrx{0} & \mtrx{M} \\ -\mtrx{M}^\top & \mtrx{0} \end{bmatrix}^\vee$.
By Definition \ref{def:PfaffAdj}, we have
\begin{align*}
\left(\begin{bmatrix} \mtrx{0} & \mtrx{M} \\ -\mtrx{M}^\top & \mtrx{0} \end{bmatrix}^\vee\right)_{\alpha,\beta}
=&
(-1)^{\beta+\alpha+H(\alpha-\beta)}
\Pf\left( \begin{bmatrix} \mtrx{0} & \mtrx{M} \\ -\mtrx{M}^\top & \mtrx{0} \end{bmatrix}_{\hat{\hat{\beta}}\hat{\hat{\alpha}}} \right)
\\=&
(-1)^{j+i+H(i-j)}
\Pf \begin{bmatrix} \mtrx{0}_{(\hat{\jmath}\hat{\imath}),(\hat{\jmath}\hat{\imath})} & \mtrx{M}_{(\hat{\jmath}\hat{\imath}),(-)} \\ \left(-\mtrx{M}^\top\right)_{(-),(\hat{\jmath}\hat{\imath})} & \mtrx{0} \end{bmatrix}
\\=&
(-1)^{j+i+H(i-j)}
\Pf \begin{bmatrix} \mtrx{0} & \mtrx{M}_{(\hat{\jmath}\hat{\imath}),(-)} \\ -\left(\mtrx{M}_{(\hat{\jmath}\hat{\imath}),(-)}\right)^\top & \mtrx{0} \end{bmatrix}
\\=&\tag{1}
(-1)^{j+i+H(i-j)}
0
 = 
0
\/,\end{align*}
where (1) comes from Lemma \ref{lem:PfaffBlockNonSquare} because $\mtrx{M}_{(\hat{\jmath}\hat{\imath}),(-)}$ is size $(n-2)\times n$.
This means that the top left block of $\begin{bmatrix} \mtrx{0} & \mtrx{M} \\ -\mtrx{M}^\top & \mtrx{0} \end{bmatrix}^\vee$ is $\mtrx{0}$.

Assume $1\leq\alpha\leq n<\beta\leq2n$. Let $i=\alpha$ and $j=\beta-n$.
The $(\alpha,\beta)$ coordinate of $\begin{bmatrix} \mtrx{0} & \mtrx{M} \\ -\mtrx{M}^\top & \mtrx{0} \end{bmatrix}^\vee$ is the $(i,j)$ coordinate of the top right block of $\begin{bmatrix} \mtrx{0} & \mtrx{M} \\ -\mtrx{M}^\top & \mtrx{0} \end{bmatrix}^\vee$.
By Definition \ref{def:PfaffAdj}, we have
\begin{align*}
\left(\begin{bmatrix} \mtrx{0} & \mtrx{M} \\ -\mtrx{M}^\top & \mtrx{0} \end{bmatrix}^\vee\right)_{\alpha,\beta}
=&
(-1)^{\beta+\alpha+H(\alpha-\beta)}
\Pf\left( \begin{bmatrix} \mtrx{0} & \mtrx{M} \\ -\mtrx{M}^\top & \mtrx{0} \end{bmatrix}_{\hat{\hat{\beta}}\hat{\hat{\alpha}}} \right)
\\=&
(-1)^{(n+j)+i}
\Pf\left( \begin{bmatrix} \mtrx{0}_{(\hat{\imath}),(\hat{\imath})} & \mtrx{M}_{(\hat{\imath}),(\hat{\jmath})} \\ -\left(\mtrx{M}_{(\hat{\imath}),(\hat{\jmath})}\right)^\top & \mtrx{0}_{(\hat{\jmath}),(\hat{\jmath})} \end{bmatrix} \right)
\\=&\tag{2}
(-1)^n
(-1)^{i+j}
\left(
(-1)^{(n-1)((n-1)-1)/2}
\det\left( \mtrx{M}_{(\hat{\imath}),(\hat{\jmath})} \right)
\right)
\\=&\tag{3}
(-1)^n
(-1)^{(n-1)(n-2)/2}
\overline{\mtrx{M}}_{j,i}
\\=&
-(-1)^{n-1}
(-1)^{(n-1)n/2-(n-1)}
\left(\overline{\mtrx{M}}^\top\right)_{i,j}
\\=&
-
(-1)^{n(n-1)/2}
\left(\overline{\mtrx{M}}^\top\right)_{i,j}
\/,\end{align*}
where (2) comes from Remark \ref{rmk:PfaffBlockSquare} because $\mtrx{M}_{(\hat{\imath}),(\hat{\jmath})}$ is size $(n-1)\times(n-1)$ and (3) comes from Definition \ref{def:PfaffAdj}.
This means that the top right block of $\begin{bmatrix} \mtrx{0} & \mtrx{M} \\ -\mtrx{M}^\top & \mtrx{0} \end{bmatrix}^\vee$ is $-
(-1)^{n(n-1)/2}
\overline{\mtrx{M}}^\top$.

Assume $1\leq\beta\leq n<\alpha\leq2n$. Let $i=\alpha-n$ and $j=\beta$.
The $(\alpha,\beta)$ coordinate of $\begin{bmatrix} \mtrx{0} & \mtrx{M} \\ -\mtrx{M}^\top & \mtrx{0} \end{bmatrix}^\vee$ is the $(i,j)$ coordinate of the bottom left block of $\begin{bmatrix} \mtrx{0} & \mtrx{M} \\ -\mtrx{M}^\top & \mtrx{0} \end{bmatrix}^\vee$.
By Definition \ref{def:PfaffAdj}, we have
\begin{align*}
\left(\begin{bmatrix} \mtrx{0} & \mtrx{M} \\ -\mtrx{M}^\top & \mtrx{0} \end{bmatrix}^\vee\right)_{\alpha,\beta}
=&
(-1)^{\beta+\alpha+H(\alpha-\beta)}
\Pf\left( \begin{bmatrix} \mtrx{0} & \mtrx{M} \\ -\mtrx{M}^\top & \mtrx{0} \end{bmatrix}_{\hat{\hat{\beta}}\hat{\hat{\alpha}}} \right)
\\=&
(-1)^{j+(n+i)+1}
\Pf\left( \begin{bmatrix} \mtrx{0}_{(\hat{\jmath}),(\hat{\jmath})} & \mtrx{M}_{(\hat{\jmath}),(\hat{\imath})} \\ -\left(\mtrx{M}_{(\hat{\jmath}),(\hat{\imath})}\right)^\top & \mtrx{0}_{(\hat{\imath}),(\hat{\imath})} \end{bmatrix} \right)
\\=&\tag{4}
(-1)^{n+1}
(-1)^{j+i}
\left(
(-1)^{(n-1)((n-1)-1)/2}
\det\left(\mtrx{M}_{(\hat{\jmath}),(\hat{\imath})}\right)
\right)
\\=&\tag{5}
(-1)^{n-1}
(-1)^{(n-1)(n-2)/2}
\overline{\mtrx{M}}_{i,j}
\\=&
(-1)^{n-1}
(-1)^{(n-1)n/2-(n-1)}
\overline{\mtrx{M}}_{i,j}
\\=&
(-1)^{n(n-1)/2}
\overline{\mtrx{M}}_{i,j}
\/,\end{align*}
where (4) comes from Remark \ref{rmk:PfaffBlockSquare} because $\mtrx{M}_{(\hat{\imath}),(\hat{\jmath})}$ is size $(n-1)\times(n-1)$ and (5) comes from Definition \ref{def:PfaffAdj}.
This means that the bottom left block of $\begin{bmatrix} \mtrx{0} & \mtrx{M} \\ -\mtrx{M}^\top & \mtrx{0} \end{bmatrix}^\vee$ is $
(-1)^{n(n-1)/2}
\overline{\mtrx{M}}$.

Assume $n+1\leq\alpha,\beta\leq2n$. Let $i=\alpha-n$ and $j=\beta-n$.
The $(\alpha,\beta)$ coordinate of $\begin{bmatrix} \mtrx{0} & \mtrx{M} \\ -\mtrx{M}^\top & \mtrx{0} \end{bmatrix}^\vee$ is the $(i,j)$ coordinate of the bottom right block of $\begin{bmatrix} \mtrx{0} & \mtrx{M} \\ -\mtrx{M}^\top & \mtrx{0} \end{bmatrix}^\vee$.
By Definition \ref{def:PfaffAdj}, we have
\begin{align*}
\left(\begin{bmatrix} \mtrx{0} & \mtrx{M} \\ -\mtrx{M}^\top & \mtrx{0} \end{bmatrix}^\vee\right)_{\alpha,\beta}
=&
(-1)^{\beta+\alpha+H(\alpha-\beta)}
\Pf\left( \begin{bmatrix} \mtrx{0} & \mtrx{M} \\ -\mtrx{M}^\top & \mtrx{0} \end{bmatrix}_{\hat{\hat{\beta}}\hat{\hat{\alpha}}} \right)
\\=&
(-1)^{j+i+H(i-j)}
\Pf \begin{bmatrix} \mtrx{0} & \mtrx{M}_{(-),(\hat{I}\hat{J})} \\ -\mtrx{M}_{(-),(\hat{I}\hat{J})}^\top & \mtrx{0} \end{bmatrix}
\\=&\tag{6}
(-1)^{j+i+H(i-j)} 0
 = 
0
\/,\end{align*}
where (6) also comes from Lemma \ref{lem:PfaffBlockNonSquare} because $\mtrx{M}_{(-),(\hat{\imath}\hat{\jmath})}$ is size $n\times(n-2)$.
This means that the bottom right block of $\begin{bmatrix} \mtrx{0} & \mtrx{M} \\ -\mtrx{M}^\top & \mtrx{0} \end{bmatrix}^\vee$ is $\mtrx{0}$.
\end{proof}

The following is used in our final results.
\begin{lemma}[{\cite[Observation 1.22]{KRV12}}]
\label{lem:partial2lastPfs}
Let $\mtrx{\partial}_2$ be an $(m+3)\times(m+3)$ skew-symmetric matrix with $m$ even. If $\mtrx{\partial}_2$ is partitioned into submatrices \[
\mtrx{\partial}_2
 = 
\begin{bmatrix} \mtrx{\varphi} & \mtrx{\psi} \\ -\mtrx{\psi}^\top & \mtrx{\Phi} \end{bmatrix}
\/,\] where $\mtrx{\varphi}$ is an $m \times m$ skew-symmetric matrix, $\mtrx{\Phi}$ is a $3 \times 3$ skew-symmetric matrix, and $\mtrx{\psi}$ is a $m \times 3$ matrix, then for each index $\ell$, with $1 \leq \ell \leq 3$, \[
\Pf_{m+\ell}(\mtrx{\partial}_2)
 = 
\Pf_\ell\left(
\mtrx{\psi}^\top \mtrx{\varphi}^\vee \mtrx{\psi}
+
\Pf(\mtrx{\varphi})
\mtrx{\Phi}
\right)
\/.\]
\end{lemma}
\begin{remark}
This version of \cite[Observation 1.22]{KRV12} replaces $\mtrx{\psi}$ in the original version with $\mtrx{\psi}^\top$ so that it matches later notation in this thesis.
\end{remark}
The last note we make here regards the degree of determinants.
\begin{remark}
\label{rmk:DetDeg}
If a $n \times n$ matrix $\mtrx{M}$ has entries in a graded ring $R$ that are all in the same graded component $R_d$ (meaning $\mtrx{M}$ is a map of pure graded degree $d$), then $\det\mtrx{M} \in R_{nd}$.

The Leibniz formula for the determinant of $\mtrx{M}$ is \[
\det\mtrx{M} = \sum_{\sigma\in S_n} \left(\sgn\sigma \prod_{h=1}^n \mtrx{M}_{h,\sigma(h)}\right)
\/,\] where $S_n$ is the $n$th symmetric group and $\sgn$ is the signature function. Since $\mtrx{M}_{h,\sigma(h)} \in R_d$ for each $h$ and $\sigma$, each term $\sgn\sigma \prod_{h=1}^n \mtrx{M}_{h,\sigma(h)}$ in this sum belongs to $R_{\sum_{h=1}^n d} = R_{nd}$, and so $\det\mtrx{M} \in R_{nd}$.
\end{remark}

\subsection{The link-$q$-compressed condition}
\label{subsec:LinkQComp}

This thesis revolves around homogeneous polynomials that are link-$q$-compressed. Here we give the technical definition of link-$q$-compressed. After this, we introduce an equivalent condition to use as an alternate definition.

\begin{notation}
\label{not:HilbertFunc}
Let $H_i$ denote the Hilbert function, which is defined by $H_i(B) = \dim_k B_i$ of a graded algebra $B$.
\end{notation}
\begin{definition}[{\cite[Definition 2.3]{RGpaper}}]
\label{def:compressed}
Let $\mathfrak{c} \subseteq P$ be a homogeneous complete intersection ideal such that $P/\mathfrak{c}$ is Artinian. Let $J \subseteq P$ be a homogeneous Gorenstein ideal containing $\mathfrak{c}$. We say that the algebra $A = P/J$, or equivalently the ideal $J$, is \textit{$\mathfrak{c}$-compressed} if, for every $i$, $H_i(P/J)$ takes on the maximum possible value, i.e., \[
H_i(P/J) = \min\{H_i(P/\mathfrak{c}),H_{s-i}(P/\mathfrak{c})\}
\] where $s$ is the degree of the socle of $A$.
\end{definition}
\begin{notation}
Let $P=k[x_1,\ldots,x_n]$ be a standard graded polynomial ring over a characteristic $p>0$ field $k$. Define the homogeneous maximal ideal $\maxI = (x_1,\ldots,x_n)$. If $q = p^e$ is a power of $p$, define the $q$th Frobenius power of the maximal ideal as $\maxI^{[q]} = (x_1^q,\ldots,x_n^q)$
\end{notation}
\begin{definition}[{\cite[Definition 2.9]{RGpaper}}]
\label{def:lqc}
Let $f \in P$ be homogeneous, and $q$ be a power of $p$. We say that $f$ is \textit{link-$q$-compressed} if the ideal $\maxI^{[q]}:f$ is $\maxI^{[q]}$-compressed.
\end{definition}

In order to connect this definition to the alternate condition in Lemma \ref{lem:lqcDef}, we first discuss Macaulay's inverse systems.

Macaulay discovered in 1918 a one-to-one correspondence between Artinian Gorenstein algebras $P/I$ (where $P$ is the polynomial ring we've been using) and cyclic $P$-submodules of the inverse power algebra $D = k[x_1^{-1},\ldots,x_n^{-1}]$, see \cite[Section IV.]{Mac94}.
Here $x_i^{-1}$ has degree $1$ in $D$, and the $P$-module action on $D$ is defined by the $k$-linear action where \[
(x_1^{a_1} \cdots x_n^{a_n}) (x_1^{-b_1} \cdots x_n^{-b_n}) = \begin{cases} x_1^{-(b_1-a_1)} \cdots x_n^{-(b_n-a_n)} & \forall i \ b_i \geq a_i \\ 0 & \text{else} \end{cases}
\] for all $i$ and $a_1,b_1,\ldots,a_n,b_n \geq 0$. This algebra $D$ is also isomorphic to the injective hull of $k$ (as proved by Northcott in \cite{MR360555}).

We call the ring of inverse polynomials $D$ and set $S = P$ because they are respectively isomorphic to the divided power $k$-algebra and the symmetric $k$-algebra of a $k$-vector space with basis $x_1,\ldots,x_n$ (see \cite[Appendix A]{IK99}).

The correspondence between homogeneous Artinian Gorenstein ideals $I$ of $P$ and cyclic graded submodules $(\varphi)$ of $D$ is shown in \cite[Lemma 2.12]{IK99}, and is as follows:

For any homogeneous element $\varphi$ of $D$, let $I(\varphi) = (0 :_S \varphi) = \{r \in S \mid r \varphi = 0\}$ be an ideal of $S = P$, which is a graded Artinian ideal. Also, set $s = \deg\varphi$, where $s$ is the socle degree of $P/I(\varphi)$.

In the other direction: for any homogeneous Artinian Gorenstein ideal $I$ of $P$, $I^\bot = (0 :_D I) = \{\varphi \in D \mid I \varphi = 0\}$ is a cyclic $S$-submodule $(\varphi)$ of $D$ because $I$ is Artinian Gorenstein. This element $\varphi$ is called an inverse system for $I$.

Given a homogeneous inverse polynomial $\varphi$ in $D$, we have a corresponding map $\Phi: S \to D$ given by $\Phi(g) = g \varphi$ for any $g \in S$. This map then restricts by degree to a family of maps $\Phi_i := \Phi|_{S_i}: S_i \to D_{s-i}$, where $s = \deg \varphi$. Note that $I(\varphi) = \ker \Phi$ by definition.

The following gives us a way to calculate $\varphi$ when $J$ contains a Frobenius power of the maximal ideal of $P$:
\begin{lemma}[{\cite[Lemma 2.7]{RGpaper}}]
\label{lem:invSystFrob}
For any Gorenstein ideal $J \subseteq P$ that contains $\mathfrak{c} = \maxI^{[q]}$, there
exists a homogeneous $f \in P$ of degree $d$ with $1 < d < q$ such that $J$ is of the form \[
J = \maxI^{[q]} : f
\/.\]
The inverse polynomial of $J$ is \[
\varphi = f [x_1^{(q-1)} \cdots x_n^{(q-1)}]
\] where multiplication on the left by $f$ is via the $S$-module structure on $D$. Note
that $\varphi$ can also be written in inverse variables as \[
\varphi = \frac{f}{(x_1^{q-1} \cdots x_n^{q-1})}
\/.\]
In particular, its degree is \[
\deg \varphi = s = n(q-1)-d
 \quad 
\text{where}
 \quad 
d = \deg f
\/.\]

Conversely, any $\varphi$ that is a linear combination of inverse power monomials of degree $s < n(q-1)$ with each variable having power strictly less than $q$ can be written in the form above and so provides an inverse system whose associated ideal $J$ is as above.
\end{lemma}
Combining this lemma with our understanding of the $\Phi$ map from earlier, we have $\ker \Phi = I(\varphi) = J = \maxI^{[q]} : f$.

\begin{lemma}[cf. {\cite[Lemma 2.5]{RGpaper}}]
\label{lem:lqcEqvs}
Let $\mathfrak{c} = \maxI^{[q]} = (x_1^q,\ldots,x_n^q) \subseteq P$ and let $C = (X_1^q,\ldots,X_n^q)$ be the lift of $\mathfrak{c}$ to the symmetric algebra $S$ obtained by replacing each $x_i$ by $X_i$. Let $J \subseteq P$ be a homogeneous Gorenstein ideal containing $\mathfrak{c}$ with inverse polynomial $\varphi$ and let $s$ be the degree of the socle of $P/J$.
The following conditions are equivalent.
\begin{enumerate}
\item[(1)]
$J$ is $\mathfrak{c}$-compressed.
\item[(2)]
For $i \leq s/2$ the kernel of the map $\Phi_i: S_i \to D_{s-i}$ is generated by the monomials $X_1^{a_1} \cdots X_n^{a_n}$ in $S_i$ with $a_j \geq q$ for some $1 \leq j \leq n$.
\end{enumerate}
\end{lemma}

For the purpose of making our later arguments more transparent, we expand on this as follows:
\begin{lemma}
\label{lem:lqcDef}
Let $f \in P$ be a homogeneous polynomial with degree $d = \deg f$.
Then $f$ is link-$q$-compressed if and only if the nonzero elements of $(\maxI^{[q]}:f)/\maxI^{[q]}$ all have degree $>s/2$, where the socle degree $s$ of $P/(\maxI^{[q]}:f)$ is $n(q-1)-d$.
\end{lemma}
\begin{proof}
Define $\varphi = \frac{f}{(x_1^{q-1} \cdots x_n^{q-1})}$ as in Lemma \ref{lem:invSystFrob}.
Also by Lemma \ref{lem:invSystFrob}, \[
\ker \Phi = I(\varphi) = J = ((x_1^q,\ldots,x_n^q) : f)
\/,\] which means \[
\ker \Phi_i = S_i \cap \ker \Phi = ((x_1^q,\ldots,x_n^q) : f)_i = (\maxI^{[q]}:f)_i
\] for any $i$.

Given any $1 \leq j \leq n$, $(x_j^q)$ is generated as a $k$-vector space by the set of $x_1^{a_1} \cdots x_n^{a_n}$ with $a_j \geq q$, so $\maxI^{[q]} = \sum_{j=1}^n (x_j^q)$ is the ideal generated by the set of monomials $x_1^{a_1} \cdots x_n^{a_n}$ with $a_j \geq q$ for some $1 \leq j \leq n$.
Thus for any $i$, $(\maxI^{[q]})_i$ is the set generated by the monomials $x_1^{a_1} \cdots x_n^{a_n} \in P_i$ with $a_j \geq q$ for some $1 \leq j \leq n$.

From Lemma \ref{lem:lqcEqvs} we know $\maxI^{[q]}:f$ is $\maxI^{[q]}$-compressed (i.e., $f$ is link-$q$-compressed) if and only if for all $i \leq s/2$ the kernel of the map $\Phi_i: S_i \to D_{s-i}$ is generated by the monomials $X_1^{a_1} \cdots X_n^{a_n}$ in $S_i$ (which correspond to $x_1^{a_1} \cdots x_n^{a_n}$ in $P_i$) with $a_j \geq q$ for some $1 \leq j \leq n$.
We showed that $\ker\Phi_i = (\maxI^{[q]}:f)_i$, and that $(\maxI^{[q]})_i$ is the set generated by the monomials $x_1^{a_1} \cdots x_n^{a_n} \in P_i$ with $a_j \geq q$ for some $1 \leq j \leq n$, meaning this condition is equivalent to requiring $(\maxI^{[q]}:f)_i = (\maxI^{[q]})_i$ for all $i \leq s/2$. Since $(\maxI^{[q]}:f) \supseteq \maxI^{[q]}$, we could rephrase this condition further as $\left((\maxI^{[q]}:f)/\maxI^{[q]}\right)_i = 0$ for all $i \leq s/2$, meaning every element of $(\maxI^{[q]}:f)/\maxI^{[q]}$ with degree $\leq s/2$ is zero, or equivalently every nonzero element has degree $> s/2$.
\end{proof}

\subsection{Prior link-$q$-compressed results}
\label{subsec:lqcFacts}

Here we discuss more facts about link-$q$-compressed polynomials. These facts all come from Miller, Rahmati, and R.G.'s results in \cite{RGpaper}.

\begin{notation}
Let $P = k[x,y,z]$ be a standard graded polynomial ring over a field $k$ of characteristic $p > 0$ with homogeneous maximal ideal $\maxI=(x,y,z)$. Let $R = P/(f)$ with homogeneous $f \in P$ with $d=\deg f\geq2$.

Let $q$ be a power of $p$. In the following theorems we work with graded ideals of the polynomial ring $P$: the graded complete intersection ideal \[
\mathfrak{c} = (c_1,c_2,c_3)
\/,\] and the linked ideals
$\mathfrak{c}+(f)$ and $\mathfrak{c}:f = \mathfrak{c}:(\mathfrak{c}+(f))$.
Since $\mathfrak{c}+(f)$ is an almost complete intersection ideal, $P/(\mathfrak{c}:f)$ is a Gorenstein ring (see \cite[Remark 2.7]{HU87} for this and more results on linkage).

We study these to specialize to the case \[
\mathfrak{c} = \maxI^{[q]} = (x^q,y^q,z^q)
\] and the linked ideals \[
\mathfrak{c}+(f) = \maxI^{[q]} + (f) = (x^q,y^q,z^q,f) \quad \text{and} \quad \mathfrak{c}:f = \maxI^{[q]}:f
\/.\]

Assume $c_1,c_2,c_3$ are part of a minimal generating set for $\mathfrak{c}:f$. Fix a set of minimal homogeneous generators for $\mathfrak{c}:f$ as follows
\[
\mathfrak{c}:f = (c_1,c_2,c_3,w_1,\ldots,w_m) \qquad \text{so that }m = \mu(\mathfrak{c}:f) - 3
\/.\]

Given $\mathfrak{c}$ and $f$ we also have matrices $\mtrx{\psi}$ and $\mtrx{\varphi}$ with entries in $P$, where $\mtrx{\psi}$ is $m\times3$ and $\mtrx{\varphi}$ is $m\times m$ and skew-symmetric.
These matrices appear in the resolution of $P/(\mathfrak{c}:f)$ from the Buchsbaum-Eisenbud structure theorem in \cite{BE77}, as well as the following resolutions (see \cite[Lemma 5.2]{RGpaper} for the form of the structure theorem these maps come from).
\end{notation}

The following propositions show the structure of the $P$-free and $R$-free resolutions of $P/(\mathfrak{c}+(f)) = R/\mathfrak{c}$. These results are refined in Theorem \ref{thm:RGlqcBetti} where we assume $\mathfrak{c} = \maxI^{[q]}$ and $f$ is link-$q$-compressed.
\begin{proposition}[{\cite[Proposition 5.6]{RGpaper}}]
\label{prop:RGPOverCandf}
Assume that $\deg c_i > d$ for $i = 1, 2, 3$. The minimal homogeneous resolution of $P/(\mathfrak{c}+(f))$ over $P$ is of the form
\[
0
 \gets 
P
 \xleftarrow{\begin{bmatrix}\vec{c}^\top&uf\end{bmatrix}} 
\bigoplus_P^{P^3}
 \xleftarrow{\begin{bmatrix}\mtrx{\psi}^\top\mtrx{\varphi}^\vee&uf\mtrx{I}\\-\vec{w}^\top&-\vec{c}^\top\end{bmatrix}} 
\bigoplus_{P^3}^{P^m}
 \xleftarrow{\begin{bmatrix}\mtrx{\varphi}\\-\mtrx{\psi}^\top\end{bmatrix}} 
P^m
 \gets 
0
\] with $\Pf(\mtrx{\varphi}) = uf$ for some unit $u \in k$, $\vec{c}^\top = \begin{bmatrix}c_1&c_2&c_3\end{bmatrix}$ and $\vec{w}^\top = \begin{bmatrix}w_1&\cdots&w_m\end{bmatrix}$.
\end{proposition}
\begin{proposition}[{\cite[Proposition 5.8]{RGpaper}}]
\label{prop:RGROverC}
Assume that $\deg c_i > d$ for $i = 1, 2, 3$. The $R$-free resolution of $R/\mathfrak{c} = P/(\mathfrak{c}+(f))$ is \[
\cdots
 \xrightarrow{\mtrx{\varphi}} 
R^m
 \xrightarrow{\mtrx{\varphi}^\vee} 
R^m
 \xrightarrow{\mtrx{\varphi}} 
R^m
 \xrightarrow{\mtrx{\psi}^\top \mtrx{\varphi}^\vee} 
R^3
 \xrightarrow{\vec{c}^\top} 
R
 \to 
0
\] where by $\mtrx{\varphi}$, $\mtrx{\varphi}^\vee$, $\mtrx{\psi}^\top$, and $\vec{c}^\top$ above we mean the images in $R$ of these $P$-matrices.

In addition, the homogeneous skew-symmetric matrix $\mtrx{\varphi}$ has Pfaffian equal to $uf$ for some unit $u \in k$, and the pair $(\mtrx{\varphi},\mtrx{\varphi}^\vee)$ of matrices over P is the matrix factorization of $uf$ over $P$ associated to the periodic portion.
\end{proposition}
The following shows that link-$q$-compressed polynomials have nice graded Betti numbers for $R/\maxI^{[q]}$.
\begin{theorem}[{\cite[Theorem 5.10 and Theorem A]{RGpaper}}]
\label{thm:RGlqcBetti}
Let $R = k[x,y,z]/(f)$ be a standard graded hypersurface ring over a field $k$ of characteristic $p > 0$ with homogeneous maximal ideal $\maxI$. Suppose that $p$ and $d = \deg f$ have opposite parity. Let $q \geq d + 3$ be a power of $p$.

Assume further that $f$ is link-$q$-compressed. Then the following hold.
\begin{itemize}
\item[(a)]
The matrix $\mtrx{\varphi}$ in the matrix factorization $(\mtrx{\varphi},\mtrx{\varphi}^\vee)$ is a $2d\times 2d$ linear matrix with Pfaffian equal to $uf$ for some unit $u \in k$, and its Pfaffian adjoint $\mtrx{\varphi}^\vee$ is a $2d\times 2d$ matrix with entries of degree $d - 1$. The matrix $\mtrx{\psi}$ is $2d\times 3$ with entries of degree $\frac{1}{2} (q - d + 1)$. In particular, the minimal graded resolution of $R/\maxI^{[q]}$ over $R$ has the following eventually 2-periodic form \[
\cdots
 \xrightarrow{\mtrx{\varphi}} 
R^{2d}(-b-d)
 \xrightarrow{\mtrx{\varphi}^\vee} 
R^{2d}(-b-1)
 \xrightarrow{\mtrx{\varphi}} 
R^{2d}(-b)
 \xrightarrow{\mtrx{\psi}^\top \mtrx{\varphi}^\vee} 
R^3(-q)
 \xrightarrow{\vec{c}^\top} 
R
 \to 
0
\] where $b = \frac{3}{2} q + \frac{1}{2} d - \frac{1}{2}$.

As a result, for any two such values of $q$, say $q_1 > q_0 \geq d + 3$, whenever $f$ is link-$q_i$-compressed for $i = 0, 1$, the graded Betti numbers in high homological degree of the $R$-modules $R/\maxI^{[q_0]}$ and $R/\maxI^{[q_1]}$ are the same, up to a constant shift of $\frac{3}{2} (q_1 - q_0)$.

\item[(b)]
The minimal graded resolution over $P = k[x,y,z]$ of $P/(\maxI^{[q]}+(f)) = R/\maxI^{[q]}$ has the form \[
0
\to
P^{2d}(-b-1)
\xrightarrow{
\begin{bmatrix}
\mtrx{\varphi} \\ -\mtrx{\psi}^\top
\end{bmatrix}
}
\bigoplus_{P^3(-q-d)}^{P^{2d}(-b)}
\xrightarrow{
\begin{bmatrix}
\mtrx{\psi}^\top \mtrx{\varphi}^\vee & u f \mtrx{I} \\ -\vec{w}^\top & -\vec{c}^\top
\end{bmatrix}
}
\bigoplus_{P(-d)}^{P^3(-q)}
\xrightarrow{
\begin{bmatrix}
\vec{c}^\top & u f
\end{bmatrix}
}
P
\to
0
\/.\]

In particular, the Castelnuovo-Mumford regularity is given by \[
\reg(R/\maxI^{[q]}) = \frac{3}{2} q + \frac{1}{2} d - \frac{5}{2}
\/.\]
\end{itemize}
\end{theorem}
\begin{notation}
For a finitely generated graded $k$-algebra $R$ with homogeneous maximal ideal $\maxI$, the Hilbert-Kunz function of $R$ is defined as \[
HK_R(q) = \dim_k R/\maxI^{[q]}
\/.\]
\end{notation}
\begin{theorem}[{\cite[Theorem 5.11]{RGpaper}}]
\label{thm:lqcHilbKunzFunc}
Suppose that $p$ and $d$ have opposite parity. If $f\in k[x,y,z]$ is link-$q$-compressed for some $q$ then the Hilbert-Kunz function of $R$ at $q$ is \[
HK_R(q) = \frac{3}{4} dq^2 - \frac{1}{12} (d^3 - d)
\/.\]
\end{theorem}

In order to show that link-$q$-compressed is a common condition for polynomials, we note the following.
\begin{remark}
\label{rmk:lqcInGeneral}
For fixed values of $e$ and $d = \deg f$, a polynomial $f$ being link-$p^e$-compressed is a Zariski open condition on the coefficients of $f$, as was shown in \cite[Remark 2.10]{RGpaper}.
If the set of $f$ with this condition is shown to also be nonempty, we say it holds for general choices of $f$.
It is a stronger condition for a given $f$ to be link-$p^e$-compressed for all $e>0$; this would mean the coefficients of $f$ lie in a countable intersection of Zariski open sets. Such a choice of $f$ is called \textit{very general}.
\end{remark}

The following tells us about the generators of $\maxI^{[q]}:f$ when $f$ is link-$q$-compressed.
\begin{corollary}[cf. {\cite[Corollary 4.4]{RGpaper}}]
\label{cor:lqcDegOfGens}
Let $P = k[x,y,z]$, assume that $f$ is link-$q$-compressed for some $q$, with $q \geq d+3$ (so that $q \leq \frac{s}{2}$).

Then the minimal homogeneous generators for $\maxI^{[q]}:f$ in degrees $>\frac{s}{2}$ lie
in degree $\frac{s}{2}+1$ when $s$ is even.
\end{corollary}
\begin{proposition}[cf. {\cite[Proposition 4.5]{RGpaper}}]
\label{prop:lqcNumOfGens}
Let $P = k[x,y,z]$, assume that $f$ is link-$q$-compressed for some $q$, with $q \geq d+3$ (so that $q \leq \frac{s}{2}$). Then $x^q,y^q,z^q$ are part of a minimal set of generators for $\maxI^{[q]}:f$, and if $s$ is even then $\maxI^{[q]}:f$ has exactly $2d$ additional minimal generators in degree $\frac{s}{2}+1$ and no others of higher degree.
\end{proposition}
We have a similar result for the socle of $R/\maxI^{[q]}$.
\begin{theorem}[cf. {\cite[Theorem 4.1]{RGpaper}}]
\label{thm:lqcSoc}
Let $R = k[x,y,z]/(f)$ be a standard graded hypersurface ring over a field $k$ with homogeneous maximal ideal $\maxI$. Set \[
q = p^e \quad d = \deg f \quad \text{and} \quad s = 3(q - 1) - d
\/.\] Assume that $q \geq d+3$. Assume further that $f$ is link-$q$-compressed. Then the following hold for the socle module $\soc(R/\maxI^{[q]})$.
\begin{enumerate}
    \item Its generators lie in degree $s_2 = \frac{1}{2} (3(q - 1) + d - 2)$ if $s$ is even.
    \item If $s$ is even, its dimension satisfies $\dim_k \soc \left( R/\maxI^{[q]} \right)_{s_2} = 2d$.
\end{enumerate}
\end{theorem}

\section{Relevant number theory}
\label{sec:NumThry}

The following sequence of coefficients is vital for our results. This section is dedicated to proving theorems about these coefficients.
\begin{notation}
\label{not:lambda}
For any integer $t \geq 0$, we define the dyadic rational $\lambda_t = 2^{-2t} \binom{2t}{t} \in \ZZ[2^{-1}]$.
\end{notation}
\begin{remark}
For any $t\geq0$, $\lambda_t$ is well defined in any $\ZZ[2^{-1}]$-algebra, including fields which have characteristic not equal to $2$.
\end{remark}
\begin{remark}
The Maclaurin expansion of the function $(1-x)^{-1/2}$ uses the $\lambda_t=2^{-2t}\binom{2t}{t}$ coefficients:
\[
\frac{1}{\sqrt{1-x}} = \sum_{t=0}^\infty \lambda_t x^t
\]
for all $-1\leq t<1$.
\end{remark}

The following are important facts we need for several results:
\begin{lemma}
\label{lem:lambdaOdds}
For any $t \geq 0$, the following hold in any ring for which they are well-defined:
\begin{enumerate}
    \item $(2t)! = 2^t t! \prod_{h=1}^t (2h-1)$, 
    \item $2^t \prod_{h=1}^t (2h-1) = t! \binom{2t}{t}$, 
    \item $\prod_{h=1}^t (2h-1) = 2^t t! \lambda_t$, and 
    \item $\left(\prod_{h=1}^t (2h-1)\right)^2 = (2t)! \lambda_t$.
\end{enumerate}
\end{lemma}
\begin{proof}
Let $t \geq 0$.

\begin{enumerate}
    \item \label{lem:lambdabullet1}
By separating the even and odd factors of $(2t)!$, we have \[
(2t)!
 = 
\prod_{h=1}^{2t} h
 = 
\prod_{h=1}^t (2h) \prod_{h=1}^t (2h-1)
 = 
\prod_{h=1}^t 2 \prod_{h=1}^t h \prod_{h=1}^t (2h-1)
 = 
2^t t! \prod_{h=1}^t (2h-1)
\/.\]

    \item \label{lem:lambdabullet2}
We also have $t! \binom{2t}{t} = \frac{(2t)!}{t!} = 2^t \prod_{h=1}^t (2h-1)$ using \ref{lem:lambdabullet1}.

    \item \label{lem:lambdabullet3}
From \ref{lem:lambdabullet2} we can see that \[
2^t t! \lambda_t = 2^t t! \left( 2^{-2t} \binom{2t}{t} \right) = 2^{-t} t! \binom{2t}{t} = \prod_{h=1}^t (2h-1)
\/.\]

    \item \label{lem:lambdabullet4}
Finally, \[
(2t)! \lambda_t = \left(2^t t! \prod_{h=1}^t (2h-1)\right) \lambda_t = 
\left(\prod_{h=1}^t (2h-1)\right)^2
\/.\qedhere\]
\end{enumerate}
\end{proof}
\begin{remark}
Note that \ref{lem:lambdabullet1} and \ref{lem:lambdabullet2} of this previous Lemma are true in any ring since only integers are involved. On the other hand, \ref{lem:lambdabullet3} and \ref{lem:lambdabullet4} are only true in $\ZZ[2^{-1}]$-algebras because they involve $\lambda_t \in \ZZ[2^{-1}]$, as well as integers.
\end{remark}

In order to prove a vital lemma about $\lambda_t$, we recall two well-known theorems about integers modulo a prime:
\begin{theorem*}[Fermat's little theorem]
Let $p$ be a prime number. Then for any integer $a$, $a^p \equiv a    \mod p$. By applying this repeatedly, we also see that $a^q \equiv a    \mod p$ if $q$ is any power of $p$.
\end{theorem*}
\begin{theorem*}[Lucas's theorem]
Let $p$ be a prime number. Let $a,b \geq 0$ be integers and let \[
a = a_n p^n + a_{n-1} p^{n-1} + \cdots + a_0 \qquad \text{and} \qquad b = b_n p^n + b_{n-1} p^{n-1} + \cdots + b_0
\] be the base $p$ expansions of $a$ and $b$. Then \[
\binom{a}{b} \equiv \prod_{i=0}^n \binom{a_i}{b_i}    \mod p
\/,\] where we use the convention that $\binom{m}{n} = 0$ if $m<n$.
\end{theorem*}

The following lemma is the reason the $\lambda_t$ coefficients are vital to our results:
\begin{lemma}
\label{lem:lambdaBinom}
Let $p$ be an odd prime and $q = p^e$ where $e > 0$. Set $\pi = (q-1)/2$. Then, in any ring of characteristic $p>2$, \[
(-1)^t \binom{\pi}{t} = 2^{-2t} \binom{2t}{t} = \lambda_t
\] for all $0 \leq t < q$.
\end{lemma}
\begin{proof}
We prove a statement in the integers to prove the lemma. The following work shows that $2^{2t} \left( (-1)^t \binom{\pi}{t} \right) \equiv \binom{2t}{t}    \mod p$ for all $0 \leq t < q$.

Define $\pi_0 = (p-1)/2$.
Let $0 \leq t_0 \leq p-1$. Then we have
\begin{align*}
t_0! \left( 2^{2t_0} \left( (-1)^{t_0} \binom{\pi_0}{t_0} \right) \right)
=&
\left( (-1)^{t_0} 2^{2t_0} \right) \left( t_0! \binom{\pi_0}{t_0} \right)
\\=&
2^{t_0} (-2)^{t_0} \prod_{h=1}^{t_0} ((\pi_0+1)-h)
\\=&
2^{t_0} \left( \prod_{h=1}^{t_0} (-2) \prod_{h=1}^{t_0} ((\pi_0+1)-h) \right)
\\=&
2^{t_0} \prod_{h=1}^{t_0} -2((\pi_0+1)-h)
\\=&
2^{t_0} \prod_{h=1}^{t_0} ((2h-1)-(2\pi_0+1))
\\=&
2^{t_0} \prod_{h=1}^{t_0} ((2h-1)-p)
\\\equiv&
2^{t_0} \prod_{h=1}^{t_0} (2h-1)
    \mod p
\\=&
t_0! \binom{2t_0}{t_0}
\/,
\end{align*}
where the last line holds by Lemma \ref{lem:lambdaOdds}. Thus
\begin{equation}
\label{eq:LambdaBinomBase}
\tag{*}
2^{2t_0} \left( (-1)^{t_0} \binom{\pi_0}{t_0} \right) \equiv \binom{2t_0}{t_0}    \mod p
\end{equation}
since $t_0! \nequiv 0    \mod p$ and $p$ is prime.

Let $0 \leq t < q$. Write the base $p$ expansion of $t$ as $t = \sum_{i=0}^{e-1} t_i p^i$, so $0 \leq t_i < p$ for all $0 \leq i \leq e-1$.
Also note that \[
\sum_{i=0}^{e-1} \pi_0 p^i = \frac{p-1}{2} \sum_{i=0}^{e-1} p^i = \frac{p^e-1}{2} = \pi
\/.\] Then we have
\begin{align*}
2^{2t} \left( (-1)^t \binom{\pi}{t} \right)
=&
(-4)^{\sum_{i=0}^{e-1} t_i p^i}
\begin{pmatrix} {\sum_{i=0}^{e-1} \pi_0 p^i} \\ {\sum_{i=0}^{e-1} t_i p^i} \end{pmatrix}
\\=&
\left(
\prod_{i=0}^{e-1} 
(-4)^{t_i p^i}
\right)
\begin{pmatrix} {\sum_{i=0}^{e-1} \pi_0 p^i} \\ {\sum_{i=0}^{e-1} t_i p^i} \end{pmatrix}
\\=&
\left(
\prod_{i=0}^{e-1} 
\left(
(-4)^{p^i}
\right)^{t_i}
\right)
\begin{pmatrix} {\sum_{i=0}^{e-1} \pi_0 p^i} \\ {\sum_{i=0}^{e-1} t_i p^i} \end{pmatrix}
\\\equiv&
\left(
\prod_{i=0}^{e-1} 
\left(
(-4)^{p^i}
\right)^{t_i}
\right)
\left(
\prod_{i=0}^{e-1} 
\binom{\pi_0}{t_i}
\right)
    \mod p
\tag{1}
\\=&
\prod_{i=0}^{e-1} 
\left(
\left(
(-4)^{p^i}
\right)^{t_i}
\binom{\pi_0}{t_i}
\right)
\\\equiv&
\prod_{i=0}^{e-1} 
\left(
(-4)^{t_i}
\binom{\pi_0}{t_i}
\right)
    \mod p
\tag{2}
\\=&
\prod_{i=0}^{e-1}
\left(
2^{2t_i}
\left(
(-1)^{t_i}
\binom{\pi_0}{t_i}
\right)
\right)
    \mod p
\\\equiv&
\prod_{i=0}^{e-1}
\binom{2t_i}{t_i}
    \mod p
\tag{3}
\/,
\end{align*}
where (1) holds by Lucas's theorem, (2) holds by Fermat's little theorem, and (3) holds by Equation (\ref{eq:LambdaBinomBase}).

If $t_i \leq \pi_0$ for all $i$, this means that $0 \leq 2 t_i < p$, and thus $2t = \sum_{i=0}^{e-1} (2t_i) p^i$ is the base $p$ expansion of $2t$, which means by Lucas's theorem we have \[
2^{2t} \left( (-1)^t \binom{\pi}{t} \right)
 \equiv 
\prod_{i=0}^{e-1}
\binom{2t_i}{t_i}
 \equiv 
\binom{\sum_{i=0}^{e-1} 2t_i p^i}{\sum_{i=0}^{e-1} t_i p^i}
 = 
\binom{2t}{t}
    \mod p
\/.\] On the other hand, if $t$ has at least one digit greater than $\pi_0$, we show that both $2^{2t} \left( (-1)^t \binom{\pi}{t} \right)$ and $\binom{2t}{t}$ are congruent to $0    \mod p$. Let $0 \leq I \leq e-1$ be the smallest index such that $t_I > \pi_0$. Then $0 \leq 2 t_i < p$ for $i < I$ and $0 \leq 2 t_I-p < p$, so $2t = \sum_{i=0}^{I-1} (2t_i) p^i + (2t_I-p) p^I + \cdots$ is the base $p$ expansion of $2t$, up to the $I$th degree.
Note that because $t_I>\pi_0$, $\binom{\pi_0}{t_I} = 0$, and thus $\binom{2t_I}{t_I} \equiv 2^{2t_I} \left( (-1)^{t_I} \binom{\pi_0}{t_I} \right) = 0    \mod p$ by Equation (\ref{eq:LambdaBinomBase}). Note that we also have $\binom{2t_I-p}{t_I} = 0$ because $t_I>2t_I-p$, which follows from $t_I<p$. Thus, we have \[
\binom{2t}{t}
 = 
\binom{\sum_{i=0}^{I-1} (2t_i) p^i + (2t_I-p) p^I + \cdots}{\sum_{i=0}^{I-1} t_i p^i + t_I p^I + \cdots}
 \equiv 
\prod_{i=0}^{I-1} 
\binom{2t_i}{t_i}
\binom{2t_I-p}{t_I}
(\cdots)
 = 
\prod_{i=0}^{I-1} 
\binom{2t_i}{t_i}
(0)
(\cdots)
 = 
0
    \mod p
\] because of Lucas's theorem and $\binom{2t_I-p}{t_I} = 0$, and \[
2^{2t} \left( (-1)^t \binom{\pi}{t} \right)
 \equiv 
\prod_{i=0}^{e-1} 
\binom{2t_i}{t_i}
 = 
\prod_{i=0}^{I-1} 
\binom{2t_i}{t_i}
\binom{2t_I}{t_I}
\prod_{i=I+1}^{e-1} 
\binom{2t_i}{t_i}
 \equiv 
\prod_{i=0}^{I-1} 
\binom{2t_i}{t_i}
(0)
\prod_{i=I+1}^{e-1} 
\binom{2t_i}{t_i}
 = 
0
    \mod p
\] because $\binom{2t_I}{t_I} \equiv 0    \mod p$. 
Therefore $2^{2t} \left( (-1)^t \binom{\pi}{t} \right) \equiv \binom{2t}{t}    \mod p$ for all $0 \leq t < q$.
\end{proof}
\begin{lemma}
\label{lem:zq}
Let $p$ be an odd prime, with $q>1$ a power of $p$. Define $\pi=\frac{q-1}{2}$. In $S[x,y,z]$, where $S$ is any characteristic $p$ ring, we have \[
z^q = z \sum_{t=0}^\pi \lambda_t (xy)^{\pi-t} (xy-z^2)^t
\/.\]
\end{lemma}
\begin{proof}
We see that
\begin{align*}
z^q
=&
z(z^2)^\pi
\\=&
z(xy-(xy-z^2))^\pi
\\=&
z \sum_{t=0}^\pi (-1)^t \binom{\pi}{t} (xy)^{\pi-t} (xy-z^2)^t
\\=&
z \sum_{t=0}^\pi \lambda_t (xy)^{\pi-t} (xy-z^2)^t
\/,
\end{align*}
where the third equality follows from binomial expansion and the fourth equality follows from Lemma \ref{lem:lambdaBinom}.
\end{proof}

\section{Determinant calculations}
\label{sec:DetCalcs}

The following matrices play such a vital role in our results that we dedicate this entire section to calculating certain determinants of it.

The matrices and determinants we describe in this section exist in $\ZZ[x,y,z]$, so they can exist in $S[x,y,z]$ where $S$ is any $\ZZ$-algebra. Occasionally we will have to assume these matrices and determinants exist in $\ZZ[2^{-1}][x,y,z]$, in which case they can exist in $S[x,y,z]$ where $S$ is any $\ZZ[2^{-1}]$-algebra.
We note the few times we must make this additional assumption when they happen.
In either case, $S$ can be a field of odd characteristic.
\begin{notation}
\label{not:MandLn}
Let $d$ be a positive integer. Define the $d \times d$ matrix \[
\mtrx{M}
 = 
\begin{bmatrix}
-(d-1)z & 1x \\ -(d-1)y & -(d-3)z & 2x \\& -(d-2)y & -(d-5)z & 3x \\&& -(d-3)y & \ddots & \ddots \\&&& \ddots & \ddots & (d-2)x \\&&&& -2y & (d-3)z & (d-1)x \\&&&&& -1y & (d-1)z
\end{bmatrix}
\/.\] The entries of $\mtrx{M}$ are in $\ZZ[x,y,z]$.

We also define a generalization of $\mtrx{M}$: for every $n \geq 0$, let $\mtrx{L}_n$ denote the $n \times n$ matrix \[
\begin{bmatrix}
-(d-1)z & 1x \\ -(d-1)y & -(d-3)z & 2x \\& -(d-2)y & \ddots & \ddots \\&& \ddots & \ddots & ((n-1)-1)x \\&&& -((d+1)-(n-1))y & (2(n-1)-(d+1))z & (n-1)x \\&&&& -((d+1)-n)y & (2n-(d+1))z
\end{bmatrix}
\] with entries in $\ZZ[x,y,z]$.

Explicitly, the entries of both $\mtrx{M}$ and $\mtrx{L}_n$ are:
\begin{itemize}
\item $(2i-(d+1))z=(2j-(d+1))z$ on the main diagonal ($i=j$),
\item $ix=(j-1)x$ on the upper diagonal ($i=j-1$),
\item $((d+1)-i)y=(d-j)y$ on the lower diagonal ($i=j+1$), and 
\item $0$ everywhere else.
\end{itemize}
From this it is clear that $\mtrx{M} = \mtrx{L}_d$, and in fact $\mtrx{L}_n$ consists of the first $n$ rows and columns of $\mtrx{M}$ for any $n \leq d$.
\end{notation}

Besides the determinant of $\mtrx{M}$ itself, we need to know $\det\left(\mtrx{M}_{(\hat{d}),(\hat{1})}\right)$, $\det\left(\mtrx{M}_{(\hat{1}),(\hat{d})}\right)$, and $\det\left(\mtrx{M}_{(\hat{1}),(\hat{1})}\right)$.
We find $\det\left(\mtrx{M}_{(\hat{d}),(\hat{1})}\right)$ and $\det\left(\mtrx{M}_{(\hat{1}),(\hat{d})}\right)$ here:
\begin{proposition}
\label{prop:detMd1and1d}
We have
$\det\left(\mtrx{M}_{(\hat{d}),(\hat{1})}\right) = (d-1)! x^{d-1}$
and
$\det\left(\mtrx{M}_{(\hat{1}),(\hat{d})}\right) = (-1)^{d-1} (d-1)! y^{d-1}$.
\end{proposition}
\begin{proof}
When we remove the first column and last row from the matrix $\mtrx{M}$, we get the matrix \[
\mtrx{M}_{(\hat{d}),(\hat{1})}
 = 
\begin{bmatrix}
1x \\ -(d-3)z & 2x \\ -(d-2)y & -(d-5)z & 3x \\& -(d-3)y & \ddots & \ddots \\&& \ddots & \ddots & (d-2)x \\&&& -2y & (d-3)z & (d-1)x
\end{bmatrix}
\/.\] Notice that this is a lower triangular matrix, with $h x$ the $h$th diagonal entry. This means that \[
\det\left(\mtrx{M}_{(\hat{d}),(\hat{1})}\right)
 = \prod_{h=1}^{d-1} (hx) = 
(d-1)! x^{d-1}
\/.\]

We can see that $\mtrx{M}_{(\hat{1}),(\hat{d})}$ is \[
\begin{bmatrix}
-(d-1)y & -(d-3)z & 2x \\& -(d-2)y & -(d-5)z & 3x \\&& -(d-3)y & \ddots & \ddots \\&&& \ddots & \ddots & (d-2)x \\&&&& -2y & (d-3)z \\&&&&& -1y
\end{bmatrix}
\/,\] which is upper triangular with $-(d-h) y$ as the $h$th diagonal entry. Thus we have \[
\det\left(\mtrx{M}_{(\hat{1}),(\hat{d})}\right)
 = \prod_{h=1}^{d-1} (-(d-h)y) = 
(-1)^{d-1} (d-1)! y^{d-1}
\/.\qedhere\]
\end{proof}

We show that $A_n$ in the following Notation is a formula for $\det\mtrx{L}_n$:
\begin{notation}
\label{not:An}
For any $t$ and $N$ with $0\leq t\leq N$, and any $\nu\in\{0,1\}$, define \[
a_{t,N,\nu}
 = 
(-1)^{t+\nu}
\left(
\prod_{h=1}^{N+t+\nu} (d-(2h-1))
\right)
\left(
\prod_{h=t+1}^N (d-2h)
\right)
\binom{N}{t}
 \in\ZZ
\/.\] Note that these products are well-defined since $0\leq t\leq N$.

Set $F = xy-z^2$.

Given $n\geq0$, set $n=2N+\nu$ with $N \geq 0$ and $\nu \in \{0,1\}$. Then we define
\begin{align*}
A_n
 = 
A_{2N+\nu}
 =& 
z^{\nu}
\sum_{t=0}^N 
a_{t,N,\nu}
(xy)^{N-t}
F^t
\\=&
z^{\nu}
\sum_{t=0}^N 
(-1)^{t+\nu}
\left(
\prod_{h=1}^{N+t+\nu} 
(d-(2h-1))
\right)
\left(
\prod_{h=t+1}^{N} 
(d-2h)
\right)
\binom{N}{t}
(xy)^{N-t}
F^t
 \in\ZZ[x,y,z]
\/.\end{align*}
\end{notation}
We prove that $\det\mtrx{L}_n = A_n$ for all $n\geq0$ by showing that both sequences satisfy the following homogeneous linear recurrence relation and initial values:
\begin{itemize}
    \item $u_0 = 1$,
    \item $u_1 = -(d-1)z$, and
    \item $u_n = (n-1)(d-(n-1)) xy u_{n-2} - (d-(2n-1)) z u_{n-1}$ for all $n \geq 2$.
\end{itemize}
We begin with $\det\mtrx{L}_n$:
\begin{lemma}
\label{lem:detLnRecursive}
The following are true regarding the sequence $\det(\mtrx{L}_n)$:
\begin{itemize}
    \item $\det(\mtrx{L}_0) = 1$,
    \item $\det(\mtrx{L}_1) = -(d-1)z$, and
    \item $\det(\mtrx{L}_n) = (n-1)(d-(n-1)) xy \det(\mtrx{L}_{n-2}) - (d-(2n-1)) z \det(\mtrx{L}_{n-1})$ for all $n \geq 2$.
\end{itemize}
\end{lemma}
\begin{proof}
For any $n \geq 2$, we apply the cofactor expansion formula repeatedly over the last columns to show the following:
\begin{align*}&
\det(\mtrx{L}_n)
\\=&
\det
\begin{bmatrix}
-(d-1)z & 1x \\ -(d-1)y & \ddots & \ddots \\& \ddots & \ddots & (n-3)x \\&& -((d+3)-n)y & (2n-(d+5))z & (n-2)x \\&&& -((d+2)-n)y & (2n-(d+3))z & (n-1)x \\&&&& -((d+1)-n)y & (2n-(d+1))z
\end{bmatrix}
\\=&
(-1)^{n+n}
(2n-(d+1))z
\det
\begin{bmatrix}
-(d-1)z & 1x \\ -(d-1)y & \ddots & \ddots \\& \ddots & \ddots & (n-3)x \\&& -((d+3)-n)y & (2n-(d+5))z & (n-2)x \\&&& -((d+2)-n)y & (2n-(d+3))z
\end{bmatrix}
\\&
+
(-1)^{n-1+n}
(n-1)x
\det
\begin{bmatrix}
-(d-1)z & 1x \\ -(d-1)y & \ddots & \ddots \\& \ddots & \ddots & (n-3)x \\&& -((d+3)-n)y & (2n-(d+5))z & (n-2)x \\&&& 0 & -((d+1)-n)y
\end{bmatrix}
\\=&
-(d-(2n-1)) z
\det(\mtrx{L}_{n-1})
\\&
-
(n-1)x
(-1)^{(n-1)+(n-1)}
(-((d+1)-n)y)
\det
\begin{bmatrix}
-(d-1)z & 1x \\ -(d-1)y & \ddots & \ddots \\& \ddots & \ddots & (n-3)x \\&& -((d+3)-n)y & (2n-(d+5))z
\end{bmatrix}
\\=&
-(d-(2n-1)) z
\det(\mtrx{L}_{n-1})
+
(n-1) (d-(n-1))
x y
\det(\mtrx{L}_{n-2})
\/.
\end{align*}
This gives us a recursive formula:
\[
\det(\mtrx{L}_n)
 = 
(n-1)(d-(n-1)) xy \det(\mtrx{L}_{n-2}) - (d-(2n-1)) z \det(\mtrx{L}_{n-1})
\] for all $n \geq 2$.
We also have $\det(\mtrx{L}_0) = 1$ because $\mtrx{L}_0$ is a $0\times0$ matrix, and $\det(\mtrx{L}_1) = \det\begin{bmatrix} -(d-1)z \end{bmatrix} = -(d-1)z$.
\end{proof}
For $A_n$, we start with the initial conditions:
\begin{lemma}
\label{lem:AnRecursiveInitial}
The following are true regarding the sequence $A_n$:
\begin{itemize}
    \item $A_0 = 1$ and
    \item $A_1 = -(d-1)z$.
\end{itemize}
\end{lemma}
\begin{proof}
For $\nu \in \{0,1\}$ we have
\begin{align*}
A_\nu
 = 
A_{2(0)+\nu}
=&
z^{\nu}
\sum_{t=0}^0 
(-1)^{t+\nu}
\left(
\prod_{h=1}^{0+t+\nu} 
(d-(2h-1))
\right)
\left(
\prod_{h=t+1}^0 
(d-2h)
\right)
\binom{0}{t}
(xy)^{0-t}
F^t
\\=&
z^{\nu}
(-1)^{0+\nu}
\left(
\prod_{h=1}^{0+0+\nu} 
(d-(2h-1))
\right)
\left(
\prod_{h=0+1}^0 
(d-2h)
\right)
\binom{0}{0}
(xy)^{0-0}
F^0
\\=&
(-1)^\nu
z^\nu
\prod_{h=1}^\nu 
(d-(2h-1))
\end{align*}
by Notation \ref{not:An}, and thus we have \[
A_0
 = 
(-1)^0
z^0
\prod_{h=1}^0 
(d-(2h-1))
 = 
1
\qquad\text{and}\qquad
A_1
 = 
(-1)^1
z^1
\prod_{h=1}^1 
(d-(2h-1))
 = 
-(d-1) z
\/.\qedhere\]
\end{proof}
Next we show that $A_n$ satisfies the recursive formula, in both the cases where $n$ is even and $n$ is odd.
\begin{notation}
\label{not:tildean}
For any $N\geq0$, any $t\in\ZZ$, and any $\nu\in\{0,1\}$, define \[
\tilde{a}_{t,N,\nu}
 = 
(-1)^{t+\nu}
\left(
\prod_{h=1}^{N+t+\nu} (\tilde{d}-(2h-1))
\right)
\left(
\prod_{h=t+1}^N (\tilde{d}-2h)
\right)
\binom{N}{t}
\/.\] Here $\tilde{d}$ is an indeterminate, so $\tilde{a}_{t,N,\nu}$ is in $\ZZ(\tilde{d})$.
We use the conventions that $\binom{\beta}{\alpha}=0$ unless $0\leq\alpha\leq\beta$, and that $\prod_{h=\alpha}^\beta f(h) = \left( \prod_{h=\beta+1}^{\alpha-1} f(h) \right)^{-1}$ for $\beta<\alpha-1$.
\end{notation}
\begin{remark}
\label{rmk:tildean}
Unless $0\leq t\leq N$, $\tilde{a}_{t,N,\nu}=0$ because $\binom{N}{t}=0$. Also unless $0\leq t\leq N$, one of the products $\prod_{h=1}^{N+t+\nu} (\tilde{d}-(2h-1))$ or $\prod_{h=t+1}^N (\tilde{d}-2h)$ might need to be inverted (using the convention mentioned above, that $\prod_{h=\alpha}^\beta f(h) = \left( \prod_{h=\beta+1}^{\alpha-1} f(h) \right)^{-1}$ for $\beta<\alpha-1$); this is the reason we introduce $\tilde{d}$ and define $\tilde{a}_{t,N,\nu}$ as an element of the field of rational fractions $\ZZ(\tilde{d})$. Even though $\tilde{a}_{t,N,\nu}$ is $0$ unless $0\leq t\leq N$, these changes are necessary to make all factors of $\tilde{a}_{t,N,\nu}$ well-defined for all $t\in\ZZ$.

On the other hand, if $0\leq t\leq N$, then $\tilde{a}_{t,N,\nu} \in \ZZ[\tilde{d}]$, and in fact $\tilde{a}_{t,N,\nu}$ with $\tilde{d}$ mapped to $d$ is $a_{t,N,\nu}$, the coefficient of $(xy)^{N-t}F^t$ in the formula for $A_{2N+\nu}$ from Notation \ref{not:An}.
\end{remark}
\begin{notation}
\label{not:tildeAn}
Given $n\geq0$, set $n=2N+\nu$ with $N\geq0$ and $\nu\in\{0,1\}$. Then we define \[
\tilde{A}_n
 = 
\tilde{A}_{2N+\nu}
 := 
z^\nu \sum_{t=0}^N \tilde{a}_{t,N,\nu} (xy)^{N-t} F^t
\] as an element of $\ZZ[\tilde{d}][x,y,z]$.
\end{notation}
\begin{remark}
\label{rmk:tildeAn}
If we map $\tilde{d}$ to $d$, then $\tilde{A}_{2N+\nu}$ becomes $A_{2N+\nu}$ from Notation \ref{not:An}.
Also, if we treat $\tilde{A}_n$ as an element of $\ZZ(\tilde{d})(x,y,z)$, we write \[
\tilde{A}_{2N+\nu}
 = 
(xy)^N z^\nu \sum_{t=-\infty}^\infty \tilde{a}_{t,N,\nu} \left(\frac{F}{xy}\right)^t
\/,\] since $\tilde{a}_{t,N,\nu} = 0$ unless $0\leq t\leq N$.
\end{remark}
First we prove the odd case:
\begin{lemma}
\label{lem:AnRecursiveFormulaOdd}
For all odd $n \geq 2$, we have \[
A_n = (n-1)(d-(n-1)) xy A_{n-2} - (d-(2n-1)) z A_{n-1}
\/.\]
\end{lemma}
\begin{proof}
Fix $N\geq1$. We first prove that
\begin{equation}
\label{eq:anRecursiveFormulaOdd}
\tag{*}
\tilde{a}_{t,N,1} = 2N(\tilde{d}-2N) \tilde{a}_{t,N-1,1} - (\tilde{d}-(4N+1)) \tilde{a}_{t,N,0}
\end{equation}
for all $t \in \ZZ$.
To prove this, we note that
\begin{itemize}
\item $\tilde{a}_{t,N,1} = (-1)^{t+1} \prod_{h=1}^{N+t+1} (\tilde{d}-(2h-1)) \prod_{h=t+1}^N (\tilde{d}-2h) \binom{N}{t}$,
\item $\tilde{a}_{t,N-1,1} = (-1)^{t+1} \prod_{h=1}^{N+t} (\tilde{d}-(2h-1)) \prod_{h=t+1}^{N-1} (\tilde{d}-2h) \binom{N-1}{t}$, and
\item $\tilde{a}_{t,N,0} = (-1)^t \prod_{h=1}^{N+t} (\tilde{d}-(2h-1)) \prod_{h=t+1}^N (\tilde{d}-2h) \binom{N}{t}$
\end{itemize}
for any $t\in\ZZ$ by Notation \ref{not:tildean}. Now we see that
\begin{align*}
2N(\tilde{d}-2N)& \tilde{a}_{t,N-1,1} - (\tilde{d}-(4N+1)) \tilde{a}_{t,N,0}
\\=&
2N(\tilde{d}-2N) \left( (-1)^{t+1} \prod_{h=1}^{N+t} (\tilde{d}-(2h-1)) \prod_{h=t+1}^{N-1} (\tilde{d}-2h) \binom{N-1}{t} \right)
\\&-
(\tilde{d}-(4N+1)) \left( (-1)^t \prod_{h=1}^{N+t} (\tilde{d}-(2h-1)) \prod_{h=t+1}^N (\tilde{d}-2h) \binom{N}{t} \right)
\\=&
(-1)^{t+1} \prod_{h=1}^{N+t} (\tilde{d}-(2h-1))
\left(
2N \prod_{h=t+1}^N (\tilde{d}-2h) \binom{N-1}{t}
+
(\tilde{d}-(4N+1)) \prod_{h=t+1}^N (\tilde{d}-2h) \binom{N}{t}
\right)
\\=&
(-1)^{t+1} \prod_{h=1}^{N+t} (\tilde{d}-(2h-1)) \prod_{h=t+1}^N (\tilde{d}-2h)
\left(
2N \binom{N-1}{t}
+
(\tilde{d}-(4N+1)) \binom{N}{t}
\right)
\\=&
(-1)^{t+1} \prod_{h=1}^{N+t} (\tilde{d}-(2h-1)) \prod_{h=t+1}^N (\tilde{d}-2h)
\left(
2(N-t) \binom{N}{t}
+
(\tilde{d}-(4N+1)) \binom{N}{t}
\right)
\\=&
(-1)^{t+1} \prod_{h=1}^{N+t} (\tilde{d}-(2h-1)) \prod_{h=t+1}^N (\tilde{d}-2h)
(
2(N-t)
+
(\tilde{d}-(4N+1))
)
\binom{N}{t}
\\=&
(-1)^{t+1} \prod_{h=1}^{N+t} (\tilde{d}-(2h-1)) \prod_{h=t+1}^N (\tilde{d}-2h) (\tilde{d}-(2(N+t+1)-1)) \binom{N}{t}
\\=&
(-1)^{t+1} \prod_{h=1}^{N+t+1} (\tilde{d}-(2h-1)) \prod_{h=t+1}^N (\tilde{d}-2h) \binom{N}{t}
\\=&
\tilde{a}_{t,N,1}
\/.\end{align*}
Here we made use of the fact that $N\binom{N-1}{t} = (N-t)\binom{N}{t}$ for all $t\in\ZZ$.

Using Remark \ref{rmk:tildeAn}, we write
\begin{itemize}
\item $\tilde{A}_{2N+1} = (xy)^N z \sum_{t=-\infty}^\infty \tilde{a}_{t,N,1} \left(\frac{F}{xy}\right)^t$,
\item $\tilde{A}_{2N} = (xy)^N \sum_{t=-\infty}^\infty \tilde{a}_{t,N,0} \left(\frac{F}{xy}\right)^t$, and
\item $\tilde{A}_{2(N-1)+1} = (xy)^{N-1} z \sum_{t=-\infty}^\infty \tilde{a}_{t,N-1,1} \left(\frac{F}{xy}\right)^t$.
\end{itemize}
Using Equation (\ref{eq:anRecursiveFormulaOdd}), which we proved earlier, we see that
\begin{align*}
2N(\tilde{d}-2N) &xy \tilde{A}_{2(N-1)+1} - (\tilde{d}-(4N+1)) z \tilde{A}_{2N}
\\=&
2N(\tilde{d}-2N) xy \left( (xy)^{N-1} z \sum_{t=-\infty}^\infty \tilde{a}_{t,N-1,1} \left(\frac{F}{xy}\right)^t \right) - (\tilde{d}-(4N+1)) z \left( (xy)^N \sum_{t=-\infty}^\infty \tilde{a}_{t,N,0} \left(\frac{F}{xy}\right)^t \right)
\\=&
(xy)^N z \sum_{t=-\infty}^\infty \left(
2N(\tilde{d}-2N) \tilde{a}_{t,N-1,1} - (\tilde{d}-(4N+1)) \tilde{a}_{t,N,0}
\right) \left(\frac{F}{xy}\right)^t
\\=&
(xy)^N z \sum_{t=-\infty}^\infty \tilde{a}_{t,N,1} \left(\frac{F}{xy}\right)^t
\\=&
\tilde{A}_{2N+1}
\end{align*}
which proves that
\[
\tilde{A}_{2N+1} = 2N(\tilde{d}-2N) xy \tilde{A}_{2(N-1)+1} - (\tilde{d}-(4N+1)) z \tilde{A}_{2N}
\] for $N\geq1$. When we set $\tilde{d}$ to $d$, this proves that \[
A_{2N+1} = 2N(d-2N) xy A_{2(N-1)+1} - (d-(4N+1)) z A_{2N}
\/.\]
If we write $n=2N+1$, then this proves the lemma.
\end{proof}
And now we prove the even case:
\begin{lemma}
\label{lem:AnRecursiveFormulaEven}
For all even $n \geq 2$, we have \[
A_n = (n-1)(d-(n-1)) xy A_{n-2} - (d-(2n-1)) z A_{n-1}
\/.\]
\end{lemma}
\begin{proof}
Fix $N\geq1$. We first want to prove that
\begin{equation}
\label{eq:anRecursiveFormulaEven}
\tilde{a}_{t,N,0} = (2N-1)(\tilde{d}-(2N-1)) \tilde{a}_{t,N-1,0} - (\tilde{d}-(4N-1)) (\tilde{a}_{t,N-1,1}-\tilde{a}_{t-1,N-1,1})
\tag{*}\end{equation}
for all $t\in\ZZ$.
To prove this, we see that
\begin{itemize}
\item $\tilde{a}_{t,N,0} = (-1)^t \prod_{h=1}^{N+t} (\tilde{d}-(2h-1)) \prod_{h=t+1}^N (\tilde{d}-2h) \binom{N}{t}$,
\item $\tilde{a}_{t,N-1,0} = (-1)^t \prod_{h=1}^{N-1+t} (\tilde{d}-(2h-1)) \prod_{h=t+1}^{N-1} (\tilde{d}-2h) \binom{N-1}{t}$,
\item $\tilde{a}_{t,N-1,1} = (-1)^{t+1} \prod_{h=1}^{N+t} (\tilde{d}-(2h-1)) \prod_{h=t+1}^{N-1} (\tilde{d}-2h) \binom{N-1}{t}$, and
\item $\tilde{a}_{t-1,N-1,1} = (-1)^t \prod_{h=1}^{N+t-1} (\tilde{d}-(2h-1)) \prod_{h=t}^{N-1} (\tilde{d}-2h) \binom{N-1}{t-1}$
\end{itemize}
for any $t\in\ZZ$ by Notation \ref{not:tildean}. Now we see that
\begin{align*}&
\tilde{a}_{t,N-1,1}-\tilde{a}_{t-1,N-1,1}
\\=&
(-1)^{t+1} \prod_{h=1}^{N+t} (\tilde{d}-(2h-1)) \prod_{h=t+1}^{N-1} (\tilde{d}-2h) \binom{N-1}{t}
-
(-1)^t \prod_{h=1}^{N+t-1} (\tilde{d}-(2h-1)) \prod_{h=t}^{N-1} (\tilde{d}-2h) \binom{N-1}{t-1}
\\=&
-(-1)^t
\prod_{h=1}^{N+t-1} (\tilde{d}-(2h-1))
\left(
(\tilde{d}-(2(N+t)-1)) \prod_{h=t+1}^{N-1} (\tilde{d}-2h) \binom{N-1}{t}
+
\prod_{h=t}^{N-1} (\tilde{d}-2h) \binom{N-1}{t-1}
\right)
\\=&
-(-1)^t
\prod_{h=1}^{N+t-1} (\tilde{d}-(2h-1))
\prod_{h=t+1}^{N-1} (\tilde{d}-2h)
\left(
((\tilde{d}-2t)-(2N-1)) \binom{N-1}{t}
+
(\tilde{d}-2t) \binom{N-1}{t-1}
\right)
\\=&
-(-1)^t
\prod_{h=1}^{N+t-1} (\tilde{d}-(2h-1))
\prod_{h=t+1}^{N-1} (\tilde{d}-2h)
\left(
(\tilde{d}-2t)
\left( \binom{N-1}{t} + \binom{N-1}{t-1} \right)
-
(2N-1) \binom{N-1}{t}
\right)
\\=&
-(-1)^t
\prod_{h=1}^{N+t-1} (\tilde{d}-(2h-1))
\prod_{h=t+1}^{N-1} (\tilde{d}-2h)
\left(
(\tilde{d}-2t) \binom{N}{t}
-
(2N-1) \binom{N-1}{t}
\right)
\/.\end{align*}
where the last line holds because $\binom{N}{t} = \binom{N-1}{t} + \binom{N-1}{t-1}$ for all $t \in \ZZ$. We use this to show that
\begin{align}&\notag
(2N-1)(\tilde{d}-(2N-1)) \tilde{a}_{t,N-1,0} - (\tilde{d}-(4N-1)) (\tilde{a}_{t,N-1,1}-\tilde{a}_{t-1,N-1,1})
\\=&\notag
(2N-1)(\tilde{d}-(2N-1)) \left( (-1)^t \prod_{h=1}^{N-1+t} (\tilde{d}-(2h-1)) \prod_{h=t+1}^{N-1} (\tilde{d}-2h) \binom{N-1}{t} \right)
\\&-\notag
(\tilde{d}-(4N-1)) \left(
-(-1)^t
\prod_{h=1}^{N+t-1} (\tilde{d}-(2h-1))
\prod_{h=t+1}^{N-1} (\tilde{d}-2h)
\left(
(\tilde{d}-2t) \binom{N}{t}
-
(2N-1) \binom{N-1}{t}
\right)
\right)
\\=&\notag
(-1)^t \prod_{h=1}^{N+t-1} (\tilde{d}-(2h-1)) \prod_{h=t+1}^{N-1} (\tilde{d}-2h)
\\&\notag
\left(
(2N-1)(\tilde{d}-(2N-1)) \binom{N-1}{t}
+
(\tilde{d}-(4N-1))
\left(
(\tilde{d}-2t) \binom{N}{t}
-
(2N-1) \binom{N-1}{t}
\right)
\right)
\\=&\notag
(-1)^t \prod_{h=1}^{N+t-1} (\tilde{d}-(2h-1)) \prod_{h=t+1}^{N-1} (\tilde{d}-2h)
\\&\notag
\left(
(2N-1) \left(
(\tilde{d}-(2N-1))
-
(\tilde{d}-(4N-1))
\right)
\binom{N-1}{t}
+
(\tilde{d}-(4N-1)) (\tilde{d}-2t) \binom{N}{t}
\right)
\\=&\notag
(-1)^t \prod_{h=1}^{N+t-1} (\tilde{d}-(2h-1)) \prod_{h=t+1}^{N-1} (\tilde{d}-2h)
\left(
(2N-1) (2N) \binom{N-1}{t}
+
(\tilde{d}-(4N-1)) (\tilde{d}-2t) \binom{N}{t}
\right)
\\=&
(-1)^t \prod_{h=1}^{N+t-1} (\tilde{d}-(2h-1)) \prod_{h=t+1}^{N-1} (\tilde{d}-2h)
\left(
(2N-1) (2(N-t)) \binom{N}{t}
+
(\tilde{d}-(4N-1)) (\tilde{d}-2t) \binom{N}{t}
\right)
\\=&\notag
(-1)^t \prod_{h=1}^{N+t-1} (\tilde{d}-(2h-1)) \prod_{h=t+1}^{N-1} (\tilde{d}-2h)
\left(
2(2N-1) (N-t)
+
(\tilde{d}-(4N-1)) (\tilde{d}-2t)
\right) \binom{N}{t}
\\=&
(-1)^t \prod_{h=1}^{N+t-1} (\tilde{d}-(2h-1)) \prod_{h=t+1}^{N-1} (\tilde{d}-2h)
\left(
(\tilde{d}-(2(N+t)-1)) (\tilde{d}-2N)
\right) \binom{N}{t}
\\=&\notag
(-1)^t \prod_{h=1}^{N+t} (\tilde{d}-(2h-1)) \prod_{h=t+1}^N (\tilde{d}-2h) \binom{N}{t}
\\=&\notag
\tilde{a}_{t,N,0}
\end{align}
where line (1) holds because $N \binom{N-1}{t} = (N-t) \binom{N}{t}$ for any $N\geq1$ and any $t \in \ZZ$, and line (2) holds because
\begin{align*}
2(2N-1) (N-t) + (\tilde{d}-(4N-1))& (\tilde{d}-2t)
\\=&
2(2N-1) (N-t) + (\tilde{d}-(4N-1)) ((\tilde{d}-2N)+2(N-t))
\\=&
2((2N-1) + (\tilde{d}-(4N-1)))(N-t) + (\tilde{d}-(4N-1)) (\tilde{d}-2N)
\\=&
2(\tilde{d} - 2N)(N-t) + (\tilde{d}-(4N-1)) (\tilde{d}-2N)
\\=&
(2(N-t) + (\tilde{d}-(4N-1))) (\tilde{d}-2N)
\\=&
(\tilde{d}-(2(N+t)-1)) (\tilde{d}-2N)
\/.
\end{align*}

Using Remark \ref{rmk:tildeAn}, we write
\begin{itemize}
\item $\tilde{A}_{2N} = (xy)^N \sum_{t=-\infty}^\infty \tilde{a}_{t,N,0} \left(\frac{F}{xy}\right)^t$,
\item $\tilde{A}_{2(N-1)+1} = (xy)^{N-1} z \sum_{t=-\infty}^\infty \tilde{a}_{t,N-1,1} \left(\frac{F}{xy}\right)^t$, and
\item $\tilde{A}_{2(N-1)} = (xy)^{N-1} \sum_{t=-\infty}^\infty \tilde{a}_{t,N-1,0} \left(\frac{F}{xy}\right)^t$.
\end{itemize}
We can see that $z^2=xy-F=xy\left(1-\frac{F}{xy}\right)$, and with this we prove that
\begin{align*}
z \tilde{A}_{2(N-1)+1}
=&
z \left( (xy)^{N-1} z \sum_{t=-\infty}^\infty \tilde{a}_{t,N-1,1} \left(\frac{F}{xy}\right)^t \right)
\\=&
(xy)^{N-1} z^2 \sum_{t=-\infty}^\infty \tilde{a}_{t,N-1,1} \left(\frac{F}{xy}\right)^t
\\=&
(xy)^N \left(1-\frac{F}{xy}\right) \sum_{t=-\infty}^\infty \tilde{a}_{t,N-1,1} \left(\frac{F}{xy}\right)^t
\\=&
(xy)^N
\left(
\sum_{t=-\infty}^\infty \tilde{a}_{t,N-1,1} \left(\frac{F}{xy}\right)^t
-
\frac{F}{xy} \sum_{t=-\infty}^\infty \tilde{a}_{t,N-1,1} \left(\frac{F}{xy}\right)^t
\right)
\\=&
(xy)^N
\left(
\sum_{t=-\infty}^\infty \tilde{a}_{t,N-1,1} \left(\frac{F}{xy}\right)^t
-
\sum_{t=-\infty}^\infty \tilde{a}_{t,N-1,1} \left(\frac{F}{xy}\right)^{t+1}
\right)
\\=&
(xy)^N
\left(
\sum_{t=-\infty}^\infty \tilde{a}_{t,N-1,1} \left(\frac{F}{xy}\right)^t
-
\sum_{t=-\infty}^\infty \tilde{a}_{t-1,N-1,1} \left(\frac{F}{xy}\right)^t
\right)
\tag{**}
\\=&
(xy)^N \sum_{t=-\infty}^\infty (\tilde{a}_{t,N-1,1}-\tilde{a}_{t-1,N-1,1}) \left(\frac{F}{xy}\right)^t
\/,\end{align*}
where (**) holds because the sum on the right is reindexed ($t$ is replaced with $t-1$).
Using Equation (\ref{eq:anRecursiveFormulaEven}), we see that
\begin{align*}&
(2N-1)(\tilde{d}-(2N-1)) xy \tilde{A}_{2(N-1)} - (\tilde{d}-(4N-1)) z \tilde{A}_{2(N-1)+1}
\\=&
(2N-1)(\tilde{d}-(2N-1)) xy 
\left( (xy)^{N-1} \sum_{t=-\infty}^\infty \tilde{a}_{t,N-1,0} \left(\frac{F}{xy}\right)^t \right)
\\&-
(\tilde{d}-(4N-1)) \left( (xy)^N \sum_{t=-\infty}^\infty (\tilde{a}_{t,N-1,1}-\tilde{a}_{t-1,N-1,1}) \left(\frac{F}{xy}\right)^t \right)
\\=&
(xy)^N \sum_{t=-\infty}^\infty (2N-1)(\tilde{d}-(2N-1)) \tilde{a}_{t,N-1,0} \left(\frac{F}{xy}\right)^t
\\&-
(xy)^N \sum_{t=-\infty}^\infty (\tilde{d}-(4N-1)) (\tilde{a}_{t,N-1,1}-\tilde{a}_{t-1,N-1,1}) \left(\frac{F}{xy}\right)^t
\\=&
(xy)^N \sum_{t=-\infty}^\infty \left( (2N-1)(\tilde{d}-(2N-1)) \tilde{a}_{t,N-1,0} - (\tilde{d}-(4N-1)) (\tilde{a}_{t,N-1,1}-\tilde{a}_{t-1,N-1,1}) \right) \left(\frac{F}{xy}\right)^t
\\=&
(xy)^N \sum_{t=-\infty}^\infty \tilde{a}_{t,N,0} \left(\frac{F}{xy}\right)^t
\\=&
\tilde{A}_{2N}
\end{align*}
which proves that \[
\tilde{A}_{2N} = (2N-1)(\tilde{d}-(2N-1)) xy \tilde{A}_{2(N-1)} - (\tilde{d}-(4N-1)) z \tilde{A}_{2(N-1)+1}
\] for $N\geq1$. When we set $\tilde{d}$ to $d$, this proves that \[
A_{2N} = (2N-1)(d-(2N-1)) xy A_{2(N-1)} - (d-(4N-1)) z A_{2(N-1)+1}
\/.\]
If we write $n=2N$, then this proves the lemma.
\end{proof}
\begin{theorem}
\label{thm:detLnFormula}
Let $F = xy-z^2$.
Given $n\geq0$, set $n=2N+\nu$ with $N \geq 0$ and $\nu \in \{0,1\}$. Then we have the following: \[
\det(\mtrx{L}_n)
 = 
\det(\mtrx{L}_{2N+\nu})
 = 
A_{2N+\nu}
 = 
z^{\nu}
\sum_{t=0}^N 
(-1)^{t+\nu}
\left(
\prod_{h=1}^{N+t+\nu} 
(d-(2h-1))
\right)
\left(
\prod_{h=t+1}^{N} 
(d-2h)
\right)
\binom{N}{t}
(xy)^{N-t}
F^t
\/.\]
\end{theorem}
\begin{proof}
Both $\det(\mtrx{L}_n)$ and $A_n$ satisfy the same following homogeneous linear recurrence relation and initial values:
\begin{itemize}
    \item $u_0 = 1$,
    \item $u_1 = -(d-1)z$,
    \item and $u_n = (n-1)(d-(n-1)) xy u_{n-2} - (d-(2n-1)) z u_{n-1}$ for all $n \geq 2$.
\end{itemize}
We have proved this with the following lemmas: 
\begin{itemize}
    \item $\det(\mtrx{L}_0) = 1$ by Lemma \ref{lem:detLnRecursive} and $A_0 = 1$ by Lemma \ref{lem:AnRecursiveInitial}),
    \item $\det(\mtrx{L}_1) = -(d-1)z$ by Lemma \ref{lem:detLnRecursive} and $A_1 = -(d-1)z$ by Lemma \ref{lem:AnRecursiveInitial},
    \item and finally, $\det(\mtrx{L}_n) = (n-1)(d-(n-1)) xy \det(\mtrx{L}_{n-2}) - (d-(2n-1)) z \det(\mtrx{L}_{n-1})$ for all $n \geq 2$ by Lemma \ref{lem:detLnRecursive} and $A_n = (n-1)(d-(n-1)) xy A_{n-2} - (d-(2n-1)) z A_{n-1}$ for all $n \geq 2$ by Lemmas \ref{lem:AnRecursiveFormulaOdd} and  \ref{lem:AnRecursiveFormulaEven}.
\end{itemize}
Therefore, $\det(\mtrx{L}_n) = A_n$ for all $n\geq0$.
\end{proof}
We only care about the case when $d$ is even. If we consider the case when $d$ is odd we see that the determinant is trivial.
\begin{corollary}
If $d$ is odd, then $\det\mtrx{M} = 0$.
\end{corollary}
\begin{proof}
Set $d=2D+1$.
We defined $\mtrx{M}$ and $\mtrx{L}_n$ in Notation \ref{not:MandLn} in a way such that $\mtrx{M} = \mtrx{L}_d$. By setting $n=d$ in the formula from Theorem \ref{thm:detLnFormula}, we know that
\begin{align*}
\det(\mtrx{M})
=&
\det(\mtrx{L}_{d})
 = 
\det(\mtrx{L}_{2D+1})
\\=&
z^1
\sum_{t=0}^D 
(-1)^{t+1}
\prod_{h=1}^{D+t+1} 
(d-(2h-1))
\prod_{h=t+1}^D 
(d-2h)
\binom{D}{t}
(xy)^{D-t}
F^t
\/.\end{align*}
We see that $\prod_{h=1}^{D+t+1} (d-(2h-1)) = 0$ if $d-(2h-1) = 0$ for some $1\leq h\leq D+t+1$. This is the case because if $h=D+1$, then $d-(2h-1) = (2D+1)-(2(D+1)-1) = 0$ and $1\leq h\leq D+t+1$ for any $0\leq t\leq D$. Therefore,
\begin{align*}
\det(\mtrx{M})
=&
z^1
\sum_{t=0}^D 
(-1)^{t+1}
\prod_{h=1}^{D+t+1} 
(d-(2h-1))
\prod_{h=t+1}^D 
(d-2h)
\binom{D}{t}
(xy)^{D-t}
F^t
\\=&
z^1
\sum_{t=0}^D 
(-1)^{t+1}
(0)
\prod_{h=t+1}^D 
(d-2h)
\binom{D}{t}
(xy)^{D-t}
F^t
 = 
z^1
\sum_{t=0}^D 
0
 = 
0
\/.\qedhere\end{align*}
\end{proof}
We now assume that $d$ is even for the rest of this paper, and write $d=2D$.
\begin{corollary}
\label{cor:detMdd}
If $d=2D$ is even, then \[
\det\left(\mtrx{M}_{(\hat{d}),(\hat{d})}\right) = -(d-1)! z \sum_{t=0}^{D-1} \lambda_t (xy)^{(D-1)-t} F^t
\] as an element of $\ZZ[2^{-1}][x,y,z]$.
\end{corollary}
\begin{proof}
\setcounter{equation}{0}
Recall that $\lambda_t = 2^{-2t}\binom{2t}{t}$ from Notation \ref{not:lambda} is a dyadic rational, which is why we require $2$ to be invertible.

We defined $\mtrx{L}_n$ in Notation \ref{not:MandLn} so that it consists of the first $n$ rows and columns of $\mtrx{M}$ if $n\leq d$, so $\mtrx{M}_{(\hat{d}),(\hat{d})} = \mtrx{L}_{d-1}$ since $\mtrx{M}$ has size $d\times d$. If we set $n=d-1$ in the formula from Theorem \ref{thm:detLnFormula}, then
\begin{align*}
\det\left(\mtrx{M}_{(\hat{d}),(\hat{d})}\right)
=&
\det(\mtrx{L}_{d-1})
 = 
\det(\mtrx{L}_{2(D-1)+1})
\\=&
z^1
\sum_{t=0}^{D-1} 
(-1)^{t+1}
\prod_{h=1}^{D-1+t+1} 
(d-(2h-1))
\prod_{h=t+1}^{D-1} 
(d-2h)
\binom{D-1}{t}
(xy)^{(D-1)-t}
F^t
\\=&
-z
\sum_{t=0}^{D-1} 
(-1)^t
\prod_{h=1}^{D+t} 
(d-(2h-1))
\prod_{h=t+1}^{D-1} 
(d-2h)
\binom{D-1}{t}
(xy)^{(D-1)-t}
F^t
\/.\end{align*}
Further,
\begin{align}&\notag
(-1)^t
\prod_{h=1}^{D+t} 
(d-(2h-1))
\prod_{h=t+1}^{D-1} 
(d-2h)
\binom{D-1}{t}
\\=&\notag
(-1)^t
\left(
\prod_{h=1}^D 
(2D-(2h-1))
\prod_{h=D+1}^{D+t} 
(2D-(2h-1))
\right)
\prod_{h=t+1}^{D-1} 
(2D-2h)
\binom{D-1}{t}
\\=&\notag
(-1)^t
\prod_{h=1}^D 
(2h-1)
\prod_{h=1}^t 
(-(2h-1))
\prod_{h=1}^{(D-1)-t} 
(2h)
\binom{D-1}{t}
\\=&\notag
\prod_{h=1}^D 
(2h-1)
\prod_{h=1}^t 
(2h-1)
\left(
2^{(D-1)-t}
((D-1)-t)!
\right)
\binom{D-1}{t}
\\=&
\prod_{h=1}^D 
(2h-1)
\left(
2^t t! \lambda_t
\right)
2^{(D-1)-t}
((D-1)-t)!
\binom{D-1}{t}
\\=&\notag
2^{D-1}
\left(
t!
((D-1)-t)!
\binom{D-1}{t}
\right)
\prod_{h=1}^D 
(2h-1)
\lambda_t
\\=&\notag
2^{D-1}
(D-1)!
\prod_{h=1}^D 
(2h-1)
\lambda_t
\\=&\notag
\left(
2^{D-1}
(D-1)!
\prod_{h=1}^{D-1} 
(2h-1)
\right)
(2D-1)
\lambda_t
\\=&
(2(D-1))!
(2D-1)
\lambda_t
\\=&\notag
(d-1)! \lambda_t
\/,\end{align}
where we use $\prod_{h=1}^t (2h-1) = 2^t t! \lambda_t$ from Lemma \ref{lem:lambdaOdds} on line (1) and $2^t t! \prod_{h=1}^t (2h-1) = (2t)!$ also from Lemma \ref{lem:lambdaOdds} on line (2).

Therefore, we have
\begin{align*}
\det\left(\mtrx{M}_{(\hat{d}),(\hat{d})}\right)
=&
-z
\sum_{t=0}^{D-1} 
(-1)^t
\prod_{h=1}^{D+t} 
(d-(2h-1))
\prod_{h=t+1}^{D-1} 
(d-2h)
\binom{D-1}{t}
(xy)^{(D-1)-t}
F^t
\\=&
-z
\sum_{t=0}^{D-1} 
((d-1)! \lambda_t)
(xy)^{(D-1)-t}
F^t
\\=&
-(d-1)!
z
\sum_{t=0}^{D-1} 
\lambda_t
(xy)^{(D-1)-t}
F^t
\/.\qedhere\end{align*}
\end{proof}
\begin{proposition}
\label{prop:detM11}
If $d=2D$ is even, then \[
\det\left(\mtrx{M}_{(\hat{1}),(\hat{1})}\right) = (d-1)! z \sum_{t=0}^{D-1} \lambda_t (xy)^{(D-1)-t} F^t
\] as an element of $\ZZ[2^{-1}][x,y,z]$.
\end{proposition}
\begin{proof}
Recall that $\lambda_t = 2^{-2t}\binom{2t}{t}$ from Notation \ref{not:lambda} is a dyadic rational, which is why we require $2$ to be invertible.

In order to turn $\det\left(\mtrx{M}_{(\hat{d}),(\hat{d})}\right)$ into $\det\left(\mtrx{M}_{(\hat{1}),(\hat{1})}\right)$, we reverse the order of all the columns and rows of $\mtrx{M}$ (which does not affect the determinant since $\mtrx{M}$ has the same number of rows and columns), in order to obtain \[
\begin{bmatrix}
(d-1)z & -1y \\ (d-1)x & (d-3)z & -2y \\& (d-2)x & (d-5)z & -3y \\&& (d-3)x & \ddots & \ddots \\&&& \ddots & \ddots & -(d-3)y \\&&&& 3x & -(d-5)z & -(d-2)y \\&&&&& 2x & -(d-3)z & -(d-1)y \\&&&&&& 1x & -(d-1)z 
\end{bmatrix}
\/.\] Since \[
\mtrx{M}
 = 
\begin{bmatrix}
-(d-1)z & 1x \\ -(d-1)y & -(d-3)z & 2x \\& -(d-2)y & -(d-5)z & 3x \\&& -(d-3)y & \ddots & \ddots \\&&& \ddots & \ddots & (d-3)x \\&&&& -3y & (d-5)z & (d-2)x \\&&&&& -2y & (d-3)z & (d-1)x \\&&&&&& -1y & (d-1)z 
\end{bmatrix}
\/,\] we can also obtain this matrix by applying an invertible linear transformation to $\mtrx{M}$ sending $x \mapsto -y$, $y \mapsto -x$, and $z \mapsto -z$. Therefore, $\det\left(\mtrx{M}_{(\hat{1}),(\hat{1})}\right)$ is equal to $\det\left(\mtrx{M}_{(\hat{d}),(\hat{d})}\right)$ with this transformation applied. We know from Corollary \ref{cor:detMdd} that \[
\det\left(\mtrx{M}_{(\hat{d}),(\hat{d})}\right) = -(d-1)! z \sum_{t=0}^{D-1} \lambda_t (xy)^{(D-1)-t} F^t
\/;\] note that $xy$ and $F=xy-z^2$, and thus $\sum_{t=0}^{D-1} \lambda_t (xy)^{(D-1)-t} F^t$, aren't affected by the simultaneous transformations $x \mapsto -y$, $y \mapsto -x$, and $z \mapsto -z$.
Therefore \[
\det\left(\mtrx{M}_{\hat{1},\hat{1}}\right)
 = 
(d-1)! z \sum_{t=0}^{D-1} \lambda_t (xy)^{(D-1)-t} F^t
\/.\qedhere\]
\end{proof}
\begin{proposition}
\label{prop:detM}
If $d=2D$ is even, then $\det\mtrx{M} = \left( \prod_{h=1}^{D} (2h-1)^2 \right) F^D$.
\end{proposition}
\begin{proof}
We defined $\mtrx{M}$ and $\mtrx{L}_n$ in Notation \ref{not:MandLn} in a away such that $\mtrx{M} = \mtrx{L}_d$. By setting $n=d$ in the formula from Theorem \ref{thm:detLnFormula}, we get
\begin{align*}
\det(\mtrx{M})
=&
\det(\mtrx{L}_{d})
 = 
\det(\mtrx{L}_{2D+0})
\\=&
z^0
\sum_{t=0}^D 
(-1)^{t+0}
\prod_{h=1}^{D+t+0} 
(d-(2h-1))
\prod_{h=t+1}^D 
(d-2h)
\binom{D}{t}
(xy)^{D-t}
F^t
\\=&
\sum_{t=0}^D 
(-1)^t
\prod_{h=1}^{D+t} 
(d-(2h-1))
\prod_{h=t+1}^D 
(d-2h)
\binom{D}{t}
(xy)^{D-t}
F^t
\/.\end{align*}
If we focus on the $\prod_{h=t+1}^D (d-2h)$ factor of the coefficients, we see that $\prod_{h=t+1}^D (d-2h) = 1$ if $t = D$, but $\prod_{h=t+1}^D (d-2h) = 0$ if $t < D$ because $d-2h = 0$ when $h = D$. Therefore, \[
\det(\mtrx{M})
 = 
(-1)^D
\prod_{h=1}^{D+D} 
(d-(2h-1))
\binom{D}{D}
(xy)^{D-D}
F^D
 = 
(-1)^D
\prod_{h=1}^{2D} 
(d-(2h-1))
F^D
\/.
\] If we look at $\prod_{h=1}^{2D} (d-(2h-1))$, we can see that its factors split nicely between $h \leq D$ and $h \geq D+1$:
\begin{align*}&
\prod_{h=1}^{2D} 
(d-(2h-1))
\\=&
\prod_{h=1}^D 
(2D-(2h-1))
\prod_{h=D+1}^{2D} 
(2D-(2h-1))
\\=&
\prod_{h=1}^D 
(2h-1)
\prod_{h=1}^D 
(-(2h-1))
\\=&
(-1)^D
\prod_{h=1}^D 
(2h-1)^2
\/.
\end{align*}
Thus $\det(\mtrx{M}) = \prod_{h=1}^D 
(2h-1)^2 F^D$.
\end{proof}

\section{Main results}
\label{sec:MainResults}

We use the results discussed thus far to prove in this section that the polynomial $(xy-z^2)^D$ is link-$q$-compressed in $k[x,y,z]$ for all powers $q$ of $p$, the odd characteristic of the field $k$, as long as $p>2D-1$. We use this fact to show most choices of degree $2D$ homogeneous polynomials in $k[x,y,z]$ are link-$q$-compressed, and thus the conclusions listed in Subsection \ref{subsec:lqcFacts} hold for these choices of polynomials. This section lists and proves these statements.

To prove our results, we make use of the following lemma:

\begin{lemma}[cf. {\cite[Lemma 2.3]{KRV12}}]
\label{lem:phipsi}
Let $P$ be a commutative Noetherian ring, $\vec{x}^\top = \begin{bmatrix} x_1 & x_2 & x_3 \end{bmatrix}$ such that $x_1,x_2,x_3$ generate a perfect grade 3 ideal in $P$, $\mtrx{\varphi}$ a skew-symmetric matrix in $P$ of size $m \times m$ (where $m$ is even), $\mtrx{\psi}$ a $m \times 3$ matrix in $P$, and $u$ a unit in $P$. Define $f = u^{-1} \Pf\mtrx{\varphi}$ and $\mtrx{X} = \begin{bmatrix} 0 & x_3 & -x_2 \\ -x_3 & 0 & x_1 \\ x_2 & -x_1 & 0 \end{bmatrix}$.
Assume that the entries of $u\mtrx{X} - \mtrx{\psi}^\top \mtrx{\varphi}^\vee \mtrx{\psi}$ are in the ideal $(f)P$. Define $\mtrx{\Phi} = (u\mtrx{X} - \mtrx{\psi}^\top \mtrx{\varphi}^\vee \mtrx{\psi})/(uf)$ and define $\mtrx{\partial}_2 = \begin{bmatrix} \mtrx{\varphi} & \mtrx{\psi} \\ -\mtrx{\psi}^\top & \mtrx{\Phi} \end{bmatrix}$.

Then the ideal $I_1(\vec{x}^\top):f$ is generated by the maximal order Pfaffians of $\mtrx{\partial}_2$.

Also, \[
0
\xrightarrow{}
P^m
\xrightarrow{
\begin{bmatrix}
\mtrx{\varphi} \\ -u \mtrx{\psi}^\top
\end{bmatrix}
}
\bigoplus_{P^3}^{P^m}
\xrightarrow{
\begin{bmatrix}
u \mtrx{\psi}^\top \mtrx{\varphi}^\vee & u f \mtrx{I} \\ -\vec{b}^\top & -\vec{x}^\top
\end{bmatrix}
}
\bigoplus_{P}^{P^3}
\xrightarrow{
\begin{bmatrix}
\vec{x}^\top & u f
\end{bmatrix}
}
P
\]
is a free $P$-resolution of $P/(f,I_1(\vec{x}^\top))$, where $\vec{b}^\top = \begin{bmatrix} \Pf_1\mtrx{\partial}_2 & \cdots & \Pf_m\mtrx{\partial}_2 \end{bmatrix}$.

Lastly, if $R = P/(f)$, then
\[
\cdots
 \to 
R^m
 \xrightarrow{\mtrx{\varphi}} 
R^m
 \xrightarrow{\mtrx{\varphi}^\vee} 
R^m
 \xrightarrow{\mtrx{\varphi}} 
R^m
 \xrightarrow{\mtrx{\psi}^\top \mtrx{\varphi}^\vee} 
R^3
 \xrightarrow{\vec{x}^\top} 
R
\]
is a free $R$-resolution of $R/I_1(\vec{x}^\top)$.
\end{lemma}
\begin{remark}
We make only two significant changes from the original lemma. The first is that $\mtrx{\psi}$ is replaced with its transpose (the original paper had $\mtrx{\psi}$ as a $3 \times m$ matrix instead of $m \times 3$). The second is that every $f$ is replaced with $u f$ (the original paper had $\Pf\mtrx{\varphi} = f$). The statement as given here is equivalent to its original formulation.
\end{remark}

We establish the following in order to apply Lemma \ref{lem:phipsi}:

\begin{notation}
\label{not:phipsi}
Let $k$ be a field with $\Char k = p$, an odd prime, and let $P = k[x,y,z]$ with homogeneous maximal ideal $\maxI = (x,y,z)$.

Define the polynomial $f = F^D$ in $P$, where $F = xy-z^2$ and $D \geq 1$. Let $d = \deg f = (\deg F)D = 2D$. Also let $R=P/(f)$.

Let $q>1$ be a power of $p$. Since $p$ is odd, $q$ is odd as well; define $\pi=(q-1)/2$.

Now set $x_1=x^q$, $x_2=y^q$, $x_3=z^q$. This sets $\vec{x}^\top = \begin{bmatrix} x^q & y^q & z^q \end{bmatrix}$, $\mtrx{X} = \begin{bmatrix} 0 & z^q & -y^q \\ -z^q & 0 & x^q \\ y^q & -x^q & 0 \end{bmatrix}$, and $\maxI^{[q]} = (x^q,y^q,z^q) = I_1(\vec{x}^\top)$.

We have \[
z^q = z \sum_{t=0}^\pi \lambda_t (xy)^{\pi-t} F^t = z \left( (xy)^{\pi-(D-1)} g + f G \right)
\/,\] where the first equality is from Lemma \ref{lem:zq}, and the second equality is clear if we define $g := \sum_{t=0}^{D-1} \lambda_t (xy)^{(D-1)-t} F^t$ and $G := \sum_{t=D}^\pi \lambda_t (xy)^{\pi-t} F^{t-D}$. Here $\lambda_t$ is from Notation \ref{not:lambda}.

We define the $2d \times 3$ block matrix \[
\mtrx{\psi}
 = 
\begin{bmatrix} d \lambda_D y^{\pi-(D-1)} \vec{e}_1 & \vec{0} & d \lambda_D x^{\pi-(D-1)} \vec{e}_d \\ \vec{0} & -x^{\pi-(D-1)} \vec{e}_1 & y^{\pi-(D-1)} \vec{e}_d \end{bmatrix}
\/,\] where we use the length $d$ elementary column vectors $\vec{e}_1 = \begin{bmatrix} 1 \\ 0 \\ \vdots \\ 0 \end{bmatrix}$ and $\vec{e}_d = \begin{bmatrix} 0 \\ \vdots \\ 0 \\ 1 \end{bmatrix}$, and also define the $2d \times 2d$ block matrix of pure graded degree $1$  \[
\mtrx{\varphi}
 = 
\begin{bmatrix} \mtrx{0} & \mtrx{M} \\ -\mtrx{M}^\top & \mtrx{0} \end{bmatrix}
\/,\] where we use the $d \times d$ tridiagonal matrix of pure graded degree $1$ \[
\mtrx{M}
 = 
\begin{bmatrix}
-(d-1)z & 1x \\ -(d-1)y & -(d-3)z & 2x \\& -(d-2)y & -(d-5)z & 3x \\&& -(d-3)y & \ddots & \ddots \\&&& \ddots & \ddots & (d-3)x \\&&&& -3y & (d-5)z & (d-2)x \\&&&&& -2y & (d-3)z & (d-1)x \\&&&&&& -1y & (d-1)z 
\end{bmatrix}
\/.\] Recall that we defined this matrix in Notation \ref{not:MandLn}.
Define $\mtrx{\Phi} = z G \begin{bmatrix}0&1&0\\-1&0&0\\0&0&0\end{bmatrix}$; we prove in Lemma \ref{lem:PhiEquation} that $u \mtrx{X} - \mtrx{\psi}^\top \mtrx{\varphi}^\vee \mtrx{\psi} = u f \mtrx{\Phi}$, where $u = (-1)^D d! \lambda_D$. Also define \[
\mtrx{\partial}_2
 = 
\begin{bmatrix}
\mtrx{\varphi} & \mtrx{\psi} \\ -\mtrx{\psi}^\top & \mtrx{\Phi}
\end{bmatrix}
 = 
\begin{bmatrix}
\mtrx{0} & \mtrx{M} & d \lambda_D y^{\pi-(D-1)} \vec{e}_1 & \vec{0} & d \lambda_D x^{\pi-(D-1)} \vec{e}_d \\ -\mtrx{M}^\top & \mtrx{0} & \vec{0} & -x^{\pi-(D-1)} \vec{e}_1 & y^{\pi-(D-1)} \vec{e}_d \\ -d \lambda_D y^{\pi-(D-1)} \vec{e}_1^\top & \vec{0}^\top & 0 & z G & 0 \\ \vec{0}^\top & x^{\pi-(D-1)} \vec{e}_1^\top & -z G & 0 & 0 \\ -d \lambda_D x^{\pi-(D-1)} \vec{e}_d^\top & -y^{\pi-(D-1)} \vec{e}_d^\top & 0 & 0 & 0
\end{bmatrix}
\/.\]
\end{notation}

Our goal is to apply Lemma \ref{lem:phipsi} to the objects we've established. This means we must prove that:
\begin{itemize}
\item $u = (-1)^D d! \lambda_D$ is a unit,
\item $\Pf\mtrx{\varphi} = u f$, and
\item the entries of $u \mtrx{X} - \mtrx{\psi}^\top \mtrx{\varphi}^\vee \mtrx{\psi}$ are in the ideal $(f)P$, which we'll prove by showing that $u \mtrx{X} - \mtrx{\psi}^\top \mtrx{\varphi}^\vee \mtrx{\psi} = u f \mtrx{\Phi}$.
\end{itemize}
Once we apply Lemma \ref{lem:phipsi}, we will know the maximal order Pfaffians of $\mtrx{\partial}_2$ are generators of the ideal $I_1(\mtrx{x}):f = \maxI^{[q]}:f$. When we show these Pfaffians have degree greater than $s/2$ except for those in $\maxI^{[q]}$, this proves that $f$ is link-$q$-compressed. Recall from Lemma \ref{lem:invSystFrob} that $s = 3(q-1)-\deg f = 2(3\pi-D)$ is the socle dimension of $R/\maxI^{[q]}$.

\begin{lemma}
\label{lem:uIsUnit}
Assume $d=2D$ is even.
We have $u := (-1)^D d! \lambda_D = (-1)^D \left( \prod_{h=1}^D (2h-1) \right)^2$ of a field $k$ with odd characteristic $p>d-1$ is a unit.
\end{lemma}
\begin{proof}
We established in Lemma  \ref{lem:lambdaOdds} that $(2t)! \lambda_t = \left( \prod_{h=1}^t (2h-1) \right)^2$ for any $t \geq 0$, and so we have \[
u = (-1)^D (2D)! \lambda_D = (-1)^D \left( \prod_{h=1}^D (2h-1) \right)^2
\/.\] This means $u$ is an integer that is a product of odd numbers at most $2D-1=d-1$. Since $\Char k > d-1$, this means each factor of $u$ is invertible, and thus so is $u$.
\end{proof}
\begin{lemma}
\label{lem:Pfphi}
Assume $d=2D$ is even.
Then $\Pf\mtrx{\varphi} = u f$.
\end{lemma}
\begin{proof}
We define $u$, $f$, and $\mtrx{\varphi}$ as in Notation \ref{not:phipsi}.

To find $\Pf\mtrx{\varphi}$, we use the formula from Remark \ref{rmk:PfaffBlockSquare} to see \[
\Pf\mtrx{\varphi} = \Pf \begin{bmatrix} \mtrx{0} & \mtrx{M} \\ -\mtrx{M}^\top & \mtrx{0} \end{bmatrix} = (-1)^{d(d-1)/2} \det\mtrx{M} = (-1)^{D(2D-1)} \det\mtrx{M} = (-1)^D \det\mtrx{M}
\/,\] and then since $\det\mtrx{M} = \left( \prod_{h=1}^{D} (2h-1) \right)^2 F^D = \left( \prod_{h=1}^{D} (2h-1) \right)^2 f$ by Proposition \ref{prop:detM} and the previous paragraph, we have $\Pf\mtrx{\varphi} = (-1)^D \left( \prod_{h=1}^{D} (2h-1) \right)^2 f = u f$ using Lemma \ref{lem:uIsUnit}.
\end{proof}
\begin{lemma}
\label{lem:PhiEquation}
We have $\mtrx{\psi}^\top \mtrx{\varphi}^\vee \mtrx{\psi} + u f \mtrx{\Phi} = u \mtrx{X}$, which means that the entries of $u \mtrx{X} - \mtrx{\psi}^\top \mtrx{\varphi}^\vee \mtrx{\psi}$ are in the ideal $(f)P$.
\end{lemma}
\begin{proof}
We define $u$, $f$, $\mtrx{\varphi}$, $\mtrx{\psi}$, $\mtrx{\Phi}$, and $\mtrx{X}$ in Notation \ref{not:phipsi}.

We use Lemma \ref{lem:PfaffBlockAdj} to see that \[
\mtrx{\varphi}^\vee = \begin{bmatrix} \mtrx{0} & \mtrx{M} \\ -\mtrx{M}^\top & \mtrx{0} \end{bmatrix}^\vee = (-1)^{d(d-1)/2} \begin{bmatrix} \mtrx{0} & -\overline{\mtrx{M}}^\top \\ \overline{\mtrx{M}} & \mtrx{0} \end{bmatrix} = (-1)^D \begin{bmatrix} \mtrx{0} & -\overline{\mtrx{M}}^\top \\ \overline{\mtrx{M}} & \mtrx{0} \end{bmatrix}
\/,\] where the blocks are $d \times d$ matrices. From this we can calculate
\begin{align*}&
\mtrx{\psi}^\top \mtrx{\varphi}^\vee \mtrx{\psi}
\\=&
\begin{bmatrix} d \lambda_D y^{\pi-(D-1)} \vec{e}_1^\top & \vec{0}^\top \\ \vec{0}^\top & -x^{\pi-(D-1)} \vec{e}_1^\top \\ d \lambda_D x^{\pi-(D-1)} \vec{e}_d^\top & y^{\pi-(D-1)} \vec{e}_d^\top \end{bmatrix}
(-1)^D \begin{bmatrix} \mtrx{0} & -\overline{\mtrx{M}}^\top \\ \overline{\mtrx{M}} & \mtrx{0} \end{bmatrix}
\mtrx{\psi}
\\\\=&
(-1)^D
\begin{bmatrix} \vec{0}^\top & -d \lambda_D y^{\pi-(D-1)} \vec{e}_1^\top \overline{\mtrx{M}}^\top \\ -x^{\pi-(D-1)} \vec{e}_1^\top \overline{\mtrx{M}} & \vec{0}^\top \\ y^{\pi-(D-1)} \vec{e}_d^\top \overline{\mtrx{M}} & -d \lambda_D x^{\pi-(D-1)} \vec{e}_d^\top \overline{\mtrx{M}}^\top \end{bmatrix}
\mtrx{\psi}
\\\\=&
(-1)^D
\begin{bmatrix} \vec{0}^\top & -d \lambda_D y^{\pi-(D-1)} \vec{e}_1^\top \overline{\mtrx{M}}^\top \\ -x^{\pi-(D-1)} \vec{e}_1^\top \overline{\mtrx{M}} & \vec{0}^\top \\ y^{\pi-(D-1)} \vec{e}_d^\top \overline{\mtrx{M}} & -d \lambda_D x^{\pi-(D-1)} \vec{e}_d^\top \overline{\mtrx{M}}^\top \end{bmatrix}
\begin{bmatrix} d \lambda_D y^{\pi-(D-1)} \vec{e}_1 & \vec{0} & d \lambda_D x^{\pi-(D-1)} \vec{e}_d \\ \vec{0} & -x^{\pi-(D-1)} \vec{e}_1 & y^{\pi-(D-1)} \vec{e}_d \end{bmatrix}
\\\\=&
(-1)^D
\begin{bmatrix} 0 & d \lambda_D (xy)^{\pi-(D-1)} \vec{e}_1^\top \overline{\mtrx{M}}^\top \vec{e}_1 & -d \lambda_D (y^2)^{\pi-(D-1)} \vec{e}_1^\top \overline{\mtrx{M}}^\top \vec{e}_d \\ -d \lambda_D (xy)^{\pi-(D-1)} \vec{e}_1^\top \overline{\mtrx{M}} \vec{e}_1 & 0 & -d \lambda_D (x^2)^{\pi-(D-1)} \vec{e}_1^\top \overline{\mtrx{M}} \vec{e}_d \\ d \lambda_D (y^2)^{\pi-(D-1)} \vec{e}_d^\top \overline{\mtrx{M}} \vec{e}_1 & d \lambda_D (x^2)^{\pi-(D-1)} \vec{e}_d^\top \overline{\mtrx{M}}^\top \vec{e}_1 & 0 \end{bmatrix}
\\\\=&
(-1)^D d \lambda_D
\begin{bmatrix} 0 & (xy)^{\pi-(D-1)} \overline{\mtrx{M}}_{1,1} & -y^{q-(d-1)} \overline{\mtrx{M}}_{d,1} \\ -(xy)^{\pi-(D-1)} \overline{\mtrx{M}}_{1,1} & 0 & -x^{q-(d-1)} \overline{\mtrx{M}}_{1,d} \\ y^{q-(d-1)} \overline{\mtrx{M}}_{d,1} & x^{q-(d-1)} \overline{\mtrx{M}}_{1,d} & 0 \end{bmatrix}
\/.\end{align*}
Note that we used the fact that $2(\pi-(D-1)) = q-(d-1)$.

We write the following using Definition \ref{def:detAdj}:
\[
\overline{\mtrx{M}}_{1,1}
 = 
(-1)^{1+1} \det\left(\mtrx{M}_{(\hat{1}),(\hat{1})}\right)
 = 
\det\left(\mtrx{M}_{(\hat{1}),(\hat{1})}\right)
\/.\]
\[
\overline{\mtrx{M}}_{1,d}
 = 
(-1)^{d+1} \det\left(\mtrx{M}_{(\hat{d}),(\hat{1})}\right)
 = 
-\det\left(\mtrx{M}_{(\hat{d}),(\hat{1})}\right)
\/.\]
and
\[
\overline{\mtrx{M}}_{d,1}
 = 
(-1)^{1+d} \det\left(\mtrx{M}_{(\hat{1}),(\hat{d})}\right)
 = 
-\det\left(\mtrx{M}_{(\hat{1}),(\hat{d})}\right)
\/.\]
This lets us write
\begin{align*}&
\mtrx{\psi}^\top \mtrx{\varphi}^\vee \mtrx{\psi}
\\=&
(-1)^D d \lambda_D
\begin{bmatrix} 0 & (xy)^{\pi-(D-1)} \det\left(\mtrx{M}_{(\hat{1}),(\hat{1})}\right) & y^{q-(d-1)} \det\left(\mtrx{M}_{(\hat{1}),(\hat{d})}\right) \\ -(xy)^{\pi-(D-1)} \det\left(\mtrx{M}_{(\hat{1}),(\hat{1})}\right) & 0 & x^{q-(d-1)} \det\left(\mtrx{M}_{(\hat{d}),(\hat{1})}\right) \\ -y^{q-(d-1)} \det\left(\mtrx{M}_{(\hat{1}),(\hat{d})}\right) & -x^{q-(d-1)} \det\left(\mtrx{M}_{(\hat{d}),(\hat{1})}\right) & 0 \end{bmatrix}
\/.\end{align*}

Using Proposition \ref{prop:detMd1and1d}, we have \[
\det\left(\mtrx{M}_{(\hat{d}),(\hat{1})}\right)
 = 
(d-1)! x^{d-1}
\] and \[
\det\left(\mtrx{M}_{(\hat{1}),(\hat{d})}\right)
 = 
(-1)^{d-1} (d-1)! y^{d-1}
 = 
-(d-1)! y^{d-1}
\/,\] as well as \[
\det\left(\mtrx{M}_{(\hat{1}),(\hat{1})}\right) = (d-1)! z \left( \sum_{t=0}^{D-1} \lambda_t (xy)^{(D-1)-t} F^t \right) = (d-1)! z g
\] by Proposition \ref{prop:detM11}. Recall that $g = \sum_{t=0}^{D-1} \lambda_t (xy)^{(D-1)-t} F^t$ from Notation \ref{not:phipsi}.
We put these together to see that
\begin{align*}&
\mtrx{\psi}^\top \mtrx{\varphi}^\vee \mtrx{\psi}
\\=&
(-1)^D d \lambda_D
\begin{bmatrix} 0 & (xy)^{\pi-(D-1)} \left((d-1)! z g\right) & y^{q-(d-1)} \left(-(d-1)! y^{d-1}\right) \\ -(xy)^{\pi-(D-1)} \left((d-1)! z g\right) & 0 & x^{q-(d-1)} \left((d-1)! x^{d-1}\right) \\ -y^{q-(d-1)} \left(-(d-1)! y^{d-1}\right) & -x^{q-(d-1)} \left((d-1)! x^{d-1}\right) & 0 \end{bmatrix}
\\=&
(-1)^D d! \lambda_D
\begin{bmatrix} 0 & (xy)^{\pi-(D-1)} z g & -y^q \\ -(xy)^{\pi-(D-1)} z g & 0 & x^q \\ y^q & -x^q & 0 \end{bmatrix}
\\=&
u
\begin{bmatrix} 0 & (xy)^{\pi-(D-1)} z g & -y^q \\ -(xy)^{\pi-(D-1)} z g & 0 & x^q \\ y^q & -x^q & 0 \end{bmatrix}
\/,\end{align*}
and so we have
\[
\mtrx{\psi}^\top \mtrx{\varphi}^\vee \mtrx{\psi}
 = 
u
\begin{bmatrix} 0 & (xy)^{\pi-(D-1)} z g & -y^q \\ -(xy)^{\pi-(D-1)} z g & 0 & x^q \\ y^q & -x^q & 0 \end{bmatrix}
\/.\]

Recall that that $z^q = z \left( (xy)^{\pi-(D-1)} g + f G \right)$ from Notation \ref{not:phipsi}.
Thus, we have
\begin{align*}
u \mtrx{X}
=&
u
\begin{bmatrix} 0 & z^q & -y^q \\ -z^q & 0 & x^q \\ y^q & -x^q & 0 \end{bmatrix}
\\=&
u
\begin{bmatrix} 0 & (xy)^{\pi-(D-1)} z g & -y^q \\ -(xy)^{\pi-(D-1)} z g & 0 & x^q \\ y^q & -x^q & 0 \end{bmatrix}
+
u
\begin{bmatrix} 0 & f z G & 0 \\ -f z G & 0 & 0 \\ 0 & 0 & 0 \end{bmatrix}
\\=&
\mtrx{\psi}^\top \mtrx{\varphi}^\vee \mtrx{\psi}
+
u
f z G
\begin{bmatrix}0&
1
&0\\
-1
&0&0\\0&0&0\end{bmatrix}
\\=&
\mtrx{\psi}^\top \mtrx{\varphi}^\vee \mtrx{\psi}
+
u f \mtrx{\Phi}
\/.
\end{align*}

Therefore, $u \mtrx{X} -\mtrx{\psi}^\top \mtrx{\varphi}^\vee \mtrx{\psi} = u f \mtrx{\Phi}$, which has entries in $(f)P$.
\end{proof}
\begin{theorem}
\label{thm:GensAndRes}
Let $D\geq1$ and let $k$ be a field with odd characteristic $p > 2D-1$. Let $q>1$ be a power of $p$.
Set $f = \left( xy-z^2 \right)^D$ as an element of $P=k[x,y,z]$.

Then the ideal $I_1(\vec{x}^\top):f = \maxI^{[q]}:f$ is generated by the maximal order Pfaffians of $\mtrx{\partial}_2$.

Also, \[
0
\xrightarrow{}
P^{2d}
\xrightarrow{
\begin{bmatrix}
\mtrx{\varphi} \\ -u \mtrx{\psi}^\top
\end{bmatrix}
}
\bigoplus_{P^3}^{P^{2d}}
\xrightarrow{
\begin{bmatrix}
u \mtrx{\psi}^\top \mtrx{\varphi}^\vee & u f \mtrx{I} \\ -\vec{b}^\top & -\vec{x}^\top
\end{bmatrix}
}
\bigoplus_{P}^{P^3}
\xrightarrow{
\begin{bmatrix}
\vec{x}^\top & u f
\end{bmatrix}
}
P
\]
is a free $P$-resolution of $P/(f,\maxI^{[q]})$, where $\vec{b}^\top = \begin{bmatrix} \Pf_1\mtrx{\partial}_2 & \cdots & \Pf_{2d}\mtrx{\partial}_2 \end{bmatrix}$.

Lastly, if $R = P/(f) = k[x,y,z]/\left( \left( xy-z^2 \right)^D \right)$, then
\[
\cdots
 \to 
R^{2d}
 \xrightarrow{\mtrx{\varphi}} 
R^{2d}
 \xrightarrow{\mtrx{\varphi}^\vee} 
R^{2d}
 \xrightarrow{\mtrx{\varphi}} 
R^{2d}
 \xrightarrow{\mtrx{\psi}^\top \mtrx{\varphi}^\vee} 
R^3
 \xrightarrow{\vec{x}^\top} 
R
\]
is a free $R$-resolution of $R/\maxI^{[q]}$.
Further, these maps are of pure graded degrees \[
\deg(\vec{x}^\top) = q
, \; 
\deg(\mtrx{\psi}) = \frac{1}{2} (q - (d - 1))
, \; 
\deg(\mtrx{\varphi}) = 1
, \; 
\deg(\mtrx{\varphi}^\vee) = d - 1
\/.\]
\end{theorem}
\begin{proof}
Lemmas \ref{lem:uIsUnit}, \ref{lem:Pfphi} and \ref{lem:PhiEquation} prove that the variables from Notation \ref{not:phipsi} satisfy Lemma \ref{lem:phipsi}.
\end{proof}
With this, we have shown that $R/\maxI^{[q]}$ with $R = k[x,y,z]/\left( \left( xy-z^2 \right)^D \right)$ is an example of the phenomenon detailed in \cite{KU09} where the tail of the $R/\maxI^{[q]}$ is independent of $q$:
\begin{corollary}
\label{cor:FDResRailIndep}
The tail end of the $R$-free resolutions of $R/\maxI^{[q]}$ are independent of $q$, since $\mtrx{\varphi}$ and $\mtrx{\varphi}^\vee$ do not depend on $q$.
\end{corollary}

In order to show that $\left( xy-z^2 \right)^D$ is link-$q$-compressed using Lemma \ref{lem:lqcDef}, we prove that the maximal order Pfaffians of $\mtrx{\partial}_2$, which generate $\maxI^{[q]}:f$ as we established above, are either in the ideal $\maxI^{[q]}$ or have degree greater than $\frac{s}{2} = \frac{3(q-1)-d}{2} = \frac{3(2\pi)-(2D)}{2} = 3\pi-D$ (in fact, they are all of degree $\frac{s}{2}+1 = 3\pi-(D-1)$).
\begin{proposition}
\label{prop:Partial2PfaffsExceptLast3}
The Pfaffians $\Pf_{\ell}(\mtrx{\partial}_2)$ for $1\leq\ell\leq 2d$ are all homogeneous of degree $\frac{s}{2}+1 = 3\pi-(D-1)$.
\end{proposition}
\begin{proof}
We discuss the following Pfaffians of $\mtrx{\partial}_2$ in pairs as follows:
\begin{itemize}
\item $\Pf_{1} \mtrx{\partial}_2$ and $\Pf_{d+1} \mtrx{\partial}_2$, 
\item $\Pf_{d} \mtrx{\partial}_2$ and $\Pf_{2d} \mtrx{\partial}_2$, and 
\item $\Pf_{\ell} \mtrx{\partial}_2$ and $\Pf_{d+\ell} \mtrx{\partial}_2$ for $1<\ell<d$.
\end{itemize}

In order to prove that $f$ is link-$q$-compressed using Lemma \ref{lem:lqcDef}, we need only show that the remaining Pfaffians (which generate $\maxI^{[q]}:f$ by Lemma \ref{lem:phipsi}) all have degree $s/2+1 = 3\pi-(D-1)$. This is the degree they must have anyway if $f$ is link-$q$-compressed by Corollary \ref{cor:lqcDegOfGens}.
We recall that $\mtrx{\partial}_2$ is a $(2d+3)\times(2d+3)$ block matrix with the structure \[
\mtrx{\partial}_2
 = 
\left[
\begin{tabular}{cc|ccc}
$\mtrx{0}$                            & $\mtrx{M}$                & $d \lambda_D y^{\pi-(D-1)} \vec{e}_1$ & $\vec{0}$                  & $d \lambda_D x^{\pi-(D-1)} \vec{e}_d$ \\
$-\mtrx{M}^\top$                      & $\mtrx{0}$                & $\vec{0}$                             & $-x^{\pi-(D-1)} \vec{e}_1$ & $y^{\pi-(D-1)} \vec{e}_d$             \\ \hline
$-d \lambda_D y^{\pi-(D-1)} \vec{e}_1^\top$ & $\vec{0}^\top$                  & $0$                                   & $z G$                      & $0$                                   \\
$\vec{0}^\top$                              & $x^{\pi-(D-1)} \vec{e}_1^\top$  & $-z G$                                & $0$                        & $0$                                   \\
$-d \lambda_D x^{\pi-(D-1)} \vec{e}_d^\top$ & $-y^{\pi-(D-1)} \vec{e}_d^\top$ & $0$                                   & $0$                        & $0$                                  
\end{tabular}
\right]
\] from Notation \ref{not:phipsi}.

Also note the size of each component of $\mtrx{\partial}_2$:
\begin{multicols}{2}
\begin{itemize}
\item $\mtrx{M}$ and $\mtrx{0}$ are $d\times d$ matrices, and
\item $\vec{e}_1$, $\vec{e}_d$, and $\vec{0}$ are length $d$ column vectors.
\end{itemize}
\end{multicols}
The remaining components of $\mtrx{\partial}_2$ are scalars. Thus, we have
\begin{itemize}
\item $\begin{bmatrix} \mtrx{0} \\ -\mtrx{M}^\top \\ -d \lambda_D y^{\pi-(D-1)} \vec{e}_1^\top \\ \vec{0}^\top \\ -d \lambda_D x^{\pi-(D-1)} \vec{e}_d^\top \end{bmatrix}$ is columns $1$ through $d$ of $\mtrx{\partial}_2$, 
\item $\begin{bmatrix} \mtrx{M} \\ \mtrx{0} \\ \vec{0}^\top \\ x^{\pi-(D-1)} \vec{e}_1^\top \\ -y^{\pi-(D-1)} \vec{e}_d^\top \end{bmatrix}$ is columns $d+1$ through $2d$ of $\mtrx{\partial}_2$, 
\item $\begin{bmatrix} d \lambda_D y^{\pi-(D-1)} \vec{e}_1 \\ \vec{0} \\ 0 \\ -z G \\ 0 \end{bmatrix}$ is column $2d+1$ of $\mtrx{\partial}_2$, 
\item $\begin{bmatrix} \vec{0} \\ -x^{\pi-(D-1)} \vec{e}_1 \\ z G \\ 0 \\ 0 \end{bmatrix}$ is column $2d+2$ of $\mtrx{\partial}_2$, and 
\item $\begin{bmatrix} d \lambda_D x^{\pi-(D-1)} \vec{e}_d \\ y^{\pi-(D-1)} \vec{e}_d \\ 0 \\ 0 \\ 0 \end{bmatrix}$ is column $2d+3$ of $\mtrx{\partial}_2$.
\end{itemize}

% ell=d+1
\begin{landscape}
We start with the calculation of the degree of $\Pf_{1} \mtrx{\partial}_2$. Using Definition \ref{def:PfaffEll}, we write
\begin{align*}&
\Pf_{1} \mtrx{\partial}_2
\\=&
\Pf_{1}
\begin{bmatrix}
\mtrx{0} & \mtrx{M} & d \lambda_D y^{\pi-(D-1)} \vec{e}_1 & \vec{0} & d \lambda_D x^{\pi-(D-1)} \vec{e}_d \\ -\mtrx{M}^\top & \mtrx{0} & \vec{0} & -x^{\pi-(D-1)} \vec{e}_1 & y^{\pi-(D-1)} \vec{e}_d \\ -d \lambda_D y^{\pi-(D-1)} \vec{e}_1^\top & \vec{0}^\top & 0 & z G & 0 \\ \vec{0}^\top & x^{\pi-(D-1)} \vec{e}_1^\top & -z G & 0 & 0 \\ -d \lambda_D x^{\pi-(D-1)} \vec{e}_d^\top & -y^{\pi-(D-1)} \vec{e}_d^\top & 0 & 0 & 0
\end{bmatrix}
\\\\=&
(-1)^{(1)+1}
\\&
\Pf
\begin{bmatrix}
\mtrx{0}_{(\hat{1}),(\hat{1})} & \mtrx{M}_{(\hat{1}),(-)} & \left(d \lambda_D y^{\pi-(D-1)} \vec{e}_1\right)_{(\hat{1}),(-)} & \vec{0}_{(\hat{1}),(-)} & \left(d \lambda_D x^{\pi-(D-1)} \vec{e}_d\right)_{(\hat{1}),(-)} \\ \left(-\mtrx{M}^\top\right)_{(-),(\hat{1})} & \mtrx{0} & \vec{0} & -x^{\pi-(D-1)} \vec{e}_1 & y^{\pi-(D-1)} \vec{e}_d \\ \left(-d \lambda_D y^{\pi-(D-1)} \vec{e}_1^\top\right)_{(-),(\hat{1})} & \vec{0}^\top & 0 & z G & 0 \\ \left(\vec{0}^\top\right)_{(-),(\hat{1})} & x^{\pi-(D-1)} \vec{e}_1^\top & -z G & 0 & 0 \\ \left(-d \lambda_D x^{\pi-(D-1)} \vec{e}_d^\top\right)_{(-),(\hat{1})} & -y^{\pi-(D-1)} \vec{e}_d^\top & 0 & 0 & 0
\end{bmatrix}
\\\\=&
\Pf
\begin{bmatrix}
\mtrx{0}_{(\hat{1}),(\hat{1})} & \mtrx{M}_{(\hat{1}),(-)} & \vec{0}_{\hat{1}} & \vec{0}_{\hat{1}} & d \lambda_D x^{\pi-(D-1)} \left(\vec{e}_d\right)_{\hat{1}} \\ -\left(\mtrx{M}_{(\hat{1}),(-)}\right)^\top & \mtrx{0} & \vec{0} & -x^{\pi-(D-1)} \vec{e}_1 & y^{\pi-(D-1)} \vec{e}_d \\ \vec{0}_{\hat{1}}^\top & \vec{0}^\top & 0 & z G & 0 \\ \vec{0}_{\hat{1}}^\top & x^{\pi-(D-1)} \vec{e}_1^\top & -z G & 0 & 0 \\ -d \lambda_D x^{\pi-(D-1)} \left(\vec{e}_d\right)_{\hat{1}}^\top & -y^{\pi-(D-1)} \vec{e}_d^\top & 0 & 0 & 0
\end{bmatrix}
\/.\end{align*}
Note that we replaced every instance of $\left(d \lambda_D y^{\pi-(D-1)} \vec{e}_1\right)_{(\hat{1}),(-)}$ with $\vec{0}_{\hat{1}}$ in the last line because the first entry in $d \lambda_D y^{\pi-(D-1)} \vec{e}_1$ is the only nonzero entry.
\end{landscape}

To make the calculations easier to read, all cofactor expansions are performed over fixed columns (fixed $j$ values).

Using the cofactor expansion formula in Definition \ref{def:Pfaffs}, we further write
\begin{align*}
\Pf_{1} \mtrx{\partial}_2
=&
\Pf
\begin{bmatrix}
\mtrx{0}_{(\hat{1}),(\hat{1})} & \mtrx{M}_{(\hat{1}),(-)} & \vec{0}_{\hat{1}} & \vec{0}_{\hat{1}} & d \lambda_D x^{\pi-(D-1)} (\vec{e}_d)_{\hat{1}} \\ -\left(\mtrx{M}_{(\hat{1}),(-)}\right)^\top & \mtrx{0} & \vec{0} & -x^{\pi-(D-1)} \vec{e}_1 & y^{\pi-(D-1)} \vec{e}_d \\ \vec{0}_{\hat{1}}^\top & \vec{0}^\top & 0 & z G & 0 \\ \vec{0}_{\hat{1}}^\top & x^{\pi-(D-1)} \vec{e}_1^\top & -z G & 0 & 0 \\ -d \lambda_D x^{\pi-(D-1)} (\vec{e}_d)_{\hat{1}}^\top & -y^{\pi-(D-1)} \vec{e}_d^\top & 0 & 0 & 0
\end{bmatrix}
\\\\=&
(-1)^{(2d+1)+(2d)+H((2d)-(2d+1))}
(-z G)
\\&
\Pf
\begin{bmatrix}
\mtrx{0}_{(\hat{1}),(\hat{1})} & \mtrx{M}_{(\hat{1}),(-)} & d \lambda_D x^{\pi-(D-1)} (\vec{e}_d)_{\hat{1}} \\ -\left(\mtrx{M}_{(\hat{1}),(-)}\right)^\top & \mtrx{0} & y^{\pi-(D-1)} \vec{e}_d \\ -d \lambda_D x^{\pi-(D-1)} (\vec{e}_d)_{\hat{1}}^\top & -y^{\pi-(D-1)} \vec{e}_d^\top & 0
\end{bmatrix}
\\\\=&
(z G)
(-1)^{(d-1)+(2d)+H((2d)-(d-1))}
(d \lambda_D x^{\pi-(D-1)})
\Pf
\begin{bmatrix}
\mtrx{0}_{(\hat{1}\hat{d}),(\hat{1}\hat{d})} & \mtrx{M}_{(\hat{1}\hat{d}),(-)} \\ -\left(\mtrx{M}_{(\hat{1}\hat{d}),(-)}\right)^\top & \mtrx{0}
\end{bmatrix}
\\&
+
(z G)
(-1)^{(2d-1)+(2d)+H((2d)-(2d-1))}
(y^{\pi-(D-1)})
\Pf
\begin{bmatrix}
\mtrx{0}_{(\hat{1}),(\hat{1})} & \mtrx{M}_{(\hat{1}),(\hat{d})} \\ -\left(\mtrx{M}_{(\hat{1}),(\hat{d})}\right)^\top & \mtrx{0}_{(\hat{d}),(\hat{d})}
\end{bmatrix}
\\\\=&\tag{i}
(z G)
(-1)^{(d-1)+(2d)+H((2d)-(d-1))}
(d \lambda_D x^{\pi-(D-1)})
(0)
\\&
+
(z G)
(-1)^{(2d-1)+(2d)+H((2d)-(2d-1))}
(y^{\pi-(D-1)})
\left(
(-1)^{(d-1)(d-2)/2}
\det\left(\mtrx{M}_{(\hat{1}),(\hat{d})}\right)
\right)
\\\\=&\tag{ii}
(z G)
(y^{\pi-(D-1)})
(-1)^{(2D-1)(2D-2)/2}
\left(
(-1)^{d-1} (d-1)! y^{d-1}
\right)
\\\\=&
(z G)
(y^{\pi-(D-1)})
(-1)^{D-1}
(-1)^{2D-1}
(d-1)! y^{2D-1}
\\\\=&
(-1)^D (d-1)! y^{\pi+D} z G
\/,\end{align*}
where line (i) is true by Lemma \ref{lem:PfaffBlockNonSquare} and Remark \ref{rmk:PfaffBlockSquare}, and line (ii) is true by Corollary \ref{prop:detMd1and1d}.

Recall that $G = \sum_{t=D}^\pi \lambda_t (xy)^{\pi-t} F^{t-D}$. Note that $(xy)^{\pi-t} F^{t-D}$ has degree $2(\pi-t)+2(t-D) = 2(\pi-D)$ for all $D \leq t \leq \pi$, so $G \in P_{2(\pi-D)}$.
Thus we have \[
\Pf_{1} \mtrx{\partial}_2
 = 
(-1)^D (d-1)! y^{\pi+D} z G
 \in 
P_{(\pi+D)+1+2(\pi-D)} = P_{3\pi-(D-1)} = P_{s/2+1}
\/,\] and so we see that $\Pf_{1} \mtrx{\partial}_2$ is either $0$ or degree $s/2+1$ as we desired.

We now show that this is true for the remaining Pfaffians.

% ell=d+1
We calculate the degree of $\Pf_{d+1} \mtrx{\partial}_2$ by a similar method to our previous calculation. Using Definition \ref{def:PfaffEll} and the cofactor expansion formula in Definition \ref{def:Pfaffs}, we write
\begin{align*}&
\Pf_{d+1} \mtrx{\partial}_2
\\=&
\Pf_{d+1}
\begin{bmatrix}
\mtrx{0}                            & \mtrx{M}                & d \lambda_D y^{\pi-(D-1)} \vec{e}_1 & \vec{0}                  & d \lambda_D x^{\pi-(D-1)} \vec{e}_d \\
-\mtrx{M}^\top                      & \mtrx{0}                & \vec{0}                             & -x^{\pi-(D-1)} \vec{e}_1 & y^{\pi-(D-1)} \vec{e}_d             \\
-d \lambda_D y^{\pi-(D-1)} \vec{e}_1^\top & \vec{0}^\top                  & 0                                   & z G                      & 0                                   \\
\vec{0}^\top                              & x^{\pi-(D-1)} \vec{e}_1^\top  & -z G                                & 0                        & 0                                   \\
-d \lambda_D x^{\pi-(D-1)} \vec{e}_d^\top & -y^{\pi-(D-1)} \vec{e}_d^\top & 0                                   & 0                        & 0                                  
\end{bmatrix}
\\\\=&
(-1)^{(d+1)+1}
\Pf
\begin{bmatrix}
\mtrx{0} & \mtrx{M}_{(-),(\hat{1})} & d \lambda_D y^{\pi-(D-1)} \vec{e}_1 & \vec{0} & d \lambda_D x^{\pi-(D-1)} \vec{e}_d \\ -\left(\mtrx{M}_{(-),(\hat{1})}\right)^\top & \mtrx{0}_{(\hat{1}),(\hat{1})} & \vec{0}_{\hat{1}} & \vec{0}_{\hat{1}} & y^{\pi-(D-1)} (\vec{e}_d)_{\hat{1}} \\ -d \lambda_D y^{\pi-(D-1)} \vec{e}_1^\top & \vec{0}_{\hat{1}}^\top & 0 & z G & 0 \\ \vec{0}^\top & \vec{0}_{\hat{1}}^\top & -z G & 0 & 0 \\ -d \lambda_D x^{\pi-(D-1)} \vec{e}_d^\top & -y^{\pi-(D-1)} (\vec{e}_d)_{\hat{1}}^\top & 0 & 0 & 0
\end{bmatrix}
\\\\=&
(-1)^{(2d)+(2d+1)+H((2d+1)-(2d))}
(z G)
\\&
\Pf
\begin{bmatrix}
\mtrx{0} & \mtrx{M}_{(-),(\hat{1})} & d \lambda_D x^{\pi-(D-1)} \vec{e}_d \\ -\left(\mtrx{M}_{(-),(\hat{1})}\right)^\top & \mtrx{0}_{(\hat{1}),(\hat{1})} & y^{\pi-(D-1)} (\vec{e}_d)_{\hat{1}} \\ -d \lambda_D x^{\pi-(D-1)} \vec{e}_d^\top & -y^{\pi-(D-1)} (\vec{e}_d)_{\hat{1}}^\top & 0
\end{bmatrix}
\\\\=&
(z G)
(-1)^{(d)+(2d)+H((2d)-(d))}
(d \lambda_D x^{\pi-(D-1)})
\Pf
\begin{bmatrix}
\mtrx{0}_{(\hat{d}),(\hat{d})} & \mtrx{M}_{(\hat{d}),(\hat{1})} \\ -\left(\mtrx{M}_{(\hat{d}),(\hat{1})}\right)^\top & \mtrx{0}_{(\hat{1}),(\hat{1})}
\end{bmatrix}
\\&
+
(z G)
(-1)^{(2d-1)+(2d)+H((2d)-(2d-1))}
(y^{\pi-(D-1)})
\Pf
\begin{bmatrix}
\mtrx{0} & \mtrx{M}_{(-),(\hat{1}\hat{d})} \\ -\left(\mtrx{M}_{(-),(\hat{1}\hat{d})}\right)^\top & \mtrx{0}_{(\hat{1}\hat{d}),(\hat{1}\hat{d})}
\end{bmatrix}
\\\\=&\tag{iii}
(z G)
(-1)^{(d)+(2d)+H((2d)-(d))}
(d \lambda_D x^{\pi-(D-1)})
\left(
(-1)^{D-1}
\det\left(\mtrx{M}_{(\hat{d}),(\hat{1})}\right)
\right)
\\&
+
(z G)
(-1)^{(2d-1)+(2d)+H((2d)-(2d-1))}
(y^{\pi-(D-1)})
(0)
\\\\=&\tag{iv}
-
(z G)
(d \lambda_D x^{\pi-(D-1)})
\left(
(-1)^{D-1}
\left(
(d-1)! x^{2D-1}
\right)
\right)
\\\\=&
(-1)^D d! \lambda_D x^{\pi+D} z G
\/,\end{align*}
where line (iii) is true by Lemma \ref{lem:PfaffBlockNonSquare} and Remark \ref{rmk:PfaffBlockSquare}, and line (iv) is true by Corollary \ref{prop:detMd1and1d}.
Thus we have \[
\Pf_{d+1} \mtrx{\partial}_2
 = 
(-1)^D d! \lambda_D x^{\pi+D} z G
 \in 
P_{(\pi+D)+1+2(\pi-D)} = P_{s/2+1}
\/,\] and so we see that $\Pf_{d+1} \mtrx{\partial}_2$ is either $0$ or degree $s/2+1$ as we desired.

% ell=d
We now calculate the degrees of $\Pf_{d} \mtrx{\partial}_2$ and $\Pf_{2d} \mtrx{\partial}_2$ together because they have a common factor. Using Definition \ref{def:PfaffEll} and the cofactor expansion formula in Definition \ref{def:Pfaffs}, we write
\begin{align*}&
\Pf_{d} \mtrx{\partial}_2
\\=&
\Pf_{d}
\begin{bmatrix}
\mtrx{0} & \mtrx{M} & d \lambda_D y^{\pi-(D-1)} \vec{e}_1 & \vec{0} & d \lambda_D x^{\pi-(D-1)} \vec{e}_d \\ -\mtrx{M}^\top & \mtrx{0} & \vec{0} & -x^{\pi-(D-1)} \vec{e}_1 & y^{\pi-(D-1)} \vec{e}_d \\ -d \lambda_D y^{\pi-(D-1)} \vec{e}_1^\top & \vec{0}^\top & 0 & z G & 0 \\ \vec{0}^\top & x^{\pi-(D-1)} \vec{e}_1^\top & -z G & 0 & 0 \\ -d \lambda_D x^{\pi-(D-1)} \vec{e}_d^\top & -y^{\pi-(D-1)} \vec{e}_d^\top & 0 & 0 & 0
\end{bmatrix}
\\\\=&
(-1)^{(d)+1}
\\&
\Pf
\begin{bmatrix}
\mtrx{0}_{(\hat{d}),(\hat{d})} & \mtrx{M}_{(\hat{d}),(-)} & d \lambda_D y^{\pi-(D-1)} (\vec{e}_1)_{\hat{d}} & \vec{0}_{\hat{d}} & \vec{0}_{\hat{d}} \\ -\left(\mtrx{M}_{(\hat{d}),(-)}\right)^\top & \mtrx{0} & \vec{0} & -x^{\pi-(D-1)} \vec{e}_1 & y^{\pi-(D-1)} \vec{e}_d \\ -d \lambda_D y^{\pi-(D-1)} (\vec{e}_1)_{\hat{d}}^\top & \vec{0}^\top & 0 & z G & 0 \\ \vec{0}_{\hat{d}}^\top & x^{\pi-(D-1)} \vec{e}_1^\top & -z G & 0 & 0 \\ \vec{0}_{\hat{d}}^\top & -y^{\pi-(D-1)} \vec{e}_d^\top & 0 & 0 & 0
\end{bmatrix}
\\\\=&
-
(-1)^{(2d-1)+(2d+2)+H((2d+2)-(2d-1))}
y^{\pi-(D-1)}
\\&
\Pf
\begin{bmatrix}
\mtrx{0}_{(\hat{d}),(\hat{d})} & \mtrx{M}_{(\hat{d}),(\hat{d})} & d \lambda_D y^{\pi-(D-1)} (\vec{e}_1)_{\hat{d}} & \vec{0}_{\hat{d}} \\ -\left(\mtrx{M}_{(\hat{d}),(\hat{d})}\right)^\top & \mtrx{0}_{(\hat{d}),(\hat{d})} & \vec{0}_{\hat{d}} & -x^{\pi-(D-1)} (\vec{e}_1)_{\hat{d}} \\ -d \lambda_D y^{\pi-(D-1)} (\vec{e}_1)_{\hat{d}}^\top & \vec{0}_{\hat{d}}^\top & 0 & z G \\ \vec{0}_{\hat{d}}^\top & x^{\pi-(D-1)} (\vec{e}_1)_{\hat{d}}^\top & -z G & 0
\end{bmatrix}
\/.\end{align*}
% ell=2d
We also have
\begin{align*}&
\Pf_{2d} \mtrx{\partial}_2
\\=&
\Pf_{2d}
\begin{bmatrix}
\mtrx{0} & \mtrx{M} & d \lambda_D y^{\pi-(D-1)} \vec{e}_1 & \vec{0} & d \lambda_D x^{\pi-(D-1)} \vec{e}_d \\ -\mtrx{M}^\top & \mtrx{0} & \vec{0} & -x^{\pi-(D-1)} \vec{e}_1 & y^{\pi-(D-1)} \vec{e}_d \\ -d \lambda_D y^{\pi-(D-1)} \vec{e}_1^\top & \vec{0}^\top & 0 & z G & 0 \\ \vec{0}^\top & x^{\pi-(D-1)} \vec{e}_1^\top & -z G & 0 & 0 \\ -d \lambda_D x^{\pi-(D-1)} \vec{e}_d^\top & -y^{\pi-(D-1)} \vec{e}_d^\top & 0 & 0 & 0
\end{bmatrix}
\\\\=&
(-1)^{(2d)+1}
\\&
\Pf
\begin{bmatrix}
\mtrx{0} & \mtrx{M}_{(-),(\hat{d})} & d \lambda_D y^{\pi-(D-1)} \vec{e}_1 & \vec{0} & d \lambda_D x^{\pi-(D-1)} \vec{e}_d \\ -\left(\mtrx{M}_{(-),(\hat{d})}\right)^\top & \mtrx{0}_{(\hat{d}),(\hat{d})} & \vec{0}_{\hat{d}} & -x^{\pi-(D-1)} (\vec{e}_1)_{\hat{d}} & \vec{0}_{\hat{d}} \\ -d \lambda_D y^{\pi-(D-1)} \vec{e}_1^\top & \vec{0}_{\hat{d}}^\top & 0 & z G & 0 \\ \vec{0}^\top & x^{\pi-(D-1)} (\vec{e}_1)_{\hat{d}}^\top & -z G & 0 & 0 \\ -d \lambda_D x^{\pi-(D-1)} \vec{e}_d^\top & \vec{0}_{\hat{d}}^\top & 0 & 0 & 0
\end{bmatrix}
\\\\=&
-
(-1)^{(d)+(2d+2)+H((2d+2)-(d))}
(d \lambda_D x^{\pi-(D-1)})
\\&
\Pf
\begin{bmatrix}
\mtrx{0}_{(\hat{d}),(\hat{d})} & \mtrx{M}_{(\hat{d}),(\hat{d})} & d \lambda_D y^{\pi-(D-1)} (\vec{e}_1)_{\hat{d}} & \vec{0}_{\hat{d}} \\ -\left(\mtrx{M}_{(\hat{d}),(\hat{d})}\right)^\top & \mtrx{0}_{(\hat{d}),(\hat{d})} & \vec{0}_{\hat{d}} & -x^{\pi-(D-1)} (\vec{e}_1)_{\hat{d}} \\ -d \lambda_D y^{\pi-(D-1)} (\vec{e}_1)_{\hat{d}}^\top & \vec{0}_{\hat{d}}^\top & 0 & z G \\ \vec{0}_{\hat{d}}^\top & x^{\pi-(D-1)} (\vec{e}_1)_{\hat{d}}^\top & -z G & 0
\end{bmatrix}
\/.
\end{align*}
To finish off both of these cases, we now calculate the degree of \[
\Pf
\begin{bmatrix}
\mtrx{0}_{(\hat{d}),(\hat{d})} & \mtrx{M}_{(\hat{d}),(\hat{d})} & d \lambda_D y^{\pi-(D-1)} (\vec{e}_1)_{\hat{d}} & \vec{0}_{\hat{d}} \\ -\left(\mtrx{M}_{(\hat{d}),(\hat{d})}\right)^\top & \mtrx{0}_{(\hat{d}),(\hat{d})} & \vec{0}_{\hat{d}} & -x^{\pi-(D-1)} (\vec{e}_1)_{\hat{d}} \\ -d \lambda_D y^{\pi-(D-1)} (\vec{e}_1)_{\hat{d}}^\top & \vec{0}_{\hat{d}}^\top & 0 & z G \\ \vec{0}_{\hat{d}}^\top & x^{\pi-(D-1)} (\vec{e}_1)_{\hat{d}}^\top & -z G & 0
\end{bmatrix}
\] by first applying more cofactor expansion, allowing us to write
\begin{align*}&
\Pf
\begin{bmatrix}
\mtrx{0}_{(\hat{d}),(\hat{d})} & \mtrx{M}_{(\hat{d}),(\hat{d})} & d \lambda_D y^{\pi-(D-1)} (\vec{e}_1)_{\hat{d}} & \vec{0}_{\hat{d}} \\ -\left(\mtrx{M}_{(\hat{d}),(\hat{d})}\right)^\top & \mtrx{0}_{(\hat{d}),(\hat{d})} & \vec{0}_{\hat{d}} & -x^{\pi-(D-1)} (\vec{e}_1)_{\hat{d}} \\ -d \lambda_D y^{\pi-(D-1)} (\vec{e}_1)_{\hat{d}}^\top & \vec{0}_{\hat{d}}^\top & 0 & z G \\ \vec{0}_{\hat{d}}^\top & x^{\pi-(D-1)} (\vec{e}_1)_{\hat{d}}^\top & -z G & 0
\end{bmatrix}
\\\\=&
(-1)^{(d)+(2d)+H((2d)-(d))}
(-x^{\pi-(D-1)})
\\&
\Pf
\begin{bmatrix}
\mtrx{0}_{(\hat{d}),(\hat{d})} & \mtrx{M}_{(\hat{d}),(\hat{1}\hat{d})} & d \lambda_D y^{\pi-(D-1)} (\vec{e}_1)_{\hat{d}} \\ -\left(\mtrx{M}_{(\hat{d}),(\hat{1}\hat{d})}\right)^\top & \mtrx{0}_{(\hat{1}\hat{d}),(\hat{1}\hat{d})} & \vec{0}_{\hat{1}\hat{d}} \\ -d \lambda_D y^{\pi-(D-1)} (\vec{e}_1)_{\hat{d}}^\top & \vec{0}_{\hat{1}\hat{d}}^\top & 0
\end{bmatrix}
\\&
+
(-1)^{(2d-1)+(2d)+H((2d)-(2d-1))}
(z G)
\\&
\Pf
\begin{bmatrix}
\mtrx{0}_{(\hat{d}),(\hat{d})} & \mtrx{M}_{(\hat{d}),(\hat{d})} \\ -\left(\mtrx{M}_{(\hat{d}),(\hat{d})}\right)^\top & \mtrx{0}_{(\hat{d}),(\hat{d})}
\end{bmatrix}
\\\\=&
x^{\pi-(D-1)}
(-1)^{(1)+(2d-2)+H((2d-2)-(1))}
(d \lambda_D y^{\pi-(D-1)})
\Pf
\begin{bmatrix}
\mtrx{0}_{(\hat{1}\hat{d}),(\hat{1}\hat{d})} & \mtrx{M}_{(\hat{1}\hat{d}),(\hat{1}\hat{d})} \\ -\left(\mtrx{M}_{(\hat{1}\hat{d}),(\hat{1}\hat{d})}\right)^\top & \mtrx{0}_{(\hat{1}\hat{d}),(\hat{1}\hat{d})}
\end{bmatrix}
\\&
+
(-1)^{(2d-1)+(2d)+H((2d)-(2d-1))}
(z G)
\Pf
\begin{bmatrix}
\mtrx{0}_{(\hat{d}),(\hat{d})} & \mtrx{M}_{(\hat{d}),(\hat{d})} \\ -\left(\mtrx{M}_{(\hat{d}),(\hat{d})}\right)^\top & \mtrx{0}_{(\hat{d}),(\hat{d})}
\end{bmatrix}
\\\\=&\tag{v}
x^{\pi-(D-1)}
(-1)^{(1)+(2d-2)+H((2d-2)-(1))}
(d \lambda_D y^{\pi-(D-1)})
\left(
(-1)^{(d-2)(d-3)/2}
\det\left(\mtrx{M}_{(\hat{1}\hat{d}),(\hat{1}\hat{d})}\right)
\right)
\\&
+
(-1)^{(2d-1)+(2d)+H((2d)-(2d-1))}
(z G)
\left(
(-1)^{D-1}
\det\left(\mtrx{M}_{(\hat{d}),(\hat{d})}\right)
\right)
\\\\=&\tag{vi}
x^{\pi-(D-1)}
(d \lambda_D y^{\pi-(D-1)})
(-1)^{(2D-2)(2D-3)/2}
\det\left(\mtrx{M}_{(\hat{1}\hat{d}),(\hat{1}\hat{d})}\right)
+
(z G)
(-1)^{D-1}
(-(d-1)! z g)
\\\\=&
(-1)^{D-1} d \lambda_D (xy)^{\pi-(D-1)}
\det\left(\mtrx{M}_{(\hat{1}\hat{d}),(\hat{1}\hat{d})}\right)
+
(-1)^D (d-1)! z^2 g G
\/.\end{align*}
Here line (v) is a result of Remark \ref{rmk:PfaffBlockSquare}, and line (vi) is true by Corollary \ref{cor:detMdd}.

Recall that $g = \sum_{t=0}^{D-1} \lambda_t (xy)^{(D-1)-t} F^t$. Note that $(xy)^{(D-1)-t} F^t$ has degree $2((D-1)-t)+2t = 2(D-1)$ for all $0 \leq t \leq D-1$, so $g \in P_{2(D-1)}$. We also have $\det\left(\mtrx{M}_{(\hat{1}\hat{d}),(\hat{1}\hat{d})}\right) \in P_{(d-2)1} = P_{2(D-1)}$ by Remark \ref{rmk:DetDeg}, since $\mtrx{M}$ is size $d \times d$ and pure graded degree $1$.
Thus we have both \[
(xy)^{\pi-(D-1)}
\det\left(\mtrx{M}_{(\hat{1}\hat{d}),(\hat{1}\hat{d})}\right)
 \in 
P_{2(\pi-(D-1))+2(D-1)} = P_{2\pi}
\] and \[
z^2 g G
 \in 
P_{2+2(D-1)+2(\pi-D)} = P_{2\pi}
\/,\] so we see that
\begin{align*}&
\Pf
\begin{bmatrix}
\mtrx{0}_{(\hat{d}),(\hat{d})} & \mtrx{M}_{(\hat{d}),(\hat{d})} & d \lambda_D y^{\pi-(D-1)} (\vec{e}_1)_{\hat{d}} & \vec{0}_{\hat{d}} \\ -\left(\mtrx{M}_{(\hat{d}),(\hat{d})}\right)^\top & \mtrx{0}_{(\hat{d}),(\hat{d})} & \vec{0}_{\hat{d}} & -x^{\pi-(D-1)} (\vec{e}_1)_{\hat{d}} \\ -d \lambda_D y^{\pi-(D-1)} (\vec{e}_1)_{\hat{d}}^\top & \vec{0}_{\hat{d}}^\top & 0 & z G \\ \vec{0}_{\hat{d}}^\top & x^{\pi-(D-1)} (\vec{e}_1)_{\hat{d}}^\top & -z G & 0
\end{bmatrix}
\\=&
(-1)^{D-1} d \lambda_D (xy)^{\pi-(D-1)}
\det\left(\mtrx{M}_{(\hat{1}\hat{d}),(\hat{1}\hat{d})}\right)
+
(-1)^D (d-1)! z^2 g G
 \in 
P_{2\pi}
\/.\end{align*}
From this we have both
\begin{align*}&
\Pf_{d} \mtrx{\partial}_2
\\=&
-
(-1)^{(2d-1)+(2d+2)+H((2d+2)-(2d-1))}
y^{\pi-(D-1)}
\\&
\Pf
\begin{bmatrix}
\mtrx{0}_{(\hat{d}),(\hat{d})} & \mtrx{M}_{(\hat{d}),(\hat{d})} & d \lambda_D y^{\pi-(D-1)} (\vec{e}_1)_{\hat{d}} & \vec{0}_{\hat{d}} \\ -\left(\mtrx{M}_{(\hat{d}),(\hat{d})}\right)^\top & \mtrx{0}_{(\hat{d}),(\hat{d})} & \vec{0}_{\hat{d}} & -x^{\pi-(D-1)} (\vec{e}_1)_{\hat{d}} \\ -d \lambda_D y^{\pi-(D-1)} (\vec{e}_1)_{\hat{d}}^\top & \vec{0}_{\hat{d}}^\top & 0 & z G \\ \vec{0}_{\hat{d}}^\top & x^{\pi-(D-1)} (\vec{e}_1)_{\hat{d}}^\top & -z G & 0
\end{bmatrix}
\\\\&\in
P_{(\pi-(D-1))+2\pi} = P_{s/2+1}
\end{align*}
and
\begin{align*}&
\Pf_{2d} \mtrx{\partial}_2
\\=&
-
(-1)^{(d)+(2d+2)+H((2d+2)-(d))}
(d \lambda_D x^{\pi-(D-1)})
\\&
\Pf
\begin{bmatrix}
\mtrx{0}_{(\hat{d}),(\hat{d})} & \mtrx{M}_{(\hat{d}),(\hat{d})} & d \lambda_D y^{\pi-(D-1)} (\vec{e}_1)_{\hat{d}} & \vec{0}_{\hat{d}} \\ -\left(\mtrx{M}_{(\hat{d}),(\hat{d})}\right)^\top & \mtrx{0}_{(\hat{d}),(\hat{d})} & \vec{0}_{\hat{d}} & -x^{\pi-(D-1)} (\vec{e}_1)_{\hat{d}} \\ -d \lambda_D y^{\pi-(D-1)} (\vec{e}_1)_{\hat{d}}^\top & \vec{0}_{\hat{d}}^\top & 0 & z G \\ \vec{0}_{\hat{d}}^\top & x^{\pi-(D-1)} (\vec{e}_1)_{\hat{d}}^\top & -z G & 0
\end{bmatrix}
\\\\&\in
P_{(\pi-(D-1))+2\pi} = P_{s/2+1}
\/,\end{align*}
which means $\Pf_{d} \mtrx{\partial}_2$ and $\Pf_{2d} \mtrx{\partial}_2$ have degree $s/2+1$.

Now that we have calculated the degrees of $\Pf_{1} \mtrx{\partial}_2$, $\Pf_{d} \mtrx{\partial}_2$, $\Pf_{d+1} \mtrx{\partial}_2$, and $\Pf_{2d} \mtrx{\partial}_2$, we only need to calculate the degree of $\Pf_{\ell} \mtrx{\partial}_2$ and $\Pf_{d+\ell} \mtrx{\partial}_2$ for $1<\ell<d$.

Let $1<\ell<d$. Using Definition \ref{def:PfaffEll} and the cofactor expansion formula in Definition \ref{def:Pfaffs}, we write
\begin{align*}&
\Pf_{\ell} \mtrx{\partial}_2
\\=&
\Pf_{\ell}
\begin{bmatrix}
\mtrx{0} & \mtrx{M} & d \lambda_D y^{\pi-(D-1)} \vec{e}_1 & \vec{0} & d \lambda_D x^{\pi-(D-1)} \vec{e}_d \\ -\mtrx{M}^\top & \mtrx{0} & \vec{0} & -x^{\pi-(D-1)} \vec{e}_1 & y^{\pi-(D-1)} \vec{e}_d \\ -d \lambda_D y^{\pi-(D-1)} \vec{e}_1^\top & \vec{0}^\top & 0 & z G & 0 \\ \vec{0}^\top & x^{\pi-(D-1)} \vec{e}_1^\top & -z G & 0 & 0 \\ -d \lambda_D x^{\pi-(D-1)} \vec{e}_d^\top & -y^{\pi-(D-1)} \vec{e}_d^\top & 0 & 0 & 0
\end{bmatrix}
\\\\=&
(-1)^{(\ell)+1}
\\&
\Pf
\begin{bmatrix}
\mtrx{0}_{(\hat{\ell}),(\hat{\ell})} & \mtrx{M}_{(\hat{\ell}),(-)} & d \lambda_D y^{\pi-(D-1)} (\vec{e}_1)_{\hat{\ell}} & \vec{0}_{\hat{\ell}} & d \lambda_D x^{\pi-(D-1)} (\vec{e}_d)_{\hat{\ell}} \\ -\left(\mtrx{M}_{(\hat{\ell}),(-)}\right)^\top & \mtrx{0} & \vec{0} & -x^{\pi-(D-1)} \vec{e}_1 & y^{\pi-(D-1)} \vec{e}_d \\ -d \lambda_D y^{\pi-(D-1)} (\vec{e}_1)_{\hat{\ell}}^\top & \vec{0}^\top & 0 & z G & 0 \\ \vec{0}_{\hat{\ell}}^\top & x^{\pi-(D-1)} \vec{e}_1^\top & -z G & 0 & 0 \\ -d \lambda_D x^{\pi-(D-1)} (\vec{e}_d)_{\hat{\ell}}^\top & -y^{\pi-(D-1)} \vec{e}_d^\top & 0 & 0 & 0
\end{bmatrix}
\\\\=&
-(-1)^\ell
(-1)^{(1)+(2d)+H((2d)-(1))}
(d \lambda_D y^{\pi-(D-1)})
\\&
\Pf
\begin{bmatrix}
\mtrx{0}_{(\hat{1}\hat{\ell}),(\hat{1}\hat{\ell})} & \mtrx{M}_{(\hat{1}\hat{\ell}),(-)} & \vec{0}_{\hat{1}\hat{\ell}} & d \lambda_D x^{\pi-(D-1)} (\vec{e}_d)_{\hat{1}\hat{\ell}} \\ -\left(\mtrx{M}_{(\hat{1}\hat{\ell}),(-)}\right)^\top & \mtrx{0} & -x^{\pi-(D-1)} \vec{e}_1 & y^{\pi-(D-1)} \vec{e}_d \\ \vec{0}_{\hat{1}\hat{\ell}}^\top & x^{\pi-(D-1)} \vec{e}_1^\top & 0 & 0 \\ -d \lambda_D x^{\pi-(D-1)} (\vec{e}_d)_{\hat{1}\hat{\ell}}^\top & -y^{\pi-(D-1)} \vec{e}_d^\top & 0 & 0
\end{bmatrix}
\\&
-
(-1)^{\ell}
(-1)^{(2d+1)+(2d)+H((2d)-(2d+1))}
(-z G)
\\&
\Pf
\begin{bmatrix}
\mtrx{0}_{(\hat{\ell}),(\hat{\ell})} & \mtrx{M}_{(\hat{\ell}),(-)} & d \lambda_D x^{\pi-(D-1)} (\vec{e}_d)_{\hat{\ell}} \\ -\left(\mtrx{M}_{(\hat{\ell}),(-)}\right)^\top & \mtrx{0} & y^{\pi-(D-1)} \vec{e}_d \\ -d \lambda_D x^{\pi-(D-1)} (\vec{e}_d)_{\hat{\ell}}^\top & -y^{\pi-(D-1)} \vec{e}_d^\top & 0
\end{bmatrix}
\\\\=&
-(-1)^\ell d \lambda_D
y^{\pi-(D-1)}
(-1)^{(d-1)+(2d-1)+H((2d-1)-(d-1))}
(-x^{\pi-(D-1)})
\\&
\Pf
\begin{bmatrix}
\mtrx{0}_{(\hat{1}\hat{\ell}),(\hat{1}\hat{\ell})} & \mtrx{M}_{(\hat{1}\hat{\ell}),(\hat{1})} & d \lambda_D x^{\pi-(D-1)} (\vec{e}_d)_{\hat{1}\hat{\ell}} \\ -\left(\mtrx{M}_{(\hat{1}\hat{\ell}),(\hat{1})}\right)^\top & \mtrx{0}_{(\hat{1}),(\hat{1})} & y^{\pi-(D-1)} (\vec{e}_d)_{\hat{1}} \\ -d \lambda_D x^{\pi-(D-1)} (\vec{e}_d)_{\hat{1}\hat{\ell}}^\top & -y^{\pi-(D-1)} (\vec{e}_d)_{\hat{1}}^\top & 0
\end{bmatrix}
\\&
-
(-1)^{\ell}
(-1)^{(2d+1)+(2d)+H((2d)-(2d+1))}
(-z G)
\\&
\Pf
\begin{bmatrix}
\mtrx{0}_{(\hat{\ell}),(\hat{\ell})} & \mtrx{M}_{(\hat{\ell}),(-)} & d \lambda_D x^{\pi-(D-1)} (\vec{e}_d)_{\hat{\ell}} \\ -\left(\mtrx{M}_{(\hat{\ell}),(-)}\right)^\top & \mtrx{0} & y^{\pi-(D-1)} \vec{e}_d \\ -d \lambda_D x^{\pi-(D-1)} (\vec{e}_d)_{\hat{\ell}}^\top & -y^{\pi-(D-1)} \vec{e}_d^\top & 0
\end{bmatrix}
\end{align*}
\begin{align*}
=&
-(-1)^\ell d \lambda_D (xy)^{\pi-(D-1)}
(-1)^{(d-2)+(2d-2)+H((2d-2)-(d-2))}
(d \lambda_D x^{\pi-(D-1)})
\\&
\Pf
\begin{bmatrix}
\mtrx{0}_{(\hat{1}\hat{\ell}\hat{d}),(\hat{1}\hat{\ell}\hat{d})} & \mtrx{M}_{(\hat{1}\hat{\ell}\hat{d}),(\hat{1})} \\ -\left(\mtrx{M}_{(\hat{1}\hat{\ell}\hat{d}),(\hat{1})}\right)^\top & \mtrx{0}_{(\hat{1}),(\hat{1})}
\end{bmatrix}
\\&
-
(-1)^\ell d \lambda_D (xy)^{\pi-(D-1)}
(-1)^{(2d-3)+(2d-2)+H((2d-2)-(2d-3))}
(y^{\pi-(D-1)})
\\&
\Pf
\begin{bmatrix}
\mtrx{0}_{(\hat{1}\hat{\ell}),(\hat{1}\hat{\ell})} & \mtrx{M}_{(\hat{1}\hat{\ell}),(\hat{1}\hat{d})} \\ -\left(\mtrx{M}_{(\hat{1}\hat{\ell}),(\hat{1}\hat{d})}\right)^\top & \mtrx{0}_{(\hat{1}\hat{d}),(\hat{1}\hat{d})}
\end{bmatrix}
\\&
-
(-1)^\ell z G
(-1)^{(d-1)+(2d)+H((2d)-(d-1))}
(d \lambda_D x^{\pi-(D-1)})
\Pf
\begin{bmatrix}
\mtrx{0}_{(\hat{\ell}\hat{d}),(\hat{\ell}\hat{d})} & \mtrx{M}_{(\hat{\ell}\hat{d}),(-)} \\ -\left(\mtrx{M}_{(\hat{\ell}\hat{d}),(-)}\right)^\top & \mtrx{0}
\end{bmatrix}
\\&
-
(-1)^\ell z G
(-1)^{(2d-1)+(2d)+H((2d)-(2d-1))}
(y^{\pi-(D-1)})
\Pf
\begin{bmatrix}
\mtrx{0}_{(\hat{\ell}),(\hat{\ell})} & \mtrx{M}_{(\hat{\ell}),(\hat{d})} \\ -\left(\mtrx{M}_{(\hat{\ell}),(\hat{d})}\right)^\top & \mtrx{0}_{(\hat{d}),(\hat{d})}
\end{bmatrix}
\\\\=&\tag{vii}
-(-1)^\ell d \lambda_D (xy^2)^{\pi-(D-1)}
\left(
(-1)^{D-1}
\det\left(\mtrx{M}_{(\hat{1}\hat{\ell}),(\hat{1}\hat{d})}\right)
\right)
\\&
-
(-1)^\ell y^{\pi-(D-1)} z G
\left(
(-1)^{D-1}
\det\left(\mtrx{M}_{(\hat{\ell}),(\hat{d})}\right)
\right)
\\\\=&
(-1)^{\ell+D}
y^{\pi-(D-1)}
\left(
d \lambda_D
(xy)^{\pi-(D-1)}
\det\left(\mtrx{M}_{(\hat{1}\hat{\ell}),(\hat{1}\hat{d})}\right)
+
z G
\det\left(\mtrx{M}_{(\hat{\ell}),(\hat{d})}\right)
\right)
\/,\end{align*}
where line (vii) is true by Lemma \ref{lem:PfaffBlockNonSquare} and Remark \ref{rmk:PfaffBlockSquare}.
By Remark \ref{rmk:DetDeg}, $\det\left(\mtrx{M}_{(\hat{1}\hat{\ell}),(\hat{1}\hat{d})}\right) \in P_{2D-2}$ and $\det\left(\mtrx{M}_{(\hat{\ell}),(\hat{d})}\right) \in P_{2D-1}$, which means \[
(xy)^{\pi-(D-1)}
\det\left(\mtrx{M}_{(\hat{1}\hat{\ell}),(\hat{1}\hat{d})}\right)
 \in 
P_{2(\pi-(D-1))+(2D-2)} = P_{2\pi}
\] and \[
z G
\det\left(\mtrx{M}_{(\hat{\ell}),(\hat{d})}\right)
 \in 
P_{1+2(\pi-D)+(2D-1)} = P_{2\pi}
\/.\] Therefore
\begin{align*}&
\Pf_{\ell} \mtrx{\partial}_2
\\=&
(-1)^{\ell+D}
y^{\pi-(D-1)}
\left(
d \lambda_D
(xy)^{\pi-(D-1)}
\det\left(\mtrx{M}_{(\hat{1}\hat{\ell}),(\hat{1}\hat{d})}\right)
+
z G
\det\left(\mtrx{M}_{(\hat{\ell}),(\hat{d})}\right)
\right)
\\&\in
P_{(\pi-(D-1))+2\pi} = P_{s/2+1}
\/,\end{align*}
so $\Pf_{\ell} \mtrx{\partial}_2$ has degree $s/2+1$.
% {\crd you could put the equations that are just about "if you remove a certain number of rows and columns, this is what degree you get" in a separate lemma that supports all of the main propositions when you split things up}

% ell=d+ell
Let $0<\ell<d$. Using Definition \ref{def:PfaffEll} and the cofactor expansion formula in Definition \ref{def:Pfaffs}, we write
\begin{align*}&
\Pf_{d+\ell} \mtrx{\partial}_2
\\=&
\Pf_{d+\ell}
\begin{bmatrix}
\mtrx{0} & \mtrx{M} & d \lambda_D y^{\pi-(D-1)} \vec{e}_1 & \vec{0} & d \lambda_D x^{\pi-(D-1)} \vec{e}_d \\ -\mtrx{M}^\top & \mtrx{0} & \vec{0} & -x^{\pi-(D-1)} \vec{e}_1 & y^{\pi-(D-1)} \vec{e}_d \\ -d \lambda_D y^{\pi-(D-1)} \vec{e}_1^\top & \vec{0}^\top & 0 & z G & 0 \\ \vec{0}^\top & x^{\pi-(D-1)} \vec{e}_1^\top & -z G & 0 & 0 \\ -d \lambda_D x^{\pi-(D-1)} \vec{e}_d^\top & -y^{\pi-(D-1)} \vec{e}_d^\top & 0 & 0 & 0
\end{bmatrix}
\\\\=&
(-1)^{(d+\ell)+1}
\\&
\Pf
\begin{bmatrix}
\mtrx{0} & \mtrx{M}_{(-),(\hat{\ell})} & d \lambda_D y^{\pi-(D-1)} \vec{e}_1 & \vec{0} & d \lambda_D x^{\pi-(D-1)} \vec{e}_d \\ -\left(\mtrx{M}_{(-),(\hat{\ell})}\right)^\top & \mtrx{0}_{(\hat{\ell}),(\hat{\ell})} & \vec{0}_{\hat{\ell}} & -x^{\pi-(D-1)} (\vec{e}_1)_{\hat{\ell}} & y^{\pi-(D-1)} (\vec{e}_d)_{\hat{\ell}} \\ -d \lambda_D y^{\pi-(D-1)} \vec{e}_1^\top & \vec{0}_{\hat{\ell}}^\top & 0 & z G & 0 \\ \vec{0}^\top & x^{\pi-(D-1)} (\vec{e}_1)_{\hat{\ell}}^\top & -z G & 0 & 0 \\ -d \lambda_D x^{\pi-(D-1)} \vec{e}_d^\top & -y^{\pi-(D-1)} (\vec{e}_d)_{\hat{\ell}}^\top & 0 & 0 & 0
\end{bmatrix}
\\\\=&
-(-1)^{\ell}
(-1)^{(1)+(2d)+H((2d)-(1))}
(d \lambda_D y^{\pi-(D-1)})
\\&
\Pf
\begin{bmatrix}
\mtrx{0}_{(\hat{1}),(\hat{1})} & \mtrx{M}_{(\hat{1}),(\hat{\ell})} & \vec{0}_{\hat{1}} & d \lambda_D x^{\pi-(D-1)} (\vec{e}_d)_{\hat{1}} \\ -\left(\mtrx{M}_{(\hat{1}),(\hat{\ell})}\right)^\top & \mtrx{0}_{(\hat{\ell}),(\hat{\ell})} & -x^{\pi-(D-1)} (\vec{e}_1)_{\hat{\ell}} & y^{\pi-(D-1)} (\vec{e}_d)_{\hat{\ell}} \\ \vec{0}_{\hat{1}}^\top & x^{\pi-(D-1)} (\vec{e}_1)_{\hat{\ell}}^\top & 0 & 0 \\ -d \lambda_D x^{\pi-(D-1)} (\vec{e}_d)_{\hat{1}}^\top & -y^{\pi-(D-1)} (\vec{e}_d)_{\hat{\ell}}^\top & 0 & 0
\end{bmatrix}
\\&
-
(-1)^{\ell}
(-1)^{(2d+1)+(2d)+H((2d)-(2d+1))}
(-z G)
\\&
\Pf
\begin{bmatrix}
\mtrx{0} & \mtrx{M}_{(-),(\hat{\ell})} & d \lambda_D x^{\pi-(D-1)} \vec{e}_d \\ -\left(\mtrx{M}_{(-),(\hat{\ell})}\right)^\top & \mtrx{0}_{(\hat{\ell}),(\hat{\ell})} & y^{\pi-(D-1)} (\vec{e}_d)_{\hat{\ell}} \\ -d \lambda_D x^{\pi-(D-1)} \vec{e}_d^\top & -y^{\pi-(D-1)} (\vec{e}_d)_{\hat{\ell}}^\top & 0
\end{bmatrix}
\\\\=&
-(-1)^{\ell} d \lambda_D
y^{\pi-(D-1)}
(-1)^{(d)+(2d-1)+H((2d-1)-(d))}
(-x^{\pi-(D-1)})
\\&
\Pf
\begin{bmatrix}
\mtrx{0}_{(\hat{1}),(\hat{1})} & \mtrx{M}_{(\hat{1}),(\hat{1}\hat{\ell})} & d \lambda_D x^{\pi-(D-1)} (\vec{e}_d)_{\hat{1}} \\ -\left(\mtrx{M}_{(\hat{1}),(\hat{1}\hat{\ell})}\right)^\top & \mtrx{0}_{(\hat{1}\hat{\ell}),(\hat{1}\hat{\ell})} & y^{\pi-(D-1)} (\vec{e}_d)_{\hat{1}\hat{\ell}} \\ -d \lambda_D x^{\pi-(D-1)} (\vec{e}_d)_{\hat{1}}^\top & -y^{\pi-(D-1)} (\vec{e}_d)_{\hat{1}\hat{\ell}}^\top & 0
\end{bmatrix}
\\&
-
(-1)^{\ell}
(-1)^{(2d+1)+(2d)+H((2d)-(2d+1))}
(-z G)
\\&
\Pf
\begin{bmatrix}
\mtrx{0} & \mtrx{M}_{(-),(\hat{\ell})} & d \lambda_D x^{\pi-(D-1)} \vec{e}_d \\ -\left(\mtrx{M}_{(-),(\hat{\ell})}\right)^\top & \mtrx{0}_{(\hat{\ell}),(\hat{\ell})} & y^{\pi-(D-1)} (\vec{e}_d)_{\hat{\ell}} \\ -d \lambda_D x^{\pi-(D-1)} \vec{e}_d^\top & -y^{\pi-(D-1)} (\vec{e}_d)_{\hat{\ell}}^\top & 0
\end{bmatrix}
\/,\end{align*}
\begin{align*}
=&
(-1)^{\ell} d \lambda_D
(xy)^{\pi-(D-1)}
(-1)^{(d-1)+(2d-2)+H((2d-2)-(d-1))}
(d \lambda_D x^{\pi-(D-1)})
\\&
\Pf
\begin{bmatrix}
\mtrx{0}_{(\hat{1}\hat{d}),(\hat{1}\hat{d})} & \mtrx{M}_{(\hat{1}\hat{d}),(\hat{1}\hat{\ell})} \\ -\left(\mtrx{M}_{(\hat{1}\hat{d}),(\hat{1}\hat{\ell})}\right)^\top & \mtrx{0}_{(\hat{1}\hat{\ell}),(\hat{1}\hat{\ell})}
\end{bmatrix}
\\&
+
(-1)^{\ell} d \lambda_D
(xy)^{\pi-(D-1)}
(-1)^{(2d-3)+(2d-2)+H((2d-2)-(2d-3))}
(y^{\pi-(D-1)})
\\&
\Pf
\begin{bmatrix}
\mtrx{0}_{(\hat{1}),(\hat{1})} & \mtrx{M}_{(\hat{1}),(\hat{1}\hat{\ell}\hat{d})} \\ -\left(\mtrx{M}_{(\hat{1}),(\hat{1}\hat{\ell}\hat{d})}\right)^\top & \mtrx{0}_{(\hat{1}\hat{\ell}\hat{d}),(\hat{1}\hat{\ell}\hat{d})}
\end{bmatrix}
\\&
-
(-1)^{\ell}
z G
(-1)^{(d)+(2d)+H((2d)-(d))}
(d \lambda_D x^{\pi-(D-1)})
\Pf
\begin{bmatrix}
\mtrx{0}_{(\hat{d}),(\hat{d})} & \mtrx{M}_{(\hat{d}),(\hat{\ell})} \\ -\left(\mtrx{M}_{(\hat{d}),(\hat{\ell})}\right)^\top & \mtrx{0}_{(\hat{\ell}),(\hat{\ell})}
\end{bmatrix}
\\&
-
(-1)^{\ell}
z G
(-1)^{(2d-1)+(2d)+H((2d)-(2d-1))}
(y^{\pi-(D-1)})
\Pf
\begin{bmatrix}
\mtrx{0} & \mtrx{M}_{(-),(\hat{\ell}\hat{d})} \\ -\left(\mtrx{M}_{(-),(\hat{\ell}\hat{d})}\right)^\top & \mtrx{0}_{(\hat{\ell}\hat{d}),(\hat{\ell}\hat{d})}
\end{bmatrix}
\\\\=&\tag{viii}
(-1)^{\ell} (d \lambda_D)^2
(x^2y)^{\pi-(D-1)}
\left(
(-1)^{D-1}
\det\left(\mtrx{M}_{(\hat{1}\hat{d}),(\hat{1}\hat{\ell})}\right)
\right)
\\&
+
(-1)^{\ell} d \lambda_D
x^{\pi-(D-1)} z G
\left(
(-1)^{D-1}
\det\left(\mtrx{M}_{(\hat{d}),(\hat{\ell})}\right)
\right)
\\\\=&
(-1)^{\ell+D-1} d \lambda_D x^{\pi-(D-1)}
\left(
d \lambda_D
(xy)^{\pi-(D-1)}
\det\left(\mtrx{M}_{(\hat{1}\hat{d}),(\hat{1}\hat{\ell})}\right)
+
z G
\det\left(\mtrx{M}_{(\hat{d}),(\hat{\ell})}\right)
\right)
\/,\end{align*}
where line (viii) is true by Lemma \ref{lem:PfaffBlockNonSquare} and Remark \ref{rmk:PfaffBlockSquare}.
By Remark \ref{rmk:DetDeg}, $\det\left(\mtrx{M}_{(\hat{1}\hat{d}),(\hat{1}\hat{\ell})}\right) \in P_{2D-2}$ and $\det\left(\mtrx{M}_{(\hat{d}),(\hat{\ell})}\right) \in P_{2D-1}$, which means \[
(xy)^{\pi-(D-1)}
\det\left(\mtrx{M}_{(\hat{1}\hat{d}),(\hat{1}\hat{\ell})}\right)
 \in 
P_{2(\pi-(D-1))+(2D-2)} = P_{2\pi}
\] and \[
z G
\det\left(\mtrx{M}_{(\hat{d}),(\hat{\ell})}\right)
 \in 
P_{1+2(\pi-D)+(2D-1)} = P_{2\pi}
\/.\] Therefore
\begin{align*}&
\Pf_{d+\ell} \mtrx{\partial}_2
\\=&
(-1)^{\ell+D-1} d \lambda_D x^{\pi-(D-1)}
\left(
d \lambda_D
(xy)^{\pi-(D-1)}
\det\left(\mtrx{M}_{(\hat{1}\hat{d}),(\hat{1}\hat{\ell})}\right)
+
z G
\det\left(\mtrx{M}_{(\hat{d}),(\hat{\ell})}\right)
\right)
\\&\in
P_{(\pi-(D-1))+2\pi} = P_{s/2+1}
\/,\end{align*}
so $\Pf_{d+\ell} \mtrx{\partial}_2$ has degree $s/2+1$.
\end{proof}
\begin{proposition}
\label{prop:Partial2PfaffsLast3}
The last three maximal order Pfaffians of $\mtrx{\partial}_2$ are as follows:
\begin{itemize}
\item $\Pf_{2d+1}(\mtrx{\partial}_2) = u x^q$,
\item $\Pf_{2d+2}(\mtrx{\partial}_2) = u y^q$, and
\item $\Pf_{2d+3}(\mtrx{\partial}_2) = u z^q$.
\end{itemize}
\end{proposition}
\begin{proof}
\setcounter{equation}{0}
We have
\begin{align}
\Pf_{2d+\ell}(\mtrx{\partial}_2)
=&
\Pf_\ell\left(
\mtrx{\psi}^\top \mtrx{\varphi}^\vee \mtrx{\psi}
+
\Pf(\mtrx{\varphi})
\mtrx{\Phi}
\right)
\\=&
\Pf_\ell\left(
\mtrx{\psi}^\top \mtrx{\varphi}^\vee \mtrx{\psi}
+
(u f)
\mtrx{\Phi}
\right)
\\=&
\Pf_\ell(u \mtrx{X})
\end{align}
for $\ell=1,2,3$, where line (1) follows from Lemma \ref{lem:partial2lastPfs}, $\Pf \mtrx{\varphi} = uf$ in line (2) by Lemma \ref{lem:Pfphi}, and $\mtrx{\psi}^\top \mtrx{\varphi}^\vee \mtrx{\psi} + u f \mtrx{\Phi} = u \mtrx{X}$ for line (3) by Lemma \ref{lem:PhiEquation}. Therefore,
\begin{itemize}
\item
$
\Pf_{2d+1}(\mtrx{\partial}_2)
 = 
\Pf_1\left(u \begin{bmatrix} 0 & z^q & -y^q \\ -z^q & 0 & x^q \\ y^q & -x^q & 0 \end{bmatrix}\right)
 = 
(-1)^{(1)+1} \Pf\left( \begin{bmatrix} 0 & u x^q \\ -u x^q & 0 \end{bmatrix} \right) = u x^q
$,
\item
$
\Pf_{2d+2}(\mtrx{\partial}_2)
 = 
\Pf_2\left(u \begin{bmatrix} 0 & z^q & -y^q \\ -z^q & 0 & x^q \\ y^q & -x^q & 0 \end{bmatrix}\right)
 = 
(-1)^{(2)+1} \Pf\left( u \begin{bmatrix} 0 & -y^q \\ y^q & 0 \end{bmatrix} \right)
 = 
u y^q
$, and
\item
$
\Pf_{2d+3}(\mtrx{\partial}_2)
 = 
\Pf_3\left(u \begin{bmatrix} 0 & z^q & -y^q \\ -z^q & 0 & x^q \\ y^q & -x^q & 0 \end{bmatrix}\right)
 = 
(-1)^{(3)+1} \Pf\left( u \begin{bmatrix} 0 & z^q \\ -z^q & 0 \end{bmatrix} \right)
 = 
u z^q
$.\qedhere
\end{itemize}
\end{proof}

\begin{theorem}
\label{thm:FDlqc}
If $\Char k = p$, where $p > 2D-1$ is an odd prime, then $f = \left( xy-z^2 \right)^D$ is link-$q$-compressed for all powers $q>1$ of $p$.
\end{theorem}
\begin{proof}
Fix a power $q>1$ of $p$.
Theorem \ref{thm:GensAndRes} showed that $\maxI^{[q]}:f$ is generated by the maximal order Pfaffians of $\mtrx{\partial}_2$, and so by Propositions \ref{prop:Partial2PfaffsExceptLast3} and \ref{prop:Partial2PfaffsLast3}, $(\maxI^{[q]}:f)/\maxI^{[q]}$ is generated by $x^q,y^q,z^q$ (which all become zero) and polynomials of degree $\frac{s}{2}+1$. By Lemma \ref{lem:lqcDef}, this means that $f$ is link-$q$-compressed.
\end{proof}

Since the set of link-q-compressed polynomials is Zariski open, most polynomials are link-$q$-compressed. We use the terminology from Remark \ref{rmk:lqcInGeneral} to write the following:
\begin{theorem}
\label{thm:generallqc}
Let $P = k[x,y,z]$ be a standard graded polynomial ring over $k$, a field of odd prime characteristic $p$. Let $d<p+1$ be an even number.
\begin{itemize}
\item A general choice of degree $d$ homogeneous polynomial $f \in P$ is link-$q$-compressed for a fixed power (or finitely many powers) $q>1$ of $p$.
\item A very general choice of degree $d$ homogeneous polynomial $f \in P$ is link-$q$-compressed for all fixed powers $q>1$ of $p$.
\end{itemize}
\end{theorem}
\begin{proof}
For a power $q>1$ of $p$, consider the set $S_q$ of all homogeneous degree $d$ link-$q$-compressed polynomials in $P$. In Remark \ref{rmk:lqcInGeneral}, we established that $S_q$ is Zariski open for any $q$, and in Theorem \ref{thm:FDlqc} we showed that $(xy-z^2)^{d/2} \in S_q$ no matter what $q$ is, so $S_q\neq\emptyset$. This means that for a fixed $q$, a general choice of $f\in P$ that is homogeneous and of degree $d$ is link-$q$-compressed.
In fact for any fixed finite set $Q$ of values of $q$, $(xy-z^2)^{d/2} \in \bigcap_{q\in Q} S_q$, and so $\bigcap_{q>1} S_q\neq\emptyset$. Thus a general choice of $f\in P$ that is homogeneous and of degree $d$ is link-$q$-compressed for all $q \in Q$.
Lastly, we have $(xy-z^2)^{d/2} \in \bigcap_{q>1} S_q$, and so $\bigcap_{q>1} S_q\neq\emptyset$. Therefore a very general choice of $f\in P$ that is homogeneous and of degree $d$ is link-$q$-compressed for all $q>1$.
\end{proof}

We use Theorem \ref{thm:generallqc} and the results for link-$q$-compressed polynomials in Section \ref{subsec:lqcFacts} from \cite{RGpaper} to conclude the following:
\begin{theorem}
\label{thm:generalBettiAndMore}
Let $P=k[x,y,z]$ be a standard graded polynomial ring over a field $k$ of odd prime characteristic $p$, with $\maxI=(x,y,z)$ the homogeneous maximal ideal of $P$. Let $d<p+1$ be an even number.

Fix a power $q\geq d+3$ of $p$. For a general choice of homogeneous $f \in P$ with $\deg f = d$, the following hold:
\begin{itemize}
    \item The minimal graded $R=P/(f)$-free resolution of $R/\maxI^{[q]}$ has the following eventually 2-periodic form \[\cdots \xrightarrow{\mtrx{\varphi}} R^{2d}(-b-d) \xrightarrow{\mtrx{\varphi}^\vee} R^{2d}(-b-1) \xrightarrow{\mtrx{\varphi}} R^{2d}(-b) \xrightarrow{\mtrx{\psi}^\top \mtrx{\varphi}^\vee} R^3(-q) \xrightarrow{\vec{c}^\top} R \to 0\] where $\mtrx{\varphi}$ is a $2d\times 2d$ linear skew-symmetric matrix with Pfaffian equal to $uf$ for some unit $u \in k$, $\mtrx{\psi}$ is a $2d\times 3$ matrix with entries of degree $\frac{1}{2}(q-d+1)$, $\vec{c}^\top = \begin{bmatrix}x^q&y^q&z^q\end{bmatrix}$, and $b = \frac{1}{2}(3q+d-1)$.
    \item The minimal graded resolution over $P = k[x,y,z]$ of $P/(\maxI^{[q]}+(f)) = R/\maxI^{[q]}$ has the form \[0 \to P^{2d}(-b-1) \xrightarrow{\begin{bmatrix}\mtrx{\varphi} \\ -\mtrx{\psi}^\top\end{bmatrix}} \bigoplus_{P^3(-q-d)}^{P^{2d}(-b)} \xrightarrow{\begin{bmatrix}\mtrx{\psi}^\top \mtrx{\varphi}^\vee & u f \mtrx{I} \\ -\vec{w}^\top & -\vec{c}^\top\end{bmatrix}} \bigoplus_{P(-d)}^{P^3(-q)} \xrightarrow{\begin{bmatrix}\vec{c}^\top & u f\end{bmatrix}} P \to 0\] where $\vec{w}^\top = \begin{bmatrix}w_1&w_2&\cdots&w_{2d}\end{bmatrix}$ consists of $2d$ many degree $\frac{s}{2}+1$ elements of $P$, which together with $x^q,y^q,z^q$ generate $\maxI^{[q]}:f$.
    \item The Castelnuovo-Mumford regularity is given by $\reg(R/\maxI^{[q]}) = \frac{1}{2}(3q+d-5)$.
    \item The Hilbert-Kunz function of $R$ at $q$ is $HK_R(q) = \frac{3}{4} dq^2 - \frac{1}{12} (d^3 - d)$.
    \item The socle module $\soc\left(R/\maxI^{[q]}\right)$ has generators that lie only in degree $s_2=\frac{1}{2}(3q+d-5)$ and has dimension $\dim_k\soc\left(R/\maxI^{[q]}\right)_{s_2}=2d$.
\end{itemize}
\end{theorem}
\begin{proof}
We know that a general choice of $f$ is link-$q$-compressed by Theorem \ref{thm:generallqc}. For link-$q$-compressed polynomials, the results (from \cite{RGpaper}) of Propositions \ref{prop:RGROverC} and \ref{prop:RGPOverCandf}, Theorems \ref{thm:RGlqcBetti} and \ref{thm:lqcHilbKunzFunc}, Corollary \ref{cor:lqcDegOfGens}, Proposition \ref{prop:lqcNumOfGens}, and Theorem \ref{thm:lqcSoc}, all hold. Some of these require that $3(q-1)-d$ be even, which holds because $d$ is even and $q$ (is a power of an odd number and thus) is odd.
\end{proof}
A general choice of $f$ is in fact link-$q$-compressed for multiple $q$ values, which allows us to note the following:
\begin{corollary}
\label{cor:generalCompareBetti}
Let $P=k[x,y,z]$ be a standard graded polynomial ring over a field $k$ of odd prime characteristic $p$, with $\maxI=(x,y,z)$ the homogeneous maximal ideal of $P$. Also let $d<p+1$ be an even number.

Fix two powers $q_1>q_0\geq d+3$ of $p$. For a general choice of homogeneous $f \in P$ with $\deg f = d$, the graded Betti numbers in high homological degree $2$ and higher of the $R=P/(f)$-modules $R/\maxI^{[q_0]}$ and $R/\maxI^{[q_1]}$ are the same up to a constant shift of $\frac{3}{2}(q_1-q_0)$.
\end{corollary}

We list here some examples of polynomials, one that is never link-$q$-compressed (meaning link-$q$-compressed for no $q$ values) and one that is sometimes link-$q$-compressed (meaning link-$q$-compressed for some $q$ values and not others). While Theorem \ref{thm:generallqc} shows that being always link-$q$-compressed (meaning link-$q$-compressed for all $q$ values) holds for very general choices of polynomials, these examples show that isn't true of all polynomials.
\begin{example}
\label{rmk:neverlqc}
Let $d\geq2$.
The polynomial $x_1^{q-d}$ generates $(\maxI^{[q]}:x_1^d)/\maxI^{[q]}$ and has degree $q-d$. Thus by Lemma \ref{lem:lqcDef}, $x_1^d \in k[x_1,\ldots,x_n]$ is link-$q$-compressed if and only if $q-d > \frac{n(q-1)-d}{2}$, which is true if and only if $d<n-(n-2)q$. So if $n\geq2$, $x_1^d$ is never link-$q$-compressed because $n-(n-2)q \leq n-(n-2) = 2$, but we assume that $d=\deg f\geq2$.
\end{example}
\begin{example}
\label{rmk:sometimeslqc}
In $\ZZ/(3)[x,y,z]$, $f=x^4+x^3y+x^3z+y^2z^2$ is link-$q$-compressed for $q=9$, but not $q=27$. The ideal $(\maxI^{[27]}:f)/\maxI^{[27]}$ has four generators, two of degree $38$ and two of degree $37=\frac{3(27-1)-4}{2} = \frac{s}{2}$, and thus $f$ is not link-$27$-compressed. This can be shown using Macaulay2. This example indicates that link-$p^e$-compressed polynomials aren't necessarily link-$p^{e+1}$-compressed.
\end{example}

After determining whether a polynomial $f$ is link-$q$-compressed or not, we know that all polynomials of the form $u\overline{T}(f)$ share that property, where $u$ is a unit and $\overline{T}$ is a linear isomorphism:
\begin{theorem}
\label{thm:lqcLinearIso}
Let $P = k[x_1,\ldots,x_n]$ with $k$ a field of characteristic $p>0$.
Fix a power $q$ of $p>0$.

The link-$q$-compressed property is unaffected by invertible scaling and linear isomorphism. In other words, a homogeneous polynomial $f\in P$ is link-$q$-compressed if and only if $u \overline{T}(f)$ is link-$q$-compressed, where $u\in k$ is nonzero and $\overline{T}$ is a linear isomorphism on $k[x_1,\ldots,x_n]$, meaning it is generated by mappings $
x_i
 \mapsto 
\sum_{j=1}^n \mtrx{T}_{i,j} x_j
$ for $1\leq i\leq n$ where $\mtrx{T}$ is an invertible $n \times n$ matrix with entries in $k$.
\end{theorem}
\begin{proof}
We can investigate invertible scaling
and linear isomorphisms separately.
Let $f$ be a homogeneous polynomial.

Let $u\in k$ be nonzero.
Since $(f)=(uf)$, $\maxI^{[q]}:(f)=\maxI^{[q]}:(uf)$, and thus they have the same generators. By Lemma \ref{lem:lqcDef}, $f$ will be link-$q$-compressed if and only if $uf$ is.

Let $\overline{T}$ be a linear isomorphism on $P$, corresponding to mappings $
x_i
 \mapsto 
\sum_{j=1}^n \mtrx{T}_{i,j} x_j
$ for $1\leq i\leq n$, where $\mtrx{T}$ is an invertible $n \times n$ matrix with entries in $k$. These mappings correspond to the matrix equation $
\begin{bmatrix}
\overline{T}(x_1) \\ \vdots \\ \overline{T}(x_n)
\end{bmatrix}
 = 
\mtrx{T}
\begin{bmatrix}
x_1 \\ \vdots \\ x_n
\end{bmatrix}
$.
Note that applying $\overline{T}$ to any element of $P$ preserves degree.

For any $1\leq i\leq n$, we have \[
\overline{T}(x_i^q)
 = 
\left(\overline{T}(x_i)\right)^q
 = 
\left(
\sum_{j=1}^n \mtrx{T}_{i,j} x_j
\right)^q
 = 
\sum_{j=1}^n 
\left(
\mtrx{T}_{i,j} x_j
\right)^q
 = 
\sum_{j=1}^n 
\mtrx{T}_{i,j}^q x_j^q
\/.\] If we define the $n\times n$ matrix $\mtrx{T}^{[q]}$ entry-wise as  $\left(\mtrx{T}^{[q]}\right)_{i,j}=\mtrx{T}_{i,j}^q$ for all $i,j$, then this means $
\begin{bmatrix}
\overline{T}(x_1^q) \\ \vdots \\ \overline{T}(x_n^q)
\end{bmatrix}
 = 
\mtrx{T}^{[q]}
\begin{bmatrix}
x_1^q \\ \vdots \\ x_n^q
\end{bmatrix}
$. Note that $\det\left(\mtrx{T}^{[q]}\right) = \left(\det\mtrx{T}\right)^q$:
\begin{align*}
\left(\det\mtrx{T}\right)^q
=&
\left(
\sum_\sigma 
\left(
\sgn\sigma
\prod_{i=1}^n 
\mtrx{T}_{i,\sigma(i)}
\right)
\right)^q
\\=&
\sum_\sigma 
\left(
\sgn\sigma
\prod_{i=1}^n 
\mtrx{T}_{i,\sigma(i)}
\right)^q
\\=&
\sum_\sigma 
\left(
(\sgn\sigma)^q
\left(
\prod_{i=1}^n 
\mtrx{T}_{i,\sigma(i)}
\right)^q
\right)
\\=&\tag{*}
\sum_\sigma 
\left(
\sgn\sigma
\prod_{i=1}^n 
\mtrx{T}_{i,\sigma(i)}^q
\right)
\\=&
\sum_\sigma 
\left(
\sgn\sigma
\prod_{i=1}^n 
\left(\mtrx{T}^{[q]}\right)_{i,\sigma(i)}
\right)
\\=&
\det\left(\mtrx{T}^{[q]}\right)
\/.\end{align*}
Here the sums are over all permutations $\sigma$ on the set $\{1,\ldots,n\}$. For (*) we have $(\sgn\sigma)^q=\sgn\sigma$ by Fermat's little theorem.
Since $\det\mtrx{T} \neq 0$ because $\mtrx{T}$ is invertible, we have
$\det\left(\mtrx{T}^{[q]}\right) = \left(\det\mtrx{T}\right)^q \neq 0$, thus $\mtrx{T}^{[q]}$ is invertible.
Because $\mtrx{T}^{[q]}$ is invertible, $\overline{T}(x_1^q),\ldots,\overline{T}(x_n^q)$ and $x_1^q,\ldots,x_n^q$ are linear combinations of each other, and therefore we have $\overline{T}(\maxI^{[q]}) = (\overline{T}(x_1^q),\ldots,\overline{T}(x_n^q)) = (x_1^q,\ldots,x_n^q) = \maxI^{[q]}$.

Set $g=\overline{T}(f)$.
Let $a\in \maxI^{[q]}:(f)$. Then $a f \in \maxI^{[q]}$, and thus \[
\overline{T}(a) g = \overline{T}(a) \overline{T}(f) = \overline{T}(a f) \in \overline{T}(\maxI^{[q]}) = \maxI^{[q]}
\/.\] This means $\overline{T}(a) \in \maxI^{[q]}:(g)$. Thus $\overline{T}(\maxI^{[q]}:(f)) \subseteq \maxI^{[q]}:(g)$.
Let $b\in \maxI^{[q]}:(g)$. Then \[
\overline{T}\left(f\ \overline{T}^{-1}(b)\right) = \overline{T}(f)\ \overline{T}\left(\overline{T}^{-1}(b)\right) = g\ b \in \maxI^{[q]} = \overline{T}(\maxI^{[q]})
\/,\] and so $f\ \overline{T}^{-1}(b) \in \maxI^{[q]}$ because $\overline{T}$ is invertible. Thus $\overline{T}^{-1}(b) \in \maxI^{[q]}:(f)$, which means that $b \in \overline{T}(\maxI^{[q]}:(f))$. Therefore $\maxI^{[q]}:(g) \subseteq \overline{T}(\maxI^{[q]}:(f))$ as well, so $\overline{T}(\maxI^{[q]}:(f)) = \maxI^{[q]}:(g)$.

Since $\overline{T}$ preserves degree, the degrees of elements of $(\maxI^{[q]}:(f))/\maxI^{[q]}$ are the same as those of $\overline{T}((\maxI^{[q]}:(f))/\maxI^{[q]}) = (\maxI^{[q]}:(g))/\maxI^{[q]}$. Therefore, by Lemma \ref{lem:lqcDef}, $f$ is link-$q$-compressed if and only if $g$ is.
\end{proof}

\begin{corollary}
If $k$ is a field of odd characteristic $p$, the polynomial $x^2-y^2-z^2$ in $P=k[x,y,z]$ is link-$q$-compressed for all powers $q>1$ of $p$.
If $k$ includes an element $i$ such that $i^2=-1$, the polynomial $x^2+y^2+z^2$ in $P=k[x,y,z]$ is link-$q$-compressed for all powers $q>1$ of $p$.
\end{corollary}
\begin{proof}
Let $q>1$ be a power of $p$.
If $\overline{T}_1: P \to P$ is the linear isomorphism induced by $x\mapsto x+y$, $y\mapsto x-y$, and $z\mapsto z$, then $\overline{T}_1(xy-z^2)=x^2-y^2-z^2$. By Theorem \ref{thm:lqcLinearIso}, $x^2-y^2-z^2$ is link-$q$-compressed because $xy-z^2$ is link-$q$-compressed.
If $\overline{T}_2: P \to P$ is the linear isomorphism induced by $x\mapsto x$, $y\mapsto iy$, and $z\mapsto iz$, then $\overline{T}_2(x^2-y^2-z^2)=x^2+y^2+z^2$. By Theorem \ref{thm:lqcLinearIso}, $x^2+y^2+z^2$ is link-$q$-compressed because $x^2-y^2-z^2$ is link-$q$-compressed.
\end{proof}
This Corollary aligns with results from \cite{KRV12}, which stated that if $R=k[x,y,z]/(x^2+y^2+z^2)$, then either $\projdim_R(R/\maxI^{[q]}) = \infty$ for all $q$ or $\projdim_R(R/\maxI^{[q]}) < \infty$ for all $q$, depending on $k$.

\bibliographystyle{amsalpha}
\bibliography{currentbib}

\end{document}